\documentclass[11pt]{amsart}
\usepackage[colorlinks,citecolor=blue,urlcolor=blue,bookmarks=false,hypertexnames=true]{hyperref}

\usepackage[a4paper,margin=23mm,top=30mm]{geometry}

\usepackage{amsfonts, amsmath, amssymb}
\usepackage[utf8]{inputenc} 

\usepackage{color}

\usepackage[all]{xy}

\usepackage{geometry}
\usepackage{hyperref}
\usepackage{mathtools,slashed}
\usepackage{setspace}
\usepackage{enumerate}

\def\d{\delta}

\def\p{\partial}

\def\Ad{\mathop{\rm Ad}\nolimits}
\def\ad{\mathop{\rm ad}\nolimits}

\usepackage{enumerate}

\newcommand{\G}[1]{\mathfrak{#1}}

\newcommand{\B}[1]{\mathbb{#1}}

\numberwithin{equation}{section}

\newtheorem{theorem}{Theorem}[section]
\newtheorem{proposition}[theorem]{Proposition}

\theoremstyle{definition}

\newtheorem{remark}[theorem]{Remark}

\author{Oğul Esen}
\address{Department of Mathematics, Gebze Technical University,  41400 Gebze-Kocaeli, Turkey}
\email{oesen@gtu.edu.tr}

\author{Hasan Gümral}
\address{Department of Mathematics, Yeditepe University,  34755 Ataşehir-İstanbul, Turkey}
\email{hgumral@yeditepe.edu.tr}

\author{Serkan Sütlü}
\address{Department of Mathematics, Işık University, 34980 Şile-İstanbul, Turkey}
\email{serkan.sutlu@isikun.edu.tr}

\begin{document}

\title{Tulczyjew's triplet for Lie groups III: \\ Higher order dynamics and  reductions for iterated bundles}

\begin{abstract}
Given a Lie group $G$, we elaborate the dynamics on $T^*T^*G$ and $T^*TG$, which is given by a Hamiltonian, as well as the dynamics on the Tulczyjew symplectic space $TT^*G$, which may be defined by a Lagrangian or a Hamiltonian function. As the trivializations we adapted respect the group structures of the iterated bundles, we exploit all possible subgroup reductions (Poisson, symplectic or both) of higher order dynamics. 

\bigskip

\noindent \textbf{MSC2010:} 70H50; 70G65; 53D20; 53D17. 

\bigskip

\noindent \textbf{Key Words:} Euler-Poincaré equations; Lie-Poisson equations,
higher order dynamics on Lie groups.
\end{abstract}
\maketitle

\tableofcontents

\setlength{\parindent}{0cm}
\setlength{\parskip}{5mm}
\onehalfspace
\section{Introduction}

The tangent and the cotangent bundles of a Lie group admit global trivializations, as well as the Lie group structures, induced from the underlying Lie group itself. These structures may further be carried over the iterated bundles $T^*TG$, $TT^*G$, and $T^*T^*G$. These iterated bundles constitute the Tulczyjew's triplet, introduced for a geometric description of the Legendre transformation from the Lagrangian description on $TG$ to the Hamiltonian description on $T^*G$ for a mechanical system having $G$ as the configuration space. 
Such a system admits $G$ as kinematical symmetries, and the reduction of the Lagrangian dynamics results in the Euler-Poincr\'{e} equations on the Lie algebra $\mathfrak{g}$ of $G$. Similarly, the reduction of the Hamiltonian dynamics to $\mathfrak{g}^*$ is described by the Lie-Poisson equations. 

The present note is intended as a sequel to \cite{esen2014tulczyjew,esen2015tulczyjew}. In the first part \cite{esen2014tulczyjew}, we gave a detailed description of the possible trivializations of the iterated bundles $T^*TG$, $TT^*G$, and $T^*T^*G$, which are Lie group isomorphisms. Moreover, we described the group structures up to the second iterated bundles, as well as the canonical involutions on them. Having explicit descriptions of the cotangent and the Tulczyjew symplectic structures, we performed the Marsden-Weinstein reduction by kinematical symmetries to obtain the reduced Tulczyjew triplet for the Legendre transformation from Euler-Poincar\'{e} to Lie-Poisson equations. Then, in the second part \cite{esen2015tulczyjew}, we studied the Lagrangian and the Hamiltonian dynamical equations at each stage of the Tulczyjew construction under the trivializations respecting the Lie group structures. The dynamics we considered is defined either by a Lagrangian on $TG$, or by a Hamiltonian on $T^*G$, which, in the framework of Tulczyjew construction, corresponds to Lagrangian submanifolds of  $T^*TG$ or $T^*T^*G$, respectively. In other words, first order dynamics considered in \cite{esen2015tulczyjew} restricts to the fiber coordinates of the second iterated bundles. 

In this work, we aim to give a complete description of the higher order dynamics, and their reductions by considering the Lagrangian and/or Hamiltonian functions on the second iterated bundles, taking full advantage of the trivializations at our disposal. Obviously, releasing the condition that the dynamics on iterated bundles are described by Lagrangian submanifolds opens up the possibility to obtain higher order forms of Euler-Poincar\'{e} and Lie-Poisson equations. The underlying structure will, indeed, offer more than this generalization. 

Immediate generalizations of the results of \cite{esen2014tulczyjew}, \cite{esen2015tulczyjew}, and the present work apply to fibered spaces admitting local trivializations, or Ehresmann connections. A recent work \cite{EsKuSu20} elaborates the parallel results in the particular case of the principal $G$-bundles, and their associated vector bundles.

\subsection{Trivializations}~

One observes that, the form of equations governing dynamics on Lie groups
depends on the kind of trivializations adapted on iterated bundles \cite%
{colombo2014higher, colombo2013optimal, gay2012invariant, marsden1991symplectic}. Additional terms in these
equations may or may not appear depending on whether trivialization
preserves semidirect product and group structures or not. If one preserves
the group structures, canonical embeddings of factors involving
trivialization defines subgroups of iterated bundles and reductions of
dynamics with these subgroups become possible.

Based on exhaustive investigation of trivializations in our previous work \cite{esen2014tulczyjew}, we shall present all reductions of dynamics on iterated bundles of a Lie group with the convenient trivialization of the first kind. In trivialization of the first kind, we identify tangent $TG$ and cotangent $T^{\ast}G$ bundles with their semidirect product trivializations $G\circledS \mathfrak{g}$ and $G\circledS\mathfrak{g}^{\ast}$, respectively. Then, we trivialize the iterated bundles $T\left( G\circledS\mathfrak{g}\right) ,$ $T\left( G\circledS\mathfrak{g}^{\ast}\right) ,$ $T^{\ast }\left( G\circledS \mathfrak{g}\right) $ and $T^{\ast}\left( G\circledS \mathfrak{g} ^{\ast}\right) $ by considering them as tangent and cotangent groups again. As an example, we obtain
\begin{equation}
TT^{\ast}G\simeq T\left( G\circledS\mathfrak{g}^{\ast}\right)
\simeq\left( G\circledS\mathfrak{g}^{\ast}\right) \circledS Lie\left(
G\circledS\mathfrak{g}^{\ast}\right) \simeq\left( G\circledS\mathfrak{g}%
^{\ast}\right) \circledS\left( \mathfrak{g}\circledS\mathfrak{g}^{\ast
}\right)  \label{1st}
\end{equation}
for which, the trivialization maps preserve lifted group structures thereby making possible various reductions of dynamics. On the other hand, in trivialization of the second kind, one distributes functors $T$ and $T^{\ast} $ to $G\circledS\mathfrak{g}$ and $G\circledS\mathfrak{g}^{\ast}$, obtains products of first order bundles and then, trivializes each factor involving the products. This results in, for example, 
\begin{equation}
^{2}TT^{\ast}G\simeq T\left( G\circledS\mathfrak{g}^{\ast}\right)
\rightarrow TG\circledS T\mathfrak{g}^{\ast}\simeq\left( G\circledS
\mathfrak{g}\right) \circledS\left( \mathfrak{g}^{\ast}\times\mathfrak{g}%
^{\ast}\right)  \label{2nd}
\end{equation}
for which distributions of functors mix up orders of fibrations, and do not preserve group structures \cite{esen2014tulczyjew}. Throughout this work we shall use trivialization of the first kind unless otherwise stated. A subscript of $\mathfrak{g}$ and $\mathfrak{g}^*$ will show its position in the original trivialization of iterated bundle. 

\subsection{Content of the work}~

Here is a brief description of what we present in each section. 

Section 2. This section is intended as a reference section of the present work. Notations and conventions are fixed. Trivializations of all spaces $TG$, $T^*G$, $T^*TG$, $TT^*G$, $T^*T^*G$, and their induced group structures are defined. Subgroups are listed. Subgroups with symplectic actions are identified. The trivialized form of the symplectic two-forms, as well as the associated one-forms and the invariant vector fields on the cotangent bundles and the Tulczyjew's symplectic space $TT^*G$ are given.

Section 3. The dynamics on the first order (both tangent and cotangent) bundles are considered. The first order Lagrangian and Hamiltonian dynamics on $TG$ and $T^*G$ are described by Euler-Lagrange and Hamilton's equations 
\begin{eqnarray}\label{preeulerlagrange-}
\frac{d}{dt}\frac{\delta\bar{L}}{\delta\xi}&=&T_e^{\ast}R_{g}\frac{\delta
\bar{L}}{\delta g}-\ad_{\xi}^{\ast}\frac{\delta\bar{L}}{\delta\xi},
\\ 
\label{ULP-}
\frac{dg}{dt}&=&T_{e}R_{g}\left( \frac{\delta\bar{H}}{\delta\mu}\right) ,\qquad \frac{d\mu}{dt} = \ad_{\frac{\delta\bar{H}}{\delta\mu}%
}^{\ast}\mu-T_{e}^{\ast}R_{g}\frac{\delta\bar{H}}{\delta g},
\end{eqnarray}
respectively. Reduction of (\ref{preeulerlagrange-}) by $G$ gives the Euler-Poincar\'{e} equations. Poisson and Marsden-Weinstein reductions on $T^*G$ are performed to obtain the Lie-Poisson equations. 

Section 4. Hamiltonian dynamics on $T^*TG$ is given by the equations
\begin{equation*}
\left( \frac{d}{dt}-ad_{\frac{\delta H}{\delta\mu}}^{\ast}\right) \left(
ad_{\xi}^{\ast}\nu-\mu\right) =T_{e}^{\ast}R_{g}\frac{\delta H}{\delta g},%
\text{ \ \ }\frac{dg}{dt}=T_{e}R_{g}\frac{\delta H}{\delta\mu}
\end{equation*}
equivalent to four component Hamilton's equations. There are remarkable differences arising from the use of different trivializations. Reductions by $G$, $\mathfrak{g}$ and $G\circledS\mathfrak{g}$ are performed. Structures of the reduced spaces are studied in detail. 

Section 5. Hamiltonian dynamics on $T^*T^*G$ is generated by the vector fields with components of the form 
\begin{equation*}
\begin{split}
\frac{dg}{dt} & =T_{e}R_{g}\left( \frac{\delta H}{\delta\nu}\right) ,\qquad 
\frac{d\mu}{dt}  =\frac{\delta H}{\delta\xi}+ad_{\frac{\delta H}{\delta\nu }}^{\ast}\mu, \\
\frac{d\nu}{dt} & =ad_{\frac{\delta H}{\delta\mu}}^{\ast}\mu+ad_{\frac {
\delta H}{\delta\nu}}^{\ast}\nu-T_{e}^{\ast}R_{g}\left( \frac{\delta H}{\delta g}\right) -ad_{\xi}^{\ast}\frac{\delta H}{\delta\xi}, \qquad 
\frac{d\xi}{dt} =-\frac{\delta H}{\delta\mu}+[\xi,\frac{\delta H}{\delta\nu
}].  
\end{split}
\end{equation*}
Reductions by $G$, $\mathfrak{g}^*$ and $G\circledS\mathfrak{g}^*$ are performed. Structures of the reduced spaces are exhibited in details. The correspondence between the dynamics on $T^*T^*G$, and on $T^*TG$, is established by symplectic diffeomorphisms and Poisson maps. 

Section 6. On $TT^*G$, there are both Lagrangian and Hamiltonian formalisms. If a function $E$ on $TT^*G$ is regarded as a Hamiltonian, then the Hamilton's equations with the Tulczyjew symplectic structure are
\begin{equation*}
\begin{split}
\dot{g}&=TR_{g}\left( \frac{\delta E}{\delta\nu}\right), \qquad  
\dot {\mu
}=-\frac{\delta E}{\delta\xi}
,\qquad 
\dot{\xi}=\frac{\delta E}{\delta\mu
}
,
\\
\dot{\nu}&=ad_{\frac{\delta E}{\delta\nu}}^{\ast}\nu-T^{\ast}R_{g}\left( \frac{\delta E}{\delta g}\right).
\end{split}
\end{equation*}
Reduction by $G$ results in reduced Tulczyjew triplet considered in \cite{esen2015tulczyjew} before. Reductions of Tulczyjew structure by $\mathfrak{g}$, by a symplectic action of $\mathfrak{g}^*$ that may be connected with a symplectic diffeomorphism from $TT^*G$ to $T^*T^*G$, by  $G\circledS\mathfrak{g}$ and by $G\circledS\mathfrak{g}^*$ are studied in detail. 

If the function $E$ on $TT^*G$ is regarded as a Lagrangian density, it then gives the Euler-Lagrange dynamics
\begin{equation*}
\begin{split}
\frac{d}{dt}\left( \frac{\delta E}{\delta\xi}\right) &
=T_{e}^{\ast}R_{g}\left( \frac{\delta E}{\delta g}\right) -ad_{\frac{\delta E%
}{\delta\mu }}^{\ast}\mu+ad_{\xi}^{\ast}\left( \frac{\delta E}{\delta\xi}%
\right) -ad_{\frac{\delta E}{\delta\nu}}^{\ast}\nu \\
\frac{d}{dt}\left( \frac{\delta E}{\delta\nu}\right) & =\frac{\delta E}{%
\delta\mu}-ad_{\xi}\frac{\delta E}{\delta\nu}.
\end{split}
\end{equation*}
Reductions of these equations by $G$, $\mathfrak{g}^*$ and $G\circledS\mathfrak{g}^*$ are described. The latter gives the Euler-Poincar\'{e} equations on $\mathfrak{g} \circledS \mathfrak{g}^*$.

\section{Geometry of iterated bundles}

Let $G$ be a Lie group, $\mathfrak{g}=Lie\left( G\right) \simeq T_{e}G$ be
its Lie algebra, and $\mathfrak{g}%
^{\ast}=Lie^{\ast }\left( G\right) $ be the dual of $\mathfrak{g}$. We shall adapt
the letters%
\begin{equation}
g,h\in G,\text{ \ \ }\xi,\eta,\zeta\in\mathfrak{g},\text{ \ \ }\mu,\nu
,\lambda\in\mathfrak{g}^{\ast}  \label{G}
\end{equation}
as elements of the spaces shown. For a tensor field which is either right or
left invariant, we shall use $V_{g}\in T_{g}G$, $\alpha_{g}\in T_{g}^{\ast}G$%
, etc... We shall denote left and right multiplications on $G$ by $L_{g}$
and $R_{g}$, respectively. The right inner automorphism $
I_{g}=L_{g^{-1}}\circ R_{g} $
is a right representation of $G$ on $G$ satisfying $I_{g}\circ I_{h}=I_{hg}.$
The right adjoint action $Ad_{g}=T_{e}I_{g}$ of $G$ on $\mathfrak{g}$ is
defined as the tangent map of $I_{g}$ at the identity $e\in G$. The
infinitesimal right adjoint representation $ad_{\xi}\eta$ is $\left[ \xi
,\eta\right] $ and is defined as derivative of $Ad_{g}$ over the identity. A
right invariant vector field $X_{\xi}^{G}$ generated by $\xi\in\mathfrak{g}$
is of the form $
X_{\xi}^{G}\left( g\right) =T_{e}R_{g}\xi$.  
The identity $
\left[ \xi,\eta\right] = [ X_{\xi}^{G},X_{\eta}^{G} ] _{JL}$ 
defines the isomorphism between $\mathfrak{g}$ and the space $\mathfrak{X}
^{R} ( G ) $ of right invariant vector fields endowed with the
Jacobi-Lie bracket. The coadjoint action $Ad_{g}^{\ast}$ of $G$ on the dual $%
\mathfrak{g}^{\ast}$ of the Lie algebra $\mathfrak{g}$ is a right
representation and is the linear algebraic dual of $Ad_{g^{-1}}$, namely,
\begin{equation}
\left\langle Ad_{g}^{\ast}\mu,\xi\right\rangle =\left\langle
\mu,Ad_{g^{-1}}\xi\right\rangle  \label{dist*}
\end{equation}
holds for all $\xi\in\mathfrak{g}$ and $\mu\in\mathfrak{g}^{\ast}$. The
inverse element $g^{-1}$ appears in the definition \eqref{dist*} in order to
make $Ad_{g}^{\ast}$ a right action. The infinitesimal coadjoint action $%
ad_{\xi}^{\ast}$ of $\mathfrak{g}$ on $\mathfrak{g}^{\ast}$ is the linear
algebraic dual of $ad_{\xi}$. Note that, the infinitesimal generator of the
coadjoint action $Ad_{g}^{\ast}$ is minus the infinitesimal coadjoint action
$ad_{\xi}^{\ast}$, that is, if $g^{t}\subset G$ is a curve passing through
the identity in the direction of $\xi\in\mathfrak{g},$ then
\begin{equation}
\left. \frac{d}{dt}\right\vert
_{t=0}Ad_{g^{t}}^{\ast}\mu=-ad_{\xi}^{\ast}\mu.  \label{Adtoad}
\end{equation}

In the diagrams of this work, EL and EP will abbreviate Euler-Lagrange and
Euler-Poincaré equations, respectively, and PR, SR, LR, EPR and EPR
will denote Poisson, symplectic, Lagrangian, and Euler-Poincaré reductions, respectively.

\subsection{The first order tangent group $TG$}\label{subsect-tangent-group}~

The trivialization 
\begin{equation}\label{trTG}
tr_{TG}:TG\rightarrow G\circledS \G{g}_1, \qquad V_{g}\mapsto (g,T_{g}R_{g^{-1}}V_{g})=:(g,\xi),   
\end{equation}
enables us to endow $TG$ with the semi-direct product group structure on  $G\circledS\G{g}_1$ given by
\begin{equation}\label{tgtri}
( g,\xi^{(1)}) ( \tilde{g},\tilde{\xi}^{(1)}) =(
g\tilde{g},\xi^{(1)}+\Ad_{g}\tilde{\xi}^{(1)}),
\end{equation}
for any $\xi^{(1)}, \tilde{\xi}^{(1)} \in \G{g}_1 = \G{g}$. Accordingly, the Lie algebra of $TG\cong G\circledS \G{g}$ is the semi-direct sum Lie algebra $\G{g}_2\circledS \G{g}_3:=\G{g}\circledS \G{g}$ with the Lie bracket
\begin{equation}
[(\xi^{(2)},\xi^{(3)}),(\tilde{\xi}^{(2)},\tilde{\xi}^{(3)})] = ([\xi^{(2)},\tilde{\xi}^{(2)}], \ad_{\xi^{(2)}}\tilde{\xi}^{(3)}-\ad_{\tilde\xi^{(2)}} {\xi}^{(3)})
\end{equation}
for any $\xi^{(2)}, \tilde{\xi}^{(2)} \in \G{g}_2 = \G{g}$, and any $\xi^{(3)}, \tilde{\xi}^{(3)} \in \G{g}_3 = \G{g}$. Here, indices on the Lie algebra $\G{g}$ serve to distinguish the copies of $\G{g}$. For further details on tangent group see  \cite{hindeleh2006tangent, KolaMichSlov-book, marsden1991symplectic, michor2008topics, Rati80}.
\subsection{The first order cotangent group $T^*G$} ~

The cotangent bundle   $T^\ast G$ can also be endowed with a group structure borrowed from the semi-direct product group $G \circledS \G{g}^\ast$ via the right trivialization
\begin{equation}
tr_{T^{\ast}G}:T^{\ast}G\rightarrow G \circledS\G{g}^{\ast}, \qquad \alpha_{g}\mapsto \left(g,\,T_{e}^{\ast}R_{g}\alpha_{g} \right) .
\label{trT*G}
\end{equation}
The group operation on  $G \circledS \G{g}^\ast$ is  
\begin{equation}\label{rgc}
(g_1,\mu_1) ( g_2,\mu_2) :=\Big(g_1g_2,\,\mu_1+\Ad_{g_1}^\ast\mu_2\Big).  
\end{equation}
On the other hand, we can use to pull back the canonical 1-form $\theta_{T^{\ast}G}$ and the symplectic
2-form $\Omega_{T^{\ast}G}$ on the cotangent bundle $T^{\ast}G$ thereby decorating $G\circledS\G{g}^\ast $ with the structure of an exact symplectic manifold symplectic 2-form $\Omega_{G\circledS\G{g}^\ast }$ and a potential 1-form $\theta^{(\lambda,\eta)}_{G\circledS\G{g}^\ast }$.
 
A right invariant vector field $X_{( \xi,\nu) }^{G\circledS\G{g}^\ast }$ on $G\circledS\G{g}^\ast $ corresponding to $(\xi,\nu) \in \G{g}\circledS\G{g}^\ast $ at a point $(g,\mu)$ is given by \cite[App.B. (B.9)]{EsSu16}
\[
X_{(\xi, \nu) }^{G\circledS\G{g}^\ast }(g,\mu) =\Big( T_eR_{g}\xi,\nu+\ad^\ast_{\xi}\mu\Big).
\]
Accordingly, the values of the canonical 1-form $\theta^{(\lambda,\eta)}_{G\circledS\G{g}^\ast}$
and the symplectic 2-form $\Omega_{G\circledS\G{g}^\ast}$ on a right invariant vector
field are \cite{abraham1978foundations,AlekGrabMarmMich94,EsSu16,manga2015geometry}
\begin{align}
& \left\langle \theta^{(\lambda,\eta)}_{G\circledS\G{g}^\ast}, X_{( \xi,\nu) }^{G\circledS\G{g}^\ast}\right\rangle (g,\mu) 
=\left\langle \lambda,\xi\right\rangle+\left\langle \nu,\eta\right\rangle, \label{OhmT*G} \\
& \left\langle \Omega_{G\circledS\G{g}^\ast};\left( X_{( \xi_1,\nu_1) }^{G\circledS\G{g}^\ast},X_{(\xi_2,\nu_2) }^{G\circledS\G{g}^\ast}\right) \right\rangle \left(g,\mu\right) 
=\langle \nu_1,\xi_2\rangle - \langle \nu_2,\xi_1\rangle -\langle \mu,[ \xi_1,\xi_2] \rangle. \label{Ohm2T*G}
\end{align}

\begin{remark}
The symplectic 2-form $\Omega_{G\circledS\G{g}^\ast}$ is not
conserved under the group operation \eqref{rgc}, as such, $G\circledS\G{g}^\ast$ is not a symplectic Lie group as defined in \cite{lichnerowicz1988lie}.
\end{remark}

\subsection{The cotangent group of tangent group $T^*TG$}

\subsubsection{Trivialization}

The global trivialization of $%
T^{\ast}TG\simeq T^{\ast}\left( G\circledS\mathfrak{g}\right) $ can be
achieved by trivializing $T^{\ast}\left( G\circledS\mathfrak{g}\right) $
into the semidirect product group $G\circledS\mathfrak{g}_{1}$ and the dual $%
\mathfrak{g}_{2}^{\ast}\times\mathfrak{g}_{3}^{\ast}$ of its Lie algebra $%
\mathfrak{g}_{2}\circledS\mathfrak{g}_{3}$
\begin{equation} \label{trT*TG}
\begin{split}
tr_{T^{\ast}\left( G\circledS\mathfrak{g}\right) }& :T^{\ast}\left(
G\circledS\mathfrak{g}_{1}\right) \rightarrow\left( G\circledS \mathfrak{g}%
_{1}\right) \circledS\left( \mathfrak{g}_{2}^{\ast}\times\mathfrak{g}%
_{3}^{\ast}\right)  \\
& :\left( \alpha_{g},\alpha_{\xi}\right) \rightarrow\left( g,\xi
,T_{e}^{\ast}R_{g}\left( \alpha_{g}\right) +\ad_{\xi}^{\ast}\alpha_{\xi
},\alpha_{\xi}\right)  
\end{split} 
\end{equation}
which preserves the group multiplication rule
\begin{equation}\label{GrT*TG}
\begin{split}
& \left( g,\xi,\mu_{1},\mu_{2}\right) \left( h,\eta,\nu_{1},\nu_{2}\right)
\\
&\qquad \qquad  =\left( gh,\xi+\Ad_{g}\eta,\mu_{1}+\Ad_{g}^{\ast}\left( \nu
_{1}+\ad_{\Ad_{g^{-1}}\xi}^{\ast}\nu_{2}\right) ,\mu_{2}+\Ad_{g}^{\ast}\nu
_{2}\right)   \\
&\qquad \qquad =\left( gh,\xi+\Ad_{g}\eta,\mu_{1}+ \Ad_{g}^{\ast}\nu_{1}+\ad_{\xi}^{\ast}\Ad_{g}^{\ast}\nu_{2} ,\mu_{2}+\Ad_{g}^{\ast}\nu
_{2}\right) 
\end{split} 
\end{equation}
on $T^{\ast}TG$ and results in the following subgroups.

\begin{proposition}
The canonical immersions of the following submanifolds
\begin{equation} \label{immg}
\begin{split}
&G, \quad \mathfrak{g}_{1}, \quad\mathfrak{g}_{2}^{\ast}, \quad\mathfrak{%
g}_{3}^{\ast}, \quad G\circledS\mathfrak{g}_{1}, \quad G\circledS\mathfrak{%
g}_{2}^{\ast}, \quad G\circledS\mathfrak{g}_{3}^{\ast }, \quad\mathfrak{g}%
_{2}^{\ast}\times\mathfrak{g}_{3}^{\ast},  \\
&\mathfrak{g}_{1}\circledS
(\mathfrak{g}_{2}^{\ast} \times\mathfrak{g}_{3}^{\ast}),\quad (G\circledS  \mathfrak{g}_{1})\circledS%
\mathfrak{g}_{2}^{\ast} , \quad G\circledS\left( \mathfrak{g}%
_{2}^{\ast}\times\mathfrak{g}_{3}^{\ast}\right) 
\end{split} 
\end{equation}
define subgroups of $ T^{\ast}TG$ and hence they act on $ T^{\ast}TG$
by actions induced from the multiplication in Eq.(\ref{GrT*TG}).
\end{proposition}

Here, the group structure on $G\circledS\mathfrak{g}_{1}$ is the one given
in \eqref{tgtri} whereas the group structure on are in the form of Eq.(\ref{rgc}) and, we obtain the
multiplications%
\begin{align}
\left( g,\xi,\mu\right) \left( h,\eta,\nu\right) & =(
gh,\xi+\Ad_{g}\eta,\mu+\Ad_{g}^{\ast}\nu)  \label{GrP} \\
\left( g,\mu_{1},\mu_{2}\right) \left( h,\nu_{1},\nu_{2}\right) & =(
gh,\mu_{1}+\Ad_{g}^{\ast}\nu_{1},\mu_{2}+\Ad_{g}^{\ast}\nu _{2})
\label{GrGg*g*} \\
\left( \xi,\mu_{1},\mu_{2}\right) \left( \eta,\nu_{1},\nu_{2}\right) & =(
\xi+\eta,\mu_{1}+\nu_{1}+\ad_{\xi}^{\ast}\nu_{2},\mu_{2}+\nu_{2})
\label{Grgg*g*}
\end{align}
defining the group structures on $(G\circledS \mathfrak{g}_{1})\circledS%
\mathfrak{g}_{2}^{\ast} $, $G\circledS\left( \mathfrak{g}%
_{2}^{\ast}\times\mathfrak{g}_{3}^{\ast}\right) $ and $ \mathfrak{g}%
_{1}\circledS( \mathfrak{g}_{2}^{\ast} \times \mathfrak{g}_{3}^{\ast})$,
respectively.

\subsubsection{Symplectic Structure}

By requiring the trivialization $tr_{T^{\ast}\left( G\circledS\mathfrak{g}%
\right) } $ be a symplectic map, we define a canonical one-form $%
\theta_{ T^{\ast}TG}$\ and a symplectic two-form $\Omega_{ T^{\ast
}TG} $ on the trivialized cotangent bundle$\ T^{\ast}TG$. To this end,
we recall that a right invariant vector field $X_{\left( \eta_1,\eta_2,\nu
_{1},\nu_{2}\right) }^{\text{ } T^{\ast}TG}$ on $ T^{\ast}TG$ is
generated by an element $\left( \eta_1,\eta_2,\nu_{1},\nu_{2}\right) $
in the Lie algebra $\left( \mathfrak{g}\circledS\mathfrak{g}\right)
\circledS\left( \mathfrak{g}^{\ast}\times\mathfrak{g}^{\ast}\right) $ of $ T^{\ast}TG$ by means of the tangent lift of right translation on $ T^{\ast}TG$. At a point $\left( g,\xi,\mu_1,\mu_2\right) $ in $ T^{\ast
}TG$, the value of such a right invariant vector field reads
\begin{equation}
X_{\left( \eta_1,\eta_2,\nu_{1},\nu_{2}\right) }^{ T^{\ast
}TG}\left( g,\xi,\mu_1,\mu_2\right) =\left( T_{e}R_{g}\eta_1,\eta_2+\ad_{\eta_1}\xi,\nu_{1}+\ad_{\eta_1}^{\ast}\mu_1+\ad_{\eta_2}^{\ast}\mu_2,\nu_{2}+\ad_{\eta_1}^{\ast}\mu_2\right)  \label{rivfT*TG}
\end{equation}
and is an element of the fiber $T_{\left( g,\xi,\mu_1,\mu_2\right) }\left(
 T^{\ast}TG\right) $. The values of canonical forms $\theta_{
 T^{\ast}TG}$ and $\Omega_{T^{\ast}TG}$ on right invariant vector
fields can then be computed as
\begin{align}
\langle\theta^{(\nu_1,\nu_2,\eta_1,\eta_2)}_{ T^{\ast}TG};X_{\left( \xi_1,\xi_2,\mu_{1},\mu_{2}\right) }^{ T^{\ast}TG}\rangle & =\left\langle \nu_1,\xi_1\right\rangle +\left\langle \nu_2,\xi_2\right\rangle +\left\langle \mu_1,\eta_1\right\rangle+\left\langle \mu_2,\eta_2\right\rangle  \label{thet1T*TG} \\
\left\langle \Omega_{ T^{\ast}TG};\left( X_{\left( \xi_1,\xi_2
,\mu_{1},\mu_{2}\right) }^{ T^{\ast}TG},X_{\left( \eta_1,%
\eta_2,\nu_{1},\nu_{2}\right) }^{\
 T^{\ast}TG}\right) \right\rangle \left(g,\xi,\lambda_1,\lambda_2\right) & =\left\langle \mu_{1} + \ad^\ast_{\xi_1}\lambda_1 + \ad^\ast_{\xi_2}\lambda_2,\eta_1\right\rangle -\left\langle \nu_{1} ,\xi_1\right\rangle  +\notag \\
&  \left\langle \mu_2+\ad^\ast_{\xi_1}\lambda_2,\eta_2\right\rangle -\left\langle \nu_{2} ,\xi_2\right\rangle.  \notag
\end{align}
The musical isomorphism $\Omega_{T^{\ast}TG}^{\flat}$, induced from
the symplectic two-form $\Omega_{ T^{\ast}TG},$ maps the tangent bundle
$T( T^{\ast}TG)$ to the cotangent bundle $T^{\ast}( T^{\ast}TG)$. It
takes the right invariant vector field in Eq.(\ref{rivfT*TG}) to an element
of the cotangent bundle $T_{\left( g,\xi,\mu_1,\mu_2\right) }^{\ast}( T^{\ast
}TG)$ with coordinates
\begin{equation} \label{bemTT*G}
\begin{split}
\Omega_{ T^{\ast}TG}^{\flat}\left( X_{\left( \eta_1,\eta_2,\lambda
_{1},\lambda_{2}\right) }^{T^{\ast}TG}\left( g,\xi,\mu_1
,\mu_2\right) \right) & = T_{(g,\xi)}^\ast R_{(g,\xi)^{-1}}(\lambda_1,\lambda_2),-(\eta_1,\eta_2)  \\
& =\left( T_{g}^{\ast}R_{g^{-1}} \lambda
_{1}-\ad_{\Ad_{g^{-1}}\xi}^{\ast}\lambda_{2} ,\lambda_{2},-\eta_1,-\eta_2\right) .
\end{split} 
\end{equation}

\begin{remark}
The actions of the subgroups $\mathfrak{g}_{2}^{\ast}$ and $\mathfrak{g}%
_{3}^{\ast}$ are not symplectic, nor are any subgroup in the list of Eq.(\ref{immg}%
) containing $\mathfrak{g}_{2}^{\ast}$ and $\mathfrak{g}_{3}^{\ast}$. There
remains only the action of the group $G\circledS\mathfrak{g}$ to perform symplectic reduction on $T^*TG$. 
\end{remark}

\subsection{The cotangent group of cotangent group $T^*T^*G$}

The
global trivialization of the iterated cotangent bundle can be achieved by
semidirect product of the group $G\circledS\mathfrak{g}_{1}^{\ast}$ and
the dual $\mathfrak{g}_{2}^{\ast}\times\mathfrak{g}_{3}$ of its Lie algebra
\cite{esen2014tulczyjew}. The trivialization map 
\begin{equation}  \label{trT*T*G}
\begin{split}
tr_{T^{\ast}T^{\ast}G}  & :T^{\ast}\left( G\circledS\mathfrak{g}^{\ast
}\right) \rightarrow\left( G\circledS\mathfrak{g} ^{\ast}\right)
\circledS\left( \mathfrak{g} ^{\ast}\times\mathfrak{g} \right)  \\
& :\left( \alpha_{g},\alpha_{\mu}\right) \rightarrow\left( g,\mu
,T_{e}^{\ast}R_{g}\left( \alpha_{g}\right)
-\ad_{\alpha_{\mu}}^{\ast}\mu,\alpha_{\mu}\right)   
\end{split} 
\end{equation}
implies on $ T^{\ast}T^{\ast}G$, the group multiplication rule
\begin{equation}\label{Gr1T*T*G}
\left( g,\mu_1,\mu_{2},\xi\right) \left( h,\nu_1,\nu_{2},\eta\right)   =\left(
gh,\mu_1+\Ad_{g}^{\ast}\nu_1,\mu_{2}+\Ad_{g}^{\ast}%
\nu_{2}-\ad_{\Ad_{g}\eta}^{\ast}\mu_1,\xi+\Ad_{g}\eta
\right) .
\end{equation}

\begin{proposition}
Embeddings of following subspaces
\begin{equation}  \label{immg2}
\begin{split}
& G,\text{ }\mathfrak{g}_{1}^{\ast}, \quad\mathfrak{g}_{2}^{\ast}, \quad%
\mathfrak{g}_{3}, \quad G\circledS\mathfrak{g}_{1}^{\ast}, \quad G\circledS%
\mathfrak{g}_{2}^{\ast}, \quad G\circledS\mathfrak{g}_{3}, \quad\mathfrak{%
g}_{1}^{\ast}\circledS\mathfrak{g}_{2}^{\ast}, \quad\mathfrak{g}%
_{2}^{\ast}\times\mathfrak{g}_{3}, \\
& (G\circledS\mathfrak{g}_{1}^{\ast})\circledS\mathfrak{g}_{2}^{\ast}, \quad%
G\circledS\left( \mathfrak{g}_{2}^{\ast}\times\mathfrak{g}_{3}\right) , \quad \mathfrak{g}_{1}^{\ast}\circledS\left( \mathfrak{g}_{2}^{\ast}\times\mathfrak{g}_{3}\right)  
\end{split} 
\end{equation}
define subgroups of $ T^{\ast}T^{\ast}G$ and hence they act on $ T^{\ast}T^{\ast}G$ by actions induced from the multiplication in Eq.(\ref%
{Gr1T*T*G}).
\end{proposition}

The group structures on $G\circledS\left( \mathfrak{g}_{2}^{\ast}\times%
\mathfrak{g}_{3}\right) $, $G\circledS\mathfrak{g}_{1}^{\ast}\circledS%
\mathfrak{g}_{2}^{\ast}$, $\mathfrak{g}_{1}^{\ast} \circledS \left( \mathfrak{g}_{2}^{\ast}\times\mathfrak{g}%
_{3}\right)$ are (up to some reordering)
given by Eqs.(\ref{GrP}), (\ref{GrGg*g*}) and (\ref{Grgg*g*}), respectively.

\subsubsection{Symplectic Structure on $T^*T^*G$}

The canonical one-form and the symplectic two form on $T^{\ast}T^{\ast}G$
can be mapped by $tr_{T^{\ast}T^{\ast}G} $ to $ T^{\ast}T^{\ast}G$
based on the fact that the trivialization map is a symplectic
diffeomorphism. Consider a right invariant vector
field $X_{\left( \eta_1,\nu_{1},\nu_{2},\eta_2\right) }^{ T^{\ast}T^{\ast}G}$ generated by an element $\left(
\eta_1,\nu_{1},\nu_{2},\eta_2\right) $ in the Lie algebra $\left(
\mathfrak{g}\circledS\mathfrak{g}^{\ast}\right) \circledS\left( \mathfrak{g}%
^{\ast}\times\mathfrak{g}\right) $ of $ T^{\ast}T^{\ast}G$. At the point $%
\left( g,\mu_1,\mu_2,\xi\right) ,$ the right invariant vector
\begin{equation}
X_{\left( \eta_1,\nu_{1},\nu_{2},\eta_2\right) }^{T^{\ast
}T^{\ast}G}\left( g,\mu_1,\mu_2,\xi\right)=\left( TR_{g}\eta_1,\nu_{1}+\ad_{\eta_1}^{\ast}\mu_1,\nu_{2}+\ad_{\eta_1}^{\ast}\mu_2-\ad_{\xi}^{\ast}\nu_{1},\eta_2+\ad_{\eta_1}\xi\right)
\label{RiT*T*G}
\end{equation}
is an element of $T_{\left( g,\mu_1,\mu_2,\xi\right) }\left(  T^{\ast
}T^{\ast}G\right) .$ The values of canonical forms $\theta_{ T^{\ast
}T^{\ast}G}$ and $\Omega_{ T^{\ast}T^{\ast}G}$ at right invariant
vector fields can now be evaluated to be
\begin{align}
\langle\theta^{(\nu_1,\eta_1,\eta_2,\nu_2)}_{ T^{\ast}T^{\ast}G},X_{\left(
\xi_1,\mu_{1},\mu_{2},\xi_2\right) }^{T^{\ast}T^{\ast}G}\rangle
& =\left\langle \nu_1,\xi_1\right\rangle +\left\langle
\mu_{1},\eta_1\right\rangle  +\left\langle
\mu_{2},\eta_2\right\rangle  +  \left\langle \nu_2,\xi_2\right\rangle\label{thet1T*T*G} \\
\langle\Omega_{ T^{\ast}T^{\ast}G};\left( X_{\left( \xi_1
,\mu_{1},\mu_{2},\xi_2\right) }^{ T^{\ast}T^{\ast}G},X_{\left(
\eta_1,\nu_{1},\nu_{2},\eta_2\right) }^{ T^{\ast}T^{\ast}G}\right) \rangle (g,\lambda_1,\lambda_2,\zeta) & =\left\langle \mu_2 +\ad^\ast_{\xi_1}\lambda_2 - \ad^\ast_{\zeta}\mu_1,\eta_{1}\right\rangle - \left\langle \mu_1,\eta_1\right\rangle + \notag
\\
&\left\langle -\nu_2+\ad^\ast_{\zeta}\nu_1,\xi_1 \right\rangle +\left\langle \nu_1,\xi_{2}\right\rangle.
\label{OhmT*T*G}
\end{align}
The musical isomorphism $\Omega_{ T^{\ast}T^{\ast}G}^{\flat}$,
induced from the symplectic two-form $\Omega_{ T^{\ast}T^{\ast}G}$
in Eq.(\ref{OhmT*T*G}), maps $T\left(  T^{\ast}T^{\ast}G\right) $ to $%
T^{\ast}\left(  T^{\ast}T^{\ast}G\right) $. At the point $\left(
g,\mu_1,\mu_2,\xi\right) $, $\Omega_{ T^{\ast}T^{\ast}G}^{\flat}$
takes the vector in Eq.(\ref{RiT*T*G}) to the element%
\begin{equation}
\begin{split}
\Omega_{T^{\ast}T^{\ast}G}^{\flat}\left( X_{\left( \eta_1
,\nu_{1},\nu_{2},\eta_2\right) }^{T^{\ast}T^{\ast}G}%
\right) &=\left( T_{(g,\mu_1)}^{\ast}R_{(g,\mu_1)^{-1}}\left(
\nu_{2},\eta_2\right) ,-\left(\eta_1,\nu_1\right)\right) \\
& = \Big(T^\ast R_{g^{-1}}(\nu_2)-\ad^\ast_{\eta_2}\Ad^\ast_{g^{-1}}\mu_1, \eta_2, -\eta_1, -\nu_1\Big)
\end{split} 
\end{equation}
in $T_{\left( g,\mu_1,\mu_2,\xi\right) }^{\ast}\left(  T^{\ast}T^{\ast
}G\right) .$

\begin{remark}
Actions of subgroups $\mathfrak{g}_{2}^{\ast}$ and $\mathfrak{g}_{3}$, and
hence any subgroup in the list (\ref{immg2}) containing $\mathfrak{g}%
_{2}^{\ast}$ and $\mathfrak{g}_{3}$, are not symplectic. Thus, there remains
only the action of the group $G\circledS\mathfrak{g}_{1}^{\ast}$ to perform symplectic reduction on $T^*T^*G$. 
\end{remark}

\subsection{The tangent group of cotangent group $TT^*G$}
~

$TT^{\ast}G\simeq T\left( G\circledS\mathfrak{g}^{\ast}\right) $ can be
trivialized as semidirect product of the group $G\circledS\mathfrak{g}%
^{\ast}$ and its Lie algebra $\mathfrak{g}\circledS\mathfrak{g}^{\ast}$ by
\begin{equation}\label{trTT*G}
\begin{split}
tr_{TT^{\ast}G}  & :T\left( G\circledS\mathfrak{g}^{\ast}\right)
\rightarrow\left( G\circledS\mathfrak{g}_{1}^{\ast}\right) \circledS\left(
\mathfrak{g}_{2}\circledS\mathfrak{g}_{3}^{\ast}\right)  
\\
& :\left( V_{g},V_{\mu}\right) \rightarrow\left(
g,\mu,TR_{g^{-1}}V_{g},V_{\mu}-\ad_{TR_{g^{-1}}V_{g}}^{\ast}\mu\right) ,
\end{split} 
\end{equation}
where $\left( V_{g},V_{\mu}\right) \in T_{\left( g,\mu\right) }\left(
G\circledS\mathfrak{g}^{\ast}\right) $ \cite{esen2014tulczyjew}. The group multiplication on $TT^{\ast}G$ is
\begin{equation}\label{GrTT*G}
\begin{split}
& \left( g,\mu_1,\xi,\mu_{2}\right) \left( h,\nu_1,\eta,\nu_{2}\right)  \\
& =\left(
gh,\mu_1+\Ad_{g}^{\ast}\nu_1,\xi+\Ad_{g}\eta,%
\mu_{2}+\Ad_{g}^{\ast}\nu_{2}-\ad_{\Ad_{g}\eta}^{\ast}\mu_1\right)
\end{split} 
\end{equation}
and embedded subgroups of $TT^{\ast}G$ follow.

\begin{proposition}
The embeddings of the subspaces
\begin{equation}
\begin{split}
& G,\quad\mathfrak{g}_{1}^{\ast},\quad\mathfrak{g}_{2},\quad\mathfrak{g}_{3}^{\ast}\text{, }%
G\circledS\mathfrak{g}_{1}^{\ast},\quad G\circledS\mathfrak{g}_{2},\quad G\circledS
\mathfrak{g}_{3}^{\ast},\quad\mathfrak{g}_{1}^{\ast}\circledS\mathfrak{g}%
_{3}^{\ast}, \quad \mathfrak{g}_{2}\circledS\mathfrak{g}_{3}^{\ast},   \\
& \left(G\circledS \mathfrak{g}_{1}^{\ast}\right)\circledS \mathfrak{g}_{3}^{\ast
}, \quad G\circledS\left( \mathfrak{g}_{2}\circledS\mathfrak{g}%
_{3}^{\ast}\right) ,\quad \mathfrak{g}_{1}^{\ast}\circledS\left( \mathfrak{g}%
_{2} \circledS\mathfrak{g}_{3}^{\ast}\right)   
\end{split} 
\end{equation}
of $ TT^{\ast}G$ define its subgroups. The group structures on $%
G\circledS\mathfrak{g,}$ $G\circledS\mathfrak{g}^{\ast}$ are defined by Eqs.(%
\ref{tgtri}) and (\ref{rgc}), respectively. The group structures on $%
\left( G\circledS\mathfrak{g}_{1}^{\ast}\right) \circledS\mathfrak{g}_{3}^{\ast
}$, $G\circledS\left( \mathfrak{g}_{2}\circledS\mathfrak{g}_{3}^{\ast
}\right) $ and $\mathfrak{g}_{1}^{\ast}\circledS\left( \mathfrak{g}_{2}
\circledS\mathfrak{g}_{3}^{\ast}\right)$ are defined (up to some reordering) by
Eqs.(\ref{GrGg*g*}),(\ref{GrP}) and (\ref{Grgg*g*}), respectively. The group
multiplications on $\mathfrak{g}_{1}^{\ast},$ $\mathfrak{g}_{2},$ $\mathfrak{%
g}_{3}^{\ast}$, $\mathfrak{g}_{1}^{\ast}\times\mathfrak{g}_{3}^{\ast}$ and $%
\mathfrak{g}_{2}\times\mathfrak{g}_{3}^{\ast}$ are vector additions.
\end{proposition}

\subsubsection{Tulczyjew symplectic strcuture on $TT^*G$.}

$TT^*G$ is central in Tulczyjew's triplet and carries a two-sided symplectic two-form. 
An element $\left( \eta_{1},\nu_{1},\eta_{2},\nu_{2}\right) $ in the
semidirect product Lie algebra $\left( \mathfrak{g}\circledS\mathfrak{g}%
^{\ast}\right) \circledS\left( \mathfrak{g}\circledS\mathfrak{g}^{\ast
}\right) $ defines a right invariant vector field on $\  TT^{\ast}G$ by
the tangent lift of right translation in $ TT^{\ast}G$. At a point $%
\left( g,\mu_1,\xi,\mu_2\right) $, a right invariant vector is given by%
\begin{equation}
X_{\left( \eta_{1},\nu_{1},\eta_{2},\nu_{2}\right) }^{ TT^{\ast}G}\left( g,\mu_1,\xi,\mu_2\right)=\left(
TR_{g}\eta_{1},\nu_{1}+\ad_{\eta_{1}}^{\ast}\mu_1,\eta_{2}+\ad_{\eta_1} \xi,\nu_{2}+\ad_{\eta_{1}}^{\ast}\mu_2-\ad_{\xi
}^{\ast}\nu_{1}\right) .  \label{RITT*G}
\end{equation}
The bundle $T\left( G\circledS\mathfrak{g}^{\ast}\right) $ carries
Tulczyjew's symplectic two-form $\Omega_{T\left( G\circledS\mathfrak{g}%
^{\ast}\right) }$ with two potential one-forms. The one-forms $\theta_{1}$
and $\theta_{2}$ are obtained by taking derivations of the symplectic
two-form $\Omega_{G\circledS\mathfrak{g}^{\ast}}$ and the canonical one-form
$\theta_{G\circledS\mathfrak{g}^{\ast}}$ respectively in Eq.(\ref{OhmT*G})
\cite{esen2014tulczyjew}. By requiring the trivialization $tr_{TT^{\ast}G}$ in Eq.(%
\ref{trTT*G}) be a symplectic mapping, we obtain an exact symplectic
structure $\Omega_{ TT^{\ast}G}$ with two potential one-forms $\theta_{1}$
and $\theta_{2}$ taking the values%
\begin{align}
& \left\langle \Omega_{\  TT^{\ast}G};\left( X_{\left( \xi_{2},\nu
_{2},\xi_{3},\nu_{3}\right) }^{ TT^{\ast}G},X_{\left( \bar{\xi}_{2},%
\bar{\nu}_{2},\bar{\xi}_{3},\bar{\nu}_{3}\right) }^{ TT^{\ast}G}\right)
\right\rangle \left( g,\mu,\xi,\nu\right) =\left\langle \nu_{3},\bar{\xi}_{2}\right\rangle -\left\langle
\nu_{2},\bar{\xi}_{3}\right\rangle +\left\langle \bar{\nu}%
_{2},\xi_{3}\right\rangle  \notag \\
&\hspace{5cm} -\left\langle \bar{\nu}_{3},\xi_{2}\right\rangle -\left\langle \nu,\left[
\xi_{2},\bar{\xi}_{2}\right] \right\rangle +\left\langle \xi,\ad^\ast_{\bar{\xi}_2}\nu_2 - \ad^\ast_{\xi_2}\bar{\nu}_2 \right\rangle ,  \label{SymTT*G} \\
& \left\langle \theta_{1}^{(\lambda_2,\eta_2,\lambda_3,\eta_3)},X_{\left( \xi_{2},\nu_{2},\xi_{3},\nu_{3}\right)
}^{ TT^{\ast}G}\right\rangle =\left\langle \lambda_2,\xi_{2}\right\rangle
+\left\langle \nu_{2},\eta_2\right\rangle +\left\langle \lambda_3,\xi _{3} \right\rangle +\left\langle \nu_3,\eta _{3} \right\rangle ,  \label{1} \\
& {\left\langle \theta_{2}^{(\lambda_2,\eta_2,\lambda_3,\eta_3)},X_{\left( \xi_{2},\nu_{2},\xi_{3},\nu_{3}\right)
}^{ TT^{\ast}G}\right\rangle =\left\langle \mu,\xi_{3}\right\rangle
+\left\langle \nu,\xi_{2}\right\rangle +\left\langle \mu,\left[ \xi,\xi _{2}%
\right] \right\rangle} ,  \label{2}
\end{align}
on right invariant vector fields of the form of Eq.(\ref{RITT*G}). At a point $%
\left( g,\mu,\xi,\nu\right) \in TT^{\ast}G$, the musical isomorphism $%
\Omega_{ TT^{\ast}G}^{\flat}$, induced from $\Omega_{ TT^{\ast}G}$%
, maps the image of a right invariant vector field $X_{\left( \xi_{2},\nu
_{2},\xi_{3},\nu_{3}\right) }^{ TT^{\ast}G}$ to an element
\begin{equation}
\Omega_{ TT^{\ast}G}^{\flat}( X_{\left( \xi_{2},\nu_{2},\xi_{3},\nu
_{3}\right) }^{ TT^{\ast}G}) =\Big( T_{g}^{\ast}R_{g^{-1}} \nu
_{3}-\ad_{\xi_3}^{\ast}\Ad^\ast_{g^{-1}}\mu ,\xi_{3}
,-\nu_{2},-\xi_{2}\Big)
\end{equation}
of $T_{\left( g,\mu,\xi,\nu\right) }^{\ast}\left(  TT^{\ast}G\right) .$

\section{Dynamics on the First Order Bundles}

\subsection{Lagrangian dynamics on tangent group $TG$}~

Given a Lagrangian
function $L:TG\to \B{R}$, let $\bar{L}:G\circledS \G{g}\to \B{R}$ be the corresponding function determined by $\bar{L}\circ tr_{TG} =L.$ The variation of the action integral of the latter is computed as
\begin{equation}\label{act}
\delta \int_{a}^{b}\bar{L}\left( \xi,g \right) dt=\int_{a}^{b} \, \left(\left\langle \frac{\delta \bar{L}}{\delta \xi },\delta \xi
\right\rangle _{e}+ 
\left\langle \frac{\delta \bar{L}}{\delta g},\delta g\right\rangle
_{g}\right)\,dt,  
\end{equation}%
applying the Hamilton's principle to the variations of the group (base) component, and the reduced variational principle
\begin{equation}\label{rvp}
\delta \xi =\dot{\eta}+\left[ \xi ,\eta \right]  
\end{equation}%
to the variations of the Lie algebra (fiber) component. For the reduced
variational principle we refer to \cite{CendMarsPekaRati03, esen2015tulczyjew, holm2008geometric,
Ho09,MarsdenRatiu-book} and for the Lagrangian dynamics on semidirect products to
\cite{BobeSuri99-II,cendra1998lagrangian, holm1998euler,ratiu1982euler,ratiu1982lagrange,marsden2000reduction,marsden1984semidirect}.
For the following result see \cite{BouMars09, colombo2013higher, colombo2013optimal, engo2003partitioned,
esen2015tulczyjew}.

\begin{proposition}
The trivialized Euler-Lagrange dynamics generated by a Lagrangian density $\bar{L}:G\circledS \G{g} \to \B{R}$ is given by 
\begin{equation}\label{preeulerlagrange}
\frac{d}{dt}\frac{\delta\bar{L}}{\delta\xi}=T_e^{\ast}R_{g}\frac{\delta
\bar{L}}{\delta g}-\ad_{\xi}^{\ast}\frac{\delta\bar{L}}{\delta\xi}.
\end{equation}
\end{proposition}

If, in addition, the Lagrangian density $\bar{L}:G\circledS \G{g} \to \B{R}$ is right invariant (namely it is independent of the group variable, that is $\bar{L}(g,\xi) =\ell( \xi)$), then Eq.\eqref{preeulerlagrange} reduces to the Euler-Poincaré equations on $\mathfrak{g}=(G\circledS \G{g})/G$
\begin{equation}\label{EPEq}
\frac{d}{dt}\frac{\delta l}{\delta\xi}=-\ad_{\xi}^{\ast}\frac{\delta l}{%
\delta\xi}  .
\end{equation}

Along the motion, for any Lagrangian $\bar{L}=\bar{L}(g,\xi)$, we compute that
\begin{equation}\label{calcc}
\begin{split}
\frac{dL}{dt}
&=\big\langle \frac{\delta \bar{L}}{\delta g},\dot{g} \big\rangle
+
\big\langle \frac{\delta \bar{L}}{\delta \xi},\dot{\xi} \big\rangle
=\big\langle \frac{\delta \bar{L}}{\delta g},T_eR_g \xi \big\rangle
+\big\langle \frac{\delta \bar{L}}{\delta \xi},\dot{\xi} \big\rangle
\\&=
\big\langle T_e^*R_g\frac{\delta \bar{L}}{\delta g},\xi \big\rangle
+\big\langle \frac{\delta \bar{L}}{\delta \xi},\dot{\xi} \big\rangle
\\
&=\big\langle\frac{d}{dt}\frac{\delta\bar{L}}{\delta\xi}+\ad_{\xi}^{\ast}\frac{\delta\bar{L}}{\delta\xi},\xi \big\rangle
+\big\langle \frac{\delta \bar{L}}{\delta \xi},\dot{\xi} \big\rangle
\\
&=\big\langle\frac{d}{dt}\frac{\delta\bar{L}}{\delta\xi},\xi \big\rangle
+\big\langle \frac{\delta \bar{L}}{\delta \xi},\dot{\xi} \big\rangle
=\frac{d}{dt}\big\langle \frac{\delta\bar{L}}{\delta\xi},\xi \big\rangle
\end{split}
\end{equation}
where, according to the trivialization \eqref{trTG},  we have employed the identification $\dot{g}=T_eR_g \xi$ in the first line whereas, we substitute the Euler-Lagrange equations (\ref{preeulerlagrange}) in the third line. The calculation \eqref{calcc} reads that the quantity $
\big\langle {\delta\bar{L}}/{\delta\xi},\xi \big\rangle-L
$
is a constant of the motion.

\subsection{Hamiltonian dynamics on cotangent group $T^*G$}~

Given $\bar{H}:G\circledS \G{g}^*\to \B{R}$, one obtains the Hamilton's equations 
\begin{equation}\label{ULP}
\frac{dg}{dt}=T_{e}R_{g}\left( \frac{\delta\bar{H}}{\delta\mu}\right) ,\qquad \frac{d\mu}{dt} = \ad_{\frac{\delta\bar{H}}{\delta\mu}%
}^{\ast}\mu-T_{e}^{\ast}R_{g}\frac{\delta\bar{H}}{\delta g}
\end{equation}
on the semidirect product $G\circledS \G{g}^*$ from the very definition
\begin{equation}\label{HamEq}
i_{X_{\bar{H}}}^{G\circledS \G{g}^*}\Omega _{G\circledS \G{g}^*}=-d\bar{H},
\end{equation}
where the right invariant vector field
\begin{equation}\label{HamVF}
X_{\bar{H}}^{G\circledS \G{g}^*} (g,\mu ) :=\left(T_{e}R_{g}\frac{\delta \bar{H}}{\delta \mu },
 \ad_{\frac{\delta \bar{H}}{\delta \mu }}^{\ast }\mu- T_{e}^{\ast }R_{g}\frac{\delta \bar{H}}{\delta g} \right)  
\end{equation}
is the  Hamiltonian vector field  associated to
\begin{equation}
\left(\frac{\delta \bar{H}}{\delta \mu},-T_{e}^{\ast }R_{g} \frac{\delta \bar{H}}{\delta g} \right)\in   \G{g}\circledS \G{g}^\ast. 
\end{equation}
For further details of Hamiltonian dynamics on semi-direct products we refer the reader to \cite{BobeSuri99-II,colombo2014higher, cendra1998maxwell,esen2015tulczyjew, holm1986hamiltonian, MaMiOrPeRa07, marsden1984semidirect,
marsden1984reduction--, marsden1991symplectic,marsden2000reduction,marsden1998symplectic, Rati80}. 

The canonical Poisson bracket of two functionals 
$\bar{F},\bar{K}: G \circledS \G{g}^{\ast} \to \B{R}$. at a point $(g,\mu) \in G \circledS \G{g}^{\ast}$ is given by
\begin{equation}\label{PoissonGg*}
\left\{ \bar{F},\bar{K}\right\} _{G \circledS \G{g}^{\ast} } (g,\mu) =\left\langle T_{e}^{\ast}R_{g}\frac{\delta\bar{F}}{\delta g},%
\frac{\delta\bar{K}}{\delta\mu}\right\rangle -\left\langle T_{e}^{\ast }R_{g}%
\frac{\delta\bar{K}}{\delta g},\frac{\delta\bar{F}}{\delta\mu }\right\rangle
+\left\langle \mu,\left[ \frac{\delta\bar{F}}{\delta\mu},\frac{\delta\bar{K}%
}{\delta\mu}\right] \right\rangle.  
\end{equation}

\subsubsection{Reduction of $T^*G$ by $G$}

The right action of $G$ on $G\circledS\mathfrak{g}^{\ast}$ is
\begin{equation}
\left( G\circledS\mathfrak{g}^{\ast}\right) \times G \rightarrow G\circledS%
\mathfrak{g}^{\ast}:\left( \left( g,\mu\right); h \right) \rightarrow\left(
gh,\mu\right)  \label{GonGxg*}
\end{equation}
with the infinitesimal generator $X_{\left( \xi,0\right) }^{G\circledS
\mathfrak{g}^{\ast}}$. If $\bar{H},$ defined on $G\circledS\mathfrak{g}%
^{\ast }$, is independent of $g$, it becomes right invariant under $G$. In this
case, dropping the terms involving $\delta\bar{H}/\delta g$ in Poisson
bracket (\ref{PoissonGg*}) is the Poisson reduction $G\circledS\mathfrak{g}%
^{\ast }\rightarrow \left( G\circledS\mathfrak{g}^{\ast}\right)/G
\simeq\mathfrak{g}^{\ast}$. When $\bar{F}$ and $\bar{K}$ are independent of
the group variable $g\in G$, that is, when $\bar{F}=f\left( \mu\right) $ and
$\bar{K}=k\left( \mu\right) $, we have the Lie-Poisson bracket%
\begin{equation}
\left\{ f,k\right\} _{\mathfrak{g}^{\ast}}\left( \mu\right) =\left\langle
\mu,\left[ \frac{\delta f}{\delta\mu},\frac{\delta k}{\delta\mu}\right]\right\rangle  \label{LPbracket}
\end{equation}
from which the Lie-Poisson equations
\begin{equation}
\dot{\mu}=ad_{\frac{\delta h}{\delta\mu}}^{\ast}\mu.  \label{LP}
\end{equation}
on the dual space $\mathfrak{g}^{\ast}$ follows. The Lie-Poisson bracket given in
Eq.(\ref{LPbracket}) can also be obtained by pulling back the non-degenerate
Poisson bracket in Eq.(\ref{PoissonGg*}) with the embedding $\mathfrak{g}%
^{\ast}\rightarrow G\circledS\mathfrak{g}^{\ast}$.

For the symplectic leaves of this Poisson structure \cite{We83}, we apply
Marsden-Weinstein symplectic reduction theorem \cite{marsden1974reduction} to $G\circledS
\mathfrak{g}^{\ast}$ with the action of $G$. The action (\ref{GonGxg*}) is
symplectic and it induces the momentum mapping%
\begin{equation}
\mathbf{J}_{G\circledS\mathfrak{g}^{\ast}}:G\circledS\mathfrak{g}^{\ast
}\longrightarrow\mathfrak{g}^{\ast}:\left( g,\mu\right) \rightarrow\mu
\end{equation}
which is also Poisson, and hence it projects trivialized
Hamiltonian dynamics in \eqref{ULP} to the Lie-Poisson dynamics in \eqref{LP}%
.

The inverse image $\mathbf{J}_{G\circledS \mathfrak{g}^{\ast }}^{-1}\left(
\mu \right) \subset G\circledS \mathfrak{g}^{\ast }$ of a regular value $\mu
\in \mathfrak{g}^{\ast }$ consists of two-tuples $\left( g,\mu \right) $ for
$g\in G$ and fixed $\mu \in \mathfrak{g}^{\ast }$. We may identify $\mathbf{J%
}_{G\circledS \mathfrak{g}^{\ast }}^{-1}\left( \mu \right) $ with the group $%
G$. Let $G_\mu$ be the isotropy group of the coadjoint action $Ad^{\ast }$,
defined in \eqref{dist*}, preserving the momenta $\mu $. Then,
we have the isomorphism
\begin{equation}
\left. \mathbf{J}_{G\circledS \mathfrak{g}^{\ast }}^{-1}\left( \mu \right)
\right/ G_{\mu }\simeq \left. G\right/ G_{\mu }\simeq \mathcal{O}_{\mu }
\label{coadorb}
\end{equation}
identifying the equivalence class $[g]$ of $g$ in $G/G_\mu$ with the coadjoint orbit 
\begin{equation}
\mathcal{O}_{\mu }=\left\{ Ad_{g}^{\ast }\mu :g\in G\right\} .
\end{equation}%
through the point $\mu $ in $\mathfrak{g}^{\ast }$ \cite{marsden1983hamiltonian2}. We denote the
reduced symplectic two-form on $\mathcal{O}_{\mu }$ by $\Omega _{G\circledS
\mathfrak{g}^{\ast }}^{/G}\left( \mu \right) $ which is the
Kostant-Kirillov-Souriau two-form \cite{Ho09,marsden1983coadjoint, marsden1983hamiltonian2}.
The value of $\Omega _{G\circledS _{R}\mathfrak{g}^{\ast }}^{/G
}\left( \mu \right) $ on two vector fields $ad_{\xi }^{\ast }\mu $, $%
ad_{\eta }^{\ast }\mu $ in $T_{\mu }\mathcal{O}_{\mu }$\ is
\begin{equation}
\left\langle \Omega _{G\circledS \mathfrak{g}^{\ast }}^{/G};\left(
ad_{\xi }^{\ast }\mu ,ad_{\eta }^{\ast }\mu \right) \right\rangle
=-\left\langle \mu ,\left[ \xi ,\eta \right] _{\mathfrak{g}%
}\right\rangle .  \label{KKS}
\end{equation}

\subsubsection{Reduction of $T^*G$  by $G_{\protect\mu}$}

The isotropy subgroup $G_{\mu}$ acts on $G\circledS\mathfrak{g}^{\ast}$ as
described by Eq.(\ref{GonGxg*}). Then, a Poisson and a symplectic reductions
of dynamics are possible. The Poisson reduction of the
symplectic manifold $G\circledS\mathfrak{g}^{\ast}$ under the action of the
isotropy group $G_{\mu}$ results in
\begin{equation*}
\left( G\circledS\mathfrak{g}^{\ast}\right)/G_{\mu} \simeq \mathcal{%
O}_{\mu}\times\mathfrak{g}^{\ast},
\end{equation*}
with Poisson
bracket
\begin{equation}
\left\{ H,K\right\} _{\mathcal{O}_{\mu}\times\mathfrak{g}^{\ast}}\left(
\mu,\nu\right) =\left\langle \mu,\left[ \frac{\delta H}{\delta\mu},\frac{%
\delta K}{\delta\mu}\right] \right\rangle +\left\langle \nu,\left[ \frac{%
\delta H}{\delta\mu},\frac{\delta K}{\delta\nu}\right] -\left[ \frac{\delta K%
}{\delta\mu},\frac{\delta H}{\delta\nu}\right] \right\rangle 
\end{equation}
which is not the direct product of Lie-Poisson structures on $%
\mathcal{O}_{\mu}$ and $\mathfrak{g}^{\ast}$.

The coadjoint action of $G\circledS\mathfrak{g}$ on the dual $\mathfrak{g}%
^{\ast}\times\mathfrak{g}^{\ast}$ of its Lie algebra is
\begin{equation}
\Ad_{\left( g,\xi\right) }^{\ast}:\mathfrak{g}^{\ast}\times\mathfrak{g}%
^{\ast}\rightarrow\mathfrak{g}^{\ast}\times\mathfrak{g}^{\ast}:\left( \mu
,\nu\right) \mapsto\left( \Ad_{g}^{\ast}\mu
+\ad_{\xi}^{\ast}\Ad_g^\ast\nu ,\Ad_{g}^{\ast}\nu\right) .  \label{coad}
\end{equation}
The symplectic reduction of $G\circledS\mathfrak{g}^{\ast}$ under the action
of the isotropy subgroup $G_{\mu}$ results in the coadjoint orbit $\mathcal{O%
}_{\left( \mu,\nu\right) }$ in $\mathfrak{g}^{\ast}\times\mathfrak{g}^{\ast}$
through the point $\left( \mu,\nu\right) $ under the action in Eq.(\ref{coad}%
). The reduced symplectic two-form $\Omega _{\mathcal{O}_{\left(
\mu,\nu\right) }}$ takes the value%
\begin{equation}
\left\langle \Omega_{\mathcal{O}_{\left( \mu,\nu\right) }};\left(
\eta,\zeta\right) ,\left( \bar{\eta},\bar{\zeta}\right) \right\rangle \left(
\mu,\nu\right) =\left\langle \mu,\left[ \bar{\eta},\eta\right] \right\rangle
+\left\langle \nu,\left[ \bar{\eta},\zeta\right] -\left[ \eta,\bar{\zeta}%
\right] \right\rangle
\end{equation}
on two vectors $\left( \eta,\zeta\right) $ and $\left( \bar{\eta},\bar{\zeta}%
\right) $ in $T_{\left( \mu,\nu\right) }\mathcal{O}_{\left( \mu,\nu\right) }$%
.

We summarize reductions of the symplectic space $G\circledS\mathfrak{g}%
^{\ast }$ in the following diagram.%
\begin{equation}
\xymatrix{\mathfrak{g}^{\ast } \ar[dd]_{\txt{Poisson \\ embedding}} &&&&
\mathcal{O}_{\mu } \ar@{_{(}->}[llll]_{\txt{symplectic leaf}}
\ar[dd]^{\txt{symplectic \\ embedding}} \\ && G\circledS\mathfrak{g}^{\ast }
\ar[ull]|-{\text{PR by G}} \ar[urr]|-{\text{SR by G}}
\ar[dll]|-{\text{PR by } G_{\mu}} \ar[drr]|-{\text{SR by }G_{\mu}} \\
\mathcal{O}_{\mu }\circledS\mathfrak{g}^{\ast }&&&&\mathcal{O}_{(\mu,\nu)}
\ar@{^{(}->}[llll]^{\txt{symplectic leaf}}
\\
&&\text{\small Reductions of $T^*G=G\circledS \G{g}^*$}
}  \label{DiagramGg*}
\end{equation}

\subsection{The Legendre Transformation}~

For a (hyper)regular Lagrangian $%
\bar{L}$ on $TG=G\circledS\mathfrak{g}$, the Legendre transformation is 
\begin{equation*}
G\circledS\mathfrak{g}\rightarrow G\circledS\mathfrak{g}^{\ast}:\left(
g,\xi\right) \rightarrow \big( g,\frac{\delta\bar{L}}{\delta\xi}=\mu\big)
\end{equation*}
which identifies $\delta\bar{L}/\delta\xi$ with the
fiber variable $\mu$ of $G\circledS\mathfrak{g}^{\ast}$. Define
a Hamiltonian function
\begin{equation}
H\left( g,\mu\right) =\left\langle \mu,\xi\right\rangle -L\left(
g,\xi\right)  \label{HLeg}
\end{equation}
for which the Hamiltonian dynamics in Eq.(\ref{ULP}) gives Euler-Lagrange equations (\ref{preeulerlagrange}%
). When $\bar{L}$ is independent of the group variable we have Euler-Poincare equations (\ref%
{EPEq}) and the Legendre transformation $\mu=\delta l/\delta\xi$ maps these
to Lie-Poisson equations (\ref{LP}) with the Hamiltonian function%
\begin{equation*}
h\left( \mu\right) =\left\langle \mu,\xi\right\rangle -l\left( \xi\right).
\end{equation*}
When the Lagrangian density is degenerate, the fiber derivative is not
invertible hence a direct passage from the Lagrangian dynamics to
Hamiltonian one is not possible. One possible way to define a general
Legendre trnasformation, including the degenerate cases, is possible in
Tulczyjew's approach \cite{Tu77}. We refer \cite{esen2014tulczyjew,esen2015tulczyjew,grabowska2016tulczyjew} where the
Tulczyjew's triplet is constructed for Lie groups.

\section{Hamiltonian Dynamics on $T^*TG$}

For a Hamiltonian function(al) $H$ on the symplectic manifold $\left( T^{\ast }TG,\Omega _{ T^{\ast }TG}\right) ,$ the Hamilton's
equations read%
\begin{equation}
i_{X_{H}^{ T^{\ast }TG}}\Omega _{ T^{\ast }TG}=-dH,
\label{HamEqT*TG}
\end{equation}%
where the right invariant Hamiltonian vector field $X_{H}^{ T^{\ast }TG}$ is generated by \cite{abrunheiro2011cubic}
\begin{equation}
\left( \frac{\delta H}{\delta \mu },\frac{\delta H}{\delta \nu }%
,-T_{e}^{\ast }R_{g}\left( \frac{\delta H}{\delta g}\right) -ad_{\xi }^{\ast
}\left( \frac{\delta H}{\delta \xi }\right) ,-\frac{\delta H}{\delta \xi }%
\right)
\in
\left( \mathfrak{g}\circledS \mathfrak{g}\right)
\circledS \left( \mathfrak{g}^{\ast }\times \mathfrak{g}^{\ast }\right).
\end{equation}%

\begin{proposition}
Components of $X_{H}^{ T^{\ast }TG}$ are trivialized Hamilton's equations on $\left(T^{\ast}TG,\Omega _{T^{\ast}TG}\right) $ 
\begin{align}
\frac{dg}{dt} & =T_{e}R_{g}\frac{\delta H}{\delta\mu},\text{ \ \ }
\label{HamT*TG1} \\
\frac{d\xi}{dt} & =\frac{\delta H}{\delta\nu}-\ad_{\xi}\frac{\delta H}{%
\delta\mu},  \label{HamT*TG2} \\
\frac{d\mu}{dt} & =-T_{e}^{\ast}R_{g}\frac{\delta H}{\delta g}-\ad_{\xi
}^{\ast}\frac{\delta H}{\delta\xi}+\ad_{\frac{\delta H}{\delta\mu}%
}^{\ast}\mu+\ad_{\frac{\delta H}{\delta\nu}}^{\ast}\nu,  \label{HamT*TG3} \\
\frac{d\nu}{dt} & =-\frac{\delta H}{\delta\xi}+\ad_{\frac{\delta H}{\delta \mu%
}}^{\ast}\nu.  \label{HamT*TG4}
\end{align}
\end{proposition}

From the equations (\ref{HamT*TG2}) and (\ref{HamT*TG4}), we single out $%
\delta H/\delta\nu$ and $\delta H/\delta\xi$, respectively. By substituting
these into Eq.(\ref{HamT*TG3}), we obtain the system%
\begin{equation}
\left( \frac{d}{dt}-\ad_{\frac{\delta H}{\delta\mu}}^{\ast}\right) \left(
\ad_{\xi}^{\ast}\nu-\mu\right) =T_{e}^{\ast}R_{g}\frac{\delta H}{\delta g},%
\qquad \frac{dg}{dt}=T_{e}R_{g}\frac{\delta H}{\delta\mu}
\end{equation}
equivalent to Eq.(\ref{HamT*TG1})-(\ref{HamT*TG4}).

\begin{remark}
The Hamilton's equations (\ref{HamT*TG1})-(\ref{HamT*TG4}) have extra terms,
compared to ones, for example,  in \cite{colombo2014higher, gay2012invariant}, coming from the choice of trivialization preserving group structure.  The trivialization of \cite{colombo2014higher} is of the second
kind given by Eq.(\ref{2nd}) whereas Eq.(\ref{HamT*TG1})-(\ref{HamT*TG4})
results from trivializations of the first kind. Reference \cite{burnett2011geometric} studies geometric
integrators of this Hamiltonian dynamics.
\end{remark}

\subsection{Reduction of $T^*TG$ by $G$}

We shall first perform Poisson reduction of Hamiltonian system on $T^{\ast}TG$ under the action of $G$ given by%
\begin{equation}
\left( \left( h,\eta,\nu_{1},\nu_{2}\right);g \right) \rightarrow\left(
hg,\eta,\nu_{1},\nu_{2}%
\right) \notag
\end{equation}
for a right invariant  Hamiltonian $H=H\left(
\xi,\mu,\nu\right) $.

\begin{proposition} \label{Prop5-3}
The Poisson reduced manifold $\mathfrak{g}_{1}\circledS\left( \mathfrak{g}%
_{2}^{\ast}\times\mathfrak{g}_{3}^{\ast}\right) $ carries the Poisson
bracket
\begin{equation}\label{Poigxg*xg*}
\begin{split}
& \left\{ H,K\right\} _{\mathfrak{g}_{1}\circledS\left( \mathfrak{g}%
_{2}^{\ast}\times\mathfrak{g}_{3}^{\ast}\right) }\left( \xi,\mu,\nu\right)
= \left\langle \frac{\delta H}{\delta\xi},\frac{\delta K}{%
\delta\nu}\right\rangle -\left\langle \frac{\delta K}{\delta\xi},\frac{\delta H}{\delta\nu }%
\right\rangle+\left\langle \mu,\left[ \frac{\delta H}{\delta\mu },%
\frac{\delta K}{\delta\mu}\right] \right\rangle    \\
& -\left\langle \ad_{\xi}^{\ast}\frac{\delta K}{\delta\xi},\frac{\delta H}{%
\delta\mu}\right\rangle +\left\langle \ad_{\xi}^{\ast}\frac{\delta H}{%
\delta\xi},\frac{\delta K}{\delta\mu}\right\rangle +\left\langle \nu,\left[
\frac{\delta H}{\delta\mu},\frac{\delta K}{\delta\nu}\right] -\left[ \frac{%
\delta K}{\delta\mu},\frac{\delta H}{\delta\nu}\right] \right\rangle ,
\end{split}
\end{equation}
for two right invariant functionals $H$ and $K$ on $T^*TG$.
\end{proposition}

\begin{remark}\label{rem4-4}
$T^{\ast}\mathfrak{g}_{1}=\mathfrak{g}_{1}\times\mathfrak{g}_{3}^{\ast}$
carries a canonical Poisson bracket, and $\mathfrak{g}_{2}^{\ast}$ carries
Lie-Poisson bracket. The immersions $\mathfrak{g}_{1}\times\mathfrak{g}%
_{3}^{\ast}\rightarrow\mathfrak{g}_{1}\circledS\left( \mathfrak{g}_{2}^{\ast
}\times\mathfrak{g}_{3}^{\ast}\right) $ and $\mathfrak{g}_{2}^{\ast
}\rightarrow\mathfrak{g}_{1}\circledS\left( \mathfrak{g}_{2}^{\ast}\times%
\mathfrak{g}_{3}^{\ast}\right) $ are Poisson maps. However, the Poisson
structure described by Eq.(\ref{Poigxg*xg*}) on $\mathfrak{g}%
_{1}\circledS\left( \mathfrak{g}_{2}^{\ast}\times\mathfrak{g}%
_{3}^{\ast}\right) $ is not a direct product of these. In fact, direct product structure on $\mathfrak{g}%
_{1}\times\left( \mathfrak{g}_{2}^{\ast}\times\mathfrak{g}_{3}^{\ast}\right)
$ arises from the trivialization of for examples was done in \cite{colombo2014higher} and \cite%
{gay2012invariant}. In these cases the second line of (\ref{Poigxg*xg*}) disappears. 

\end{remark}

\begin{proposition}
The Marsden-Weinstein symplectic reduction by the action of $G$ on $\
 T^{\ast}TG$ with the momentum mapping
\begin{equation*}
\mathbf{J}_{ T^{\ast}TG}^{G}: T^{\ast}TG\rightarrow
\mathfrak{g} ^{\ast}:\left( g,\xi,\mu,\nu\right) \rightarrow\mu
\end{equation*}
results in the reduced symplectic two-form $\Omega_{ T^{\ast}TG}^{/G}$ on the reduced space $\mathcal{O}_{\mu}\times \mathfrak{g}_{1}\times\mathfrak{g}_{3}^{\ast}$. The value of $%
\Omega _{ T^{\ast}TG}^{/G}$ on two vectors $%
\left( \eta_{\mathfrak{g}^{\ast}}\left( \mu\right) ,\zeta,\lambda\right) $
and $\left( \bar{\eta}_{\mathfrak{g}^{\ast}}\left( \mu\right) ,\bar{\zeta},%
\bar{\lambda}\right) $ is%
\begin{equation}
\Omega_{ T^{\ast}TG}^{/G}\left( \left( \eta_{%
\mathfrak{g}^{\ast}}\left( \mu\right) ,\zeta,\lambda\right) ,\left( \bar{\eta%
}_{\mathfrak{g}^{\ast}}\left( \mu\right) ,\bar{\zeta},\bar{\lambda }\right)
\right) =\left\langle \lambda,\bar{\zeta}\right\rangle -\left\langle \bar{%
\lambda},\zeta\right\rangle -\left\langle \mu,[\eta ,\bar{\eta}]\right\rangle
\label{RedOhmT*TG}
\end{equation}
and the reduced Hamilton's equations for a right invariant Hamiltonian $H$
are

\begin{equation*}
\frac{d\zeta}{dt}=\frac{\delta H}{\delta\lambda},\text{ \ \ }\frac{d\lambda
}{dt}=-\frac{\delta H}{\delta\zeta},\text{\ \ \ }\frac{d\mu}{dt}=ad_{\frac{%
\delta H}{\delta\mu}}^{\ast}\mu.
\end{equation*}
\end{proposition}

\begin{remark}
The reduced space $\mathcal{O}_{\mu}\times\mathfrak{g}_{1}\times \mathfrak{g}%
_{3}^{\ast}$ is a symplectic leaf \cite{We83} for the Poisson manifold $%
\mathfrak{g}_{1}\circledS\left( \mathfrak{g}_{2}^{\ast}\times\mathfrak{g}%
_{3}^{\ast}\right) $ of Proposition \ref{Prop5-3} as well as for $\mathfrak{g}_{1}\times\mathfrak{g}%
_{2}^{\ast}\times\mathfrak{g}_{3}^{\ast}$ with direct product Poisson
structure described in Remark \ref{rem4-4} above.
\end{remark}

The symplectic two-form $\Omega_{T^{\ast}TG}^{/G}$ given in Eq.(\ref{RedOhmT*TG})\ on $\mathcal{O}_{\mu}\times\mathfrak{g}%
_{1}\times\mathfrak{g}_{3}^{\ast}$ is in a direct product form. Hence a
reduction is possible by the additive action of $\mathfrak{g}_1$ to the second factor in
$\mathcal{O}_{\mu}\times\mathfrak{g}_{1}\times\mathfrak{g}_{3}^{\ast}$.

\begin{proposition}
The momentum map of additive action of $\mathfrak{g}_1$ on the symplectic
manifold $(\mathcal{O}_{\mu}\times\mathfrak{g}_{1}\times\mathfrak{g}%
_{3}^{\ast},\Omega_{ T^{\ast}TG}^{/G})$ is
\begin{equation*}
\mathbf{J}_{\mathcal{O}_{\mu}\times\mathfrak{g}_{1}\times\mathfrak{g}%
_{3}^{\ast}}^{\mathfrak{g}}:\mathcal{O}_{\mu}\times\mathfrak{g}_{1}\times%
\mathfrak{g}_{3}^{\ast}\rightarrow\mathfrak{g}_{3}^{\ast}:\left(
\mu,\xi,\nu\right) \rightarrow\nu
\end{equation*}
and the symplectic reduction results in the orbit $\mathcal{O}_{\mu}$
with Kostant-Kirillov-Souriau two-form (\ref{KKS}).
\end{proposition}

\subsection{Reduction of $T^*TG$ by $\mathfrak{g}$}

The vector space structure of $\mathfrak{g}$ makes it an Abelian group, and
according to the immersion in Eq.(\ref{immg}), $\mathfrak{g}$ is an Abelian
subgroup of $ T^{\ast}TG$. It acts on the total space $ T^{\ast}TG$ by
\begin{equation}
\left( \left( h,\eta,\mu,\nu\right);\xi \right) \rightarrow\left(
h,\eta+\Ad_h\xi,\mu,\nu\right) .  \label{gonT*TG}
\end{equation}
Since the action of $G\circledS\mathfrak{g}$ on its cotangent bundle $T^{\ast}TG$ is symplectic, the subgroup $\mathfrak{g}$ of $G\circledS
\mathfrak{g}$ also acts on $ T^{\ast}TG$ symplectically. Following
results describes Poisson and symplectic reductions of $ T^{\ast}TG$ by $%
\mathfrak{g}$ assuming that functions $K=K\left( g,\mu,\nu\right)$ and $H=H\left( g,\mu,\nu\right) $
defined on $G\circledS\left( \mathfrak{g}_{2}^{\ast}\times\mathfrak{g}%
_{3}^{\ast}\right) $ are right invariant under the above action of $%
\mathfrak{g}$.

\begin{proposition}
Poisson reduction of $ T^{\ast}TG$ by Abelian subgroup $\mathfrak{g}$
gives the Poisson manifold $G\circledS\left( \mathfrak{g}_{2}^{\ast}\times
\mathfrak{g}_{3}^{\ast}\right) $ endowed with the Poisson bracket
\begin{equation}\label{PoGxg*xg*}
\begin{split}
\left\{ H,K\right\} _{G\circledS\left( \mathfrak{g}_{2}^{\ast}\times%
\mathfrak{g}_{3}^{\ast}\right) } & \left( g,\mu,\nu\right) = \left\langle T_{e}^{\ast}R_{g}\frac{\delta H}{\delta g},\frac{%
\delta K}{\delta\mu}\right\rangle -\left\langle
T_{e}^{\ast}R_{g}\frac{\delta K}{\delta g},\frac{\delta H}{\delta\mu}%
\right\rangle  \\
& + \left\langle \mu,\left[ \frac{\delta H}{\delta\mu},\frac{\delta K}{%
\delta\mu}\right] \right\rangle +\left\langle \nu,\left[ \frac{\delta H}{%
\delta\mu},\frac{\delta K}{\delta\nu}\right] -\left[ \frac{\delta K}{%
\delta\mu},\frac{\delta H}{\delta\nu}\right] \right\rangle .
\end{split}
\end{equation}
\end{proposition}

\begin{remark}
Recall that $G\times\mathfrak{g}_{2}^{\ast}$ is canonically symplectic with
the Poisson bracket in Eq(\ref{PoissonGg*}) and the immersion $G\times
\mathfrak{g}_{2}^{\ast}\rightarrow G\circledS\left( \mathfrak{g}_{2}^{\ast
}\times\mathfrak{g}_{3}^{\ast}\right) $ is a Poisson map. On the other hand,
$\mathfrak{g}_{3}^{\ast}$ is naturally Lie-Poisson and $\mathfrak{g}%
_{3}^{\ast}\rightarrow\mathfrak{g}_{1}\circledS\left( \mathfrak{g}_{2}^{\ast
}\times\mathfrak{g}_{3}^{\ast}\right) $ is also a Poisson map. The Poisson
bracket in Eq(\ref{PoGxg*xg*}) is, however, not a direct product of these
structures.
\end{remark}

\begin{proposition}
The Marsden-Weinstein symplectic reduction by the action of $\mathfrak{g}$
on $ T^{\ast}TG$ with the momentum mapping
\begin{equation*}
\mathbf{J}_{ T^{\ast}TG}^{\mathfrak{g}}: T^{\ast
}TG\rightarrow\mathfrak{g}_{3}^{\ast}:\left( g,\xi,\mu,\nu\right)
\rightarrow\nu
\end{equation*}
results in the reduced symplectic space $\left( \mathbf{J}_{T^{\ast}TG}^{\mathfrak{g}}\right) ^{-1}/\mathfrak{g}$ isomorphic to $%
G\circledS\mathfrak{g}_{2}^{\ast}$ and with the canonical symplectic
two-from $\Omega_{G\circledS\mathfrak{g}_{2}^{\ast}}$ in Eq(\ref{Ohm2T*G}).
\end{proposition}

It follows that the immersion $G\circledS\mathfrak{g}_{2}^{\ast}\rightarrow
G\circledS\left( \mathfrak{g}_{2}^{\ast}\times\mathfrak{g}_{3}^{\ast}\right)
$ defines symplectic leaves of the Poisson manifold $G\circledS\left(
\mathfrak{g}_{2}^{\ast}\times\mathfrak{g}_{3}^{\ast}\right) $. 
The symplectic reduction of $G\circledS\mathfrak{g}_{2}^{\ast}$ under the
action of $G$ results in the total space $\mathcal{O}_{\mu}$ with
Kostant-Kirillov-Souriau two-form (\ref{KKS}). We arrive at the following
proposition.

\begin{proposition}
Reductions by actions of $\mathfrak{g}$\ and $G$ make the following diagram
commutative
\begin{equation}
\xymatrix{& (G\circledS\mathfrak{g}_{1})\circledS(\mathfrak{g}_{2}^{\ast
}\times\mathfrak{g}_{3}^{\ast }) \ar[dl]|-{\text{SR by } \mathfrak{g}_{1}}
\ar[dr]|-{\text{SR by } G \text{ at } \mu} \\
G\circledS\mathfrak{g}_{2}^{\ast} \ar[dr]|-{\text{SR by } G \text{ at }
\mu} && \mathcal{O}_{\mu}\times\mathfrak{g}_{1}\times\mathfrak{g}_{3}^{\ast
} \ar[dl]|-{\text{SR by } \mathfrak{g}_{1}} \\ &\mathcal{O}_{\mu} 
\\
&\text{\small  Symplectic Reductions of $T^*TG$}
}
\end{equation}
\end{proposition}

Note that, the symplectic reduction of $ T^{\ast}TG$ by the total action
of the group $G\circledS\mathfrak{g}_{1}$ does not result in $\mathcal{O}%
_{\mu}$ as  reduced space. This is a matter of Hamiltonian reduction by
stages theorem \cite{MaMiOrPeRa07}. In the following subsection, we will
discuss the reduction of $ T^{\ast}TG$ under the action of $G\circledS%
\mathfrak{g}_{1}$ as well as the implications of the Hamiltonian reduction
by stages theorem for this case.

\subsection{Reduction of $T^*TG$ by $G\circledS\mathfrak{g}$}

The Lie algebra of $G\circledS\mathfrak{g}\ $is the space $\mathfrak{g}%
\circledS\mathfrak{g}$ endowed with the semidirect product Lie algebra
bracket
\begin{equation}
\left[ \left( \xi_{1},\xi_{2}\right) ,\left( \eta_{1},\eta_{2}\right) \right]
_{\mathfrak{g}\circledS\mathfrak{g}}=\left( \left[ \xi_{1},\eta _{1}\right] ,%
\left[ \xi_{1},\eta_{2}\right] -\left[ \eta_{1},\xi _{2}\right] \right)
\end{equation}
for $\left( \xi_{1},\xi_{2}\right) $ and $\left( \eta_{1},\eta_{2}\right) $
in $\mathfrak{g}\circledS\mathfrak{g}$. Accordingly, the dual space $%
\mathfrak{g}_{2}^{\ast}\times\mathfrak{g}_{3}^{\ast}$ has the Lie-Poisson
bracket
\begin{align}
\left\{ F,E\right\} _{\mathfrak{g}_{2}^{\ast}\times\mathfrak{g}%
_{3}^{\ast}}\left( \mu,\nu\right) =\left\langle \mu,\left[ \frac{\delta F}{\delta\mu},\frac{\delta E}{%
\delta\mu}\right] \right\rangle +\left\langle \nu,\left[ \frac{\delta F}{%
\delta\mu},\frac{\delta E}{\delta\nu}\right] -\left[ \frac{\delta E}{%
\delta\mu},\frac{\delta F}{\delta\nu}\right] \right\rangle  \label{LPBg*xg*}
\end{align}
for two functionals $F$ and $E$ on $\mathfrak{g}_{2}^{\ast}\times \mathfrak{g%
}_{3}^{\ast}$.

\begin{proposition}
The Lie-Poisson structure on $\mathfrak{g}_{2}^{\ast}\times\mathfrak{g}%
_{3}^{\ast}$ is given by the bracket in Eq.(\ref{LPBg*xg*}) and the
Lie-Poisson equations for a function $H\left( \mu,\nu\right) $ on $\mathfrak{%
g}_{2}^{\ast}\times\mathfrak{g}_{3}^{\ast}$ read%
\begin{equation}
\frac{d\mu}{dt}=\ad_{\frac{\delta H}{\delta\mu}}^{\ast}\mu+\ad_{\frac{\delta H%
}{\delta\nu}}^{\ast}\nu,\text{ \ \ }\frac{d\nu}{dt}=\ad_{\frac{\delta H}{%
\delta\mu}}^{\ast}\nu.  \label{LPg*g*}
\end{equation}
\end{proposition}

Alternatively, the Lie-Poisson equations (\ref{LPg*g*}) can be obtained by
Poisson reduction of $T^{\ast }TG$ with the action of $G\circledS
\mathfrak{g}$ given by%
\begin{equation}
\left( \left( h,\eta ,\mu ,\nu \right); \left( g,\xi \right) \right)
\mapsto \left(\big( hg,\eta +\Ad_{h }\xi , \mu ,\nu \right)
\label{GxgonT*TG}
\end{equation}%
and restricting the Hamiltonian function $H$ to the fiber variables $\left(
\mu ,\nu \right) $. In this case, the Lie-Poisson dynamics of $g$ and $\xi $
remains the same but the dynamics governing $\mu $ and $\nu $ have the
reduced form given by Eq.(\ref{LPg*g*}). This is a manifestation of the fact
that the projections to the last two factors in the trivialization (\ref%
{trT*TG})\ is a momentum map under the left Hamiltonian action of the group $%
G\circledS \mathfrak{g}_{1}$ to its trivialized cotangent bundle $T^{\ast }TG$. Yet another way is to reduce the bracket (\ref{Poigxg*xg*}%
) on $\mathfrak{g}_{1}\circledS \left( \mathfrak{g}_{2}^{\ast }\times
\mathfrak{g}_{3}^{\ast }\right) $ by assuming that functionals depend on
elements of the dual spaces. That is, to consider the Abelian group action
of $\mathfrak{g}_{1}$ on $\mathfrak{g}_{1}\circledS \left( \mathfrak{g}%
_{2}^{\ast }\times \mathfrak{g}_{3}^{\ast }\right) $ given by%
\begin{equation}
\left( \left( \eta ,\mu ,\nu \right);\xi \right) \mapsto \left( \eta+\xi
 ,\mu  ,\nu \right)  \notag
\end{equation}%
and then, apply Poisson reduction. Note finally that, the immersion $%
\mathfrak{g}_{2}^{\ast }\times \mathfrak{g}_{3}^{\ast }\rightarrow \mathfrak{%
g}_{1}\circledS \left( \mathfrak{g}_{2}^{\ast }\times \mathfrak{g}_{3}^{\ast
}\right) $ is a Poisson map. 

Application of the Marsden-Weinstein reduction to the symplectic manifold $T^{\ast}TG$ results in the symplectic leaves of Poisson structure on $%
\mathfrak{g}_{2}^{\ast}\times\mathfrak{g}_{3}^{\ast}$. The action in Eq.(\ref%
{GxgonT*TG}) has the momentum mapping
\begin{equation*}
\mathbf{J}_{ T^{\ast}TG}^{G\circledS\mathfrak{g}}:T^{\ast}TG\rightarrow\mathfrak{g}_{2}^{\ast}\times\mathfrak{g}_{3}^{\ast
}:\left( g,\xi,\mu,\nu\right) \rightarrow\left( \mu,\nu\right) .
\end{equation*}
The pre-image $\left( \mathbf{J}_{ T^{\ast}TG}^{G\circledS
\mathfrak{g}}\right) ^{-1}\left( \mu,\nu\right) $ of an element $\left(
\mu,\nu\right) \in\mathfrak{g}_{2}^{\ast}\times\mathfrak{g}_{3}^{\ast}$ is
diffeomorphic to $G\circledS\mathfrak{g}$. The isotropy group $\left(
G\circledS\mathfrak{g}\right) _{\left( \mu,\nu\right) }$ of $\left(
\mu,\nu\right) $ consists of pairs $\left( g,\xi\right) $ in $G\circledS
\mathfrak{g}$ satisfying
\begin{equation}
\Ad_{\left( g,\xi\right) }^{\ast}\left( \mu,\nu\right) =\left(
\Ad_{g}^{\ast}\left( \mu+\ad_{\xi}^{\ast}\nu\right) ,\Ad_{g}^{\ast}\nu\right)
=\left( \mu,\nu\right)  \label{coad2}
\end{equation}
which means that $g\in G_{\nu}\cap G_{\mu}$ and the representation of $%
Ad_{g}\xi$ on $\nu$ is null, that is $ad_{Ad_{g}\xi}^{\ast}\nu=0.$ From the
general theory, we deduce that the quotient space
\begin{equation*}
\left. \left( \mathbf{J}_{ T^{\ast}TG}^{G\circledS\mathfrak{g}%
}\right) ^{-1}\left( \mu,\nu\right) \right/ \left( G\circledS \mathfrak{g}%
\right) _{\left( \mu,\nu\right) }\simeq\mathcal{O}_{\left( \mu,\nu\right) }
\end{equation*}
is diffeomorphic to the coadjoint orbit $\mathcal{O}_{\left( \mu,\nu\right)
} $ in $\mathfrak{g}_{2}^{\ast}\times\mathfrak{g}_{3}^{\ast}$ through the
point $\left( \mu,\nu\right) $ under the action $Ad_{\left( g,\xi\right)
}^{\ast}$ in Eq.(\ref{coad2}), that is,
\begin{equation}
\mathcal{O}_{\left( \mu,\nu\right) }=\left\{ \left( \mu,\nu\right) \in%
\mathfrak{g}_{2}^{\ast}\times\mathfrak{g}_{3}^{\ast}:Ad_{\left( g,\xi\right)
}^{\ast}\left( \mu,\nu\right) =\left( \mu,\nu\right) \right\} .
\label{Omunu}
\end{equation}

\begin{proposition}
The symplectic reduction of $ T^{\ast}TG$ results in the coadjoint orbit $
\mathcal{O}_{\left( \mu,\nu\right) }$ in $\mathfrak{g}_{2}^{\ast}\times%
\mathfrak{g}_{3}^{\ast}$ through the point $\left( \mu,\nu\right) $. The
reduced symplectic two-form $\Omega_{ T^{\ast}TG}^{G\circledS \mathfrak{%
g}\backslash}$ (denoted simply by $\Omega_{\mathcal{O}_{\left(
\mu,\nu\right) }}$) takes the value%
\begin{equation}
\left\langle \Omega_{\mathcal{O}_{\left( \mu,\nu\right) }};\left(
\eta,\zeta\right) ,\left( \bar{\eta},\bar{\zeta}\right) \right\rangle \left(
\mu,\nu\right) =\left\langle \mu,\left[ \bar{\eta},\eta\right] \right\rangle
+\left\langle \nu,\left[ \bar{\eta},\zeta\right] -\left[ \eta,\bar{\zeta}%
\right] \right\rangle  \label{Symp/Gxg}
\end{equation}
on two vectors $\left( \eta,\zeta\right) $ and $\left( \bar{\eta},\bar{\zeta}%
\right) $ in $T_{\left( \mu,\nu\right) }\mathcal{O}_{\left( \mu,\nu\right) }$%
.
\end{proposition}

This reduction can also be achieved by stages as described in \cite%
{holm1998euler, MaMiOrPeRa07}. That is, first trivialize dynamics by the action
of Lie algebra $\mathfrak{g}$ on $ T^{\ast}TG$ which results in the
Poisson structure on the product $G\circledS\left( \mathfrak{g}%
_{2}^{\ast}\times\mathfrak{g}_{3}^{\ast}\right) $ given by Eq.(\ref%
{PoGxg*xg*}). Symplectic leaves of this Poisson structure are spaces
diffeomorphic to $G\circledS\mathfrak{g}_{2}^{\ast}$ with symplectic
two-form given in Eq.(\ref{Ohm2T*G}). The isotropy group $G_{\mu}$ of an
element $\mu\in\mathfrak{g}^{\ast}$ acts on $G\circledS\mathfrak{g}%
_{2}^{\ast}$ by the same way as assigned in Eq.(\ref{GonGxg*}), that is,
\begin{equation}
\left( G\circledS\mathfrak{g}_{2}^{\ast}\right) \times G_{\mu} \rightarrow
G\circledS\mathfrak{g}_{2}^{\ast}:\left(\left( h,\lambda\right);g \right)
\rightarrow\left( hg,\lambda\right) .  \label{Gmu}
\end{equation}
Then, the Hamiltonian reduction by stages theorem states that, the
symplectic reduction of $G\circledS\mathfrak{g}_{2}^{\ast}$ under the action
of $G_{\mu}$ will result in $\mathcal{O}_{\left( \mu,\nu\right) }$ as the
reduced space endowed with the symplectic two-form $\Omega_{\mathcal{O}%
_{\left( \mu ,\nu\right) }}$ in Eq.(\ref{Symp/Gxg}). Following diagram
summarizes the Hamiltonian reduction by stages theorem for the case of $%
 T^{\ast}TG$ under consideration
\begin{equation}
\xymatrix{&& (G\circledS\mathfrak{g}_{1})\circledS(\mathfrak{g}_{2}^{\ast
}\times\mathfrak{g}_{3}^{\ast }) \ar[dll]|-{\text{SR by } \mathfrak{g}_{1}
\text{ at }\mu } \ar[dd]|-{\text{SR by } G\times\mathfrak{g}_{1} \text{ at
}(\mu,\nu)} \\ G\circledS\mathfrak{g}_{2}^{\ast} \ar[drr]|-{\text{SR by }
G_{\mu} \text{ at } \nu} \\ &&\mathcal{O}_{(\mu,\nu)} 
\\&&
\text{\small Hamiltonian reduction by stages for $T^*TG$}
}  \label{HRBS}
\end{equation}
There exists a momentum mapping $\mathbf{J}_{G\circledS\mathfrak{g}%
_{2}^{\ast }}^{G_{\mu}}$ from $G\circledS\mathfrak{g}_{2}^{\ast}$ to the
dual space $\mathfrak{g}_{\mu}^{\ast}$ of the isotropy subalgebra $\mathfrak{%
g}_{\mu}$ of $G_{\mu}$. Isotropy subgroup $G_{\mu,\nu}$ of the coadjoint
action is
\begin{equation*}
G_{\mu,\nu}=\left\{ g\in G_{\mu}:\Ad_{g^{-1}}^{\ast}\nu=\nu\right\} .
\end{equation*}
The quotient symplectic space
\begin{equation*}
\left. \left( \mathbf{J}_{G\circledS\mathfrak{g}_{2}^{\ast}}^{G_{\mu}}%
\right) ^{-1}\left( \nu\right) \right/ G_{\mu,\nu}\simeq\mathcal{O}_{\left(
\mu,\nu\right) }
\end{equation*}
is diffeomorphic to the coadjoint orbit $\mathcal{O}_{\left( \mu,\nu\right)
} $ defined in (\ref{Omunu}).

It is also possible to establish the Poisson reduction of the symplectic
manifold $G\circledS\mathfrak{g}_{2}^{\ast}$ under the action of the
isotropy group $G_{\mu}$. This results in
\begin{equation*}
G_{\mu}\backslash\left( G\circledS\mathfrak{g}_{2}^{\ast}\right) \simeq%
\mathcal{O}_{\mu}\times\mathfrak{g}_{2}^{\ast}
\end{equation*}
with the Poisson bracket
\begin{equation}
\left\{ H,K\right\} _{\mathcal{O}_{\mu}\times\mathfrak{g}_{2}^{\ast}}\left(
\mu,\nu\right) =\left\langle \mu,\left[ \frac{\delta H}{\delta\mu},\frac{%
\delta K}{\delta\mu}\right] \right\rangle +\left\langle \nu, \left[\frac{\delta H}{\delta\mu}, \frac{\delta K}{\delta\nu}\right]-\left[ \frac{%
\delta K}{\delta\mu}, \frac{\delta H}{\delta\nu}\right] \right\rangle .
\end{equation}
We note again, that, the Poisson structure on $\mathcal{O}_{\mu}\times%
\mathfrak{g}^{\ast}$ is not a direct product of the Lie-Poisson structures
on $\mathcal{O}_{\mu}$ and $\mathfrak{g}_{2}^{\ast}$. Following diagram
illustrates various reductions of $T^{\ast}TG$ under the actions of $G$,
$\mathfrak{g}$ and $G\circledS\mathfrak{g}$. Diagram (\ref{DiagramGg*}),
describing reductions of $G\circledS\mathfrak{g}^{\ast}$, can be attached to
the lower right corner of this to have a complete picture of reductions.
\begin{equation}
\xymatrix{ \mathfrak{g}_{1}\circledS(\mathfrak{g}_{2}^{\ast
}\times\mathfrak{g}_{3}^{\ast }) &&&& \mathcal{O}_{\mu }\times
\mathfrak{g}_{1}\times\mathfrak{g}_{3}^{\ast }
\ar@{_{(}->}[llll]_{\txt{symplectic leaf}} \\\\ \mathfrak{g}_{2}^{\ast
}\times\mathfrak{g}_{3}^{\ast } \ar[dd]_{\txt{Poisson \\ embedding}}
\ar[uu]^{\txt{Poisson \\ embedding}} &&(
G\circledS\mathfrak{g}_{1})\circledS(\mathfrak{g}_{2}^{\ast
}\times\mathfrak{g}_{3}^{\ast }) \ar[uull]|-{\text{PR by G}}
\ar[uurr]|-{\text{SR by G}} \ar[ddll]|-{\text{PR by } \mathfrak{g}}
\ar[ddrr]|-{\text{SR by }\mathfrak{g}} \ar[rr]|-{\text{SR by
}G\circledS\mathfrak{g}} \ar[ll]|-{\text{PR by }G\circledS\mathfrak{g}} &&
\mathcal{O}_{(\mu,\nu)} \ar[dd]^{\txt{symplectic \\ embedding}}
\ar[uu]_{\txt{symplectic \\ embedding}}
\ar@/_{5pc}/@{.>}[llll]^{\txt{symplectic\\leaf}} \\\\
G\circledS(\mathfrak{g}_{2}^{\ast }\times\mathfrak{g}_{3}^{\ast
})&&&&G\circledS\mathfrak{g}_{2}^{\ast } \ar@{^{(}->}[llll]^{\txt{symplectic
leaf}} 
\\
&&\text{Reductions of $T^*TG$}
}  \label{T*TG}
\end{equation}
\section{Hamiltonian Dynamics on $T^*T^*G$}

\begin{proposition}
A Hamiltonian function $H$ on$\  T^{\ast}T^{\ast}G$ determines the
Hamilton's equations
\begin{equation*}
i_{X_{H}^{ T^{\ast}T^{\ast}G}}\Omega_{ T^{\ast}T^{\ast}G}=-dH
\end{equation*}
by uniquely defining Hamiltonian vector field $X_{H}^{T^{\ast}T^{\ast}G}$%
. The Hamiltonian vector field is a right invariant vector field generated
by a 4-tuple Lie algebra element
\begin{equation*}
\left( \frac{\delta H}{\delta\nu},\frac{\delta H}{\delta\xi},ad_{\frac{%
\delta H}{\delta\mu}}^{\ast}\mu-T_{e}^{\ast}R_{g}\left( \frac{\delta H}{%
\delta g}\right) ,-\frac{\delta H}{\delta\mu}\right)
\end{equation*}
in $\left( \mathfrak{g}\circledS\mathfrak{g}^{\ast}\right) \circledS\left(
\mathfrak{g}^{\ast}\times\mathfrak{g}\right) $. At the point $\left(
g,\mu,\nu,\xi\right) ,$ the Hamilton's equations are%
\begin{align}
\frac{dg}{dt} & =T_{e}R_{g}\left( \frac{\delta H}{\delta\nu}\right) ,\text{
\ \ }  \label{HamT*T*G1} \\
\frac{d\mu}{dt} & =\frac{\delta H}{\delta\xi}+\ad_{\frac{\delta H}{\delta\nu }%
}^{\ast}\mu \\
\frac{d\nu}{dt} & =\ad_{\frac{\delta H}{\delta\mu}}^{\ast}\mu+\ad_{\frac {%
\delta H}{\delta\nu}}^{\ast}\nu-T_{e}^{\ast}R_{g}\left( \frac{\delta H}{%
\delta g}\right) -\ad_{\xi}^{\ast}\frac{\delta H}{\delta\xi} \\
\frac{d\xi}{dt} & =-\frac{\delta H}{\delta\mu}+[\frac{\delta H}{\delta\nu%
},\xi].  \label{HamT*T*G}
\end{align}
\end{proposition}

\subsection{Reduction of $T^*T^*G$ by $G$}

It follows from Eq.(\ref{Gr1T*T*G}) that the right action of $G$ on $T^{\ast}T^{\ast}G$ is%
\begin{equation}
\left(\left( h,\nu,\lambda_{2},\xi_{2}\right) ; g\right) \rightarrow \left(
hg,\nu,\lambda_{2},\xi_{2}%
\right)  \label{GonT*T*G}
\end{equation}
with the infinitesimal generator $X_{\left( \eta,0,0,0\right) }^{ T^{\ast}T^{\ast}G}$ being a right invariant vector field as in Eq.(\ref%
{RiT*T*G}) generated by $\left( \eta,0,0,0\right) $ for $\eta \in\mathfrak{g}
$.

\begin{proposition}
Poisson reduction of $ T^{\ast}T^{\ast}G$ under the action of $G$ results
in $\mathfrak{g}_{1}^{\ast}\circledS\left( \mathfrak{g}_{2}^{\ast}\times%
\mathfrak{g}_{3}\right) $ endowed with the Poisson bracket%
\begin{equation}\label{Poig*g*g}
\begin{split}
& \left\{ H,K\right\} _{\mathfrak{g}_{1}^{\ast}\circledS\left( \mathfrak{g}%
_{2}^{\ast}\times\mathfrak{g}_{3}\right) }\left( \mu,\nu ,\xi\right)
=\left\langle \frac{\delta H}{\delta\mu},\frac{\delta K}{%
\delta\xi}\right\rangle -\left\langle \frac{\delta K}{\delta\mu},\frac{\delta H}{\delta\xi}%
\right\rangle +\left\langle \nu,\left[ \frac{\delta H}{\delta \nu},%
\frac{\delta K}{\delta\nu}\right] \right\rangle   \\
& +\left\langle \xi,\ad_{\frac{\delta H}{\delta\nu}}^{\ast}\frac{\delta K}{%
\delta\xi}-\ad_{\frac{\delta K}{\delta\nu}}^{\ast}\frac{\delta H}{\delta\xi }%
\right\rangle +\left\langle \mu,\left[ \frac{\delta H}{\delta\mu},\frac{%
\delta K}{\delta\nu}\right] -\left[ \frac{\delta K}{\delta\mu},\frac{\delta H%
}{\delta\nu}\right] \right\rangle ,  
\end{split}
\end{equation}
and symplectic reduction gives $\mathcal{O}_{\mu}\times\mathfrak{g}\times%
\mathfrak{g}^{\ast}$ with the symplectic two-form defined by
\begin{equation}
\Omega_{ T^{\ast}T^{\ast}G}^{/G }\left( \left(
\eta_{\mathfrak{g}^{\ast}}\left( \mu\right) ,\lambda,\zeta\right) ,\left(
\bar{\eta}_{\mathfrak{g}^{\ast}}\left( \mu\right) ,\bar{\lambda},\bar{\zeta }%
\right) \right) =\left\langle \zeta,\bar{\lambda}\right\rangle -\left\langle
\bar{\zeta},\lambda\right\rangle -\left\langle \mu,[\eta ,\bar{\eta}%
]\right\rangle  \label{RedOhmT*T*G}
\end{equation}
on two elements $\left( \eta_{\mathfrak{g}^{\ast}}\left( \mu\right)
,\lambda,\zeta\right) $ and $\left( \bar{\eta}_{\mathfrak{g}^{\ast}}\left(
\mu\right) ,\bar{\lambda},\bar{\zeta}\right) $ of $T_{\mu}\mathcal{O}_{\mu
}\times\mathfrak{g}\times\mathfrak{g}^{\ast}$.
\end{proposition}

Recall that, in previous section, the Poisson and symplectic reductions of $ T^{\ast}TG$ result in reduced spaces $\mathfrak{g}\circledS\left(
\mathfrak{g}^{\ast}\times\mathfrak{g}^{\ast}\right) $ and $\mathcal{O}_{\mu
}\times\mathfrak{g}\times\mathfrak{g}^{\ast}$, respectively. The reduced
Poisson bracket on $%
\mathfrak{g}\circledS\left( \mathfrak{g}^{\ast}\times\mathfrak{g}^{\ast
}\right) $ is given by Eq.(\ref{Poigxg*xg*}) and the reduced symplectic
two-form $\Omega_{ T^{\ast}TG}^{/G}$ on $%
\mathcal{O}_{\mu}\times\mathfrak{g}\times\mathfrak{g}$ is in Eq.(\ref%
{RedOhmT*TG}). We have the following proposition from \cite{esen2014tulczyjew} relating the
reductions of cotangent bundles $ T^{\ast}T^{\ast}G$ and $T^{\ast }TG$. We refer to \cite{kupershmidt1983canonical} for a detailed study on the canonical maps between
semidirect products.

\subsection{Reduction of $T^*T^*G$ by $\mathfrak{g}^{\ast}$}

The action of $\mathfrak{g}^{\ast}$ on $ T^{\ast}T^{\ast}G$, given by
\begin{equation}
\left( \left( g,\mu,\nu,\xi\right);\lambda \right) \rightarrow\left(
g,\mu+\Ad^\ast_g\lambda,\nu,\xi\right)  \label{g*onT*T*G}
\end{equation}
generated by $X_{\left( 0,\lambda ,0,0\right) }^{T^{\ast}T^{\ast}G}=\left( 0,\lambda ,-\ad_{\xi}^{\ast
}\lambda ,0\right) $. As the action of $G\circledS
\mathfrak{g}^{\ast}$ on its cotangent bundle $ T^{\ast}T^{\ast}G$ is
symplectic, and $\mathfrak{g}^{\ast}$ is a subgroup. The action in Eq\eqref{g*onT*T*G} is symplectic hence we can perform a
Poisson and a symplectic reductions of $T^{\ast}T^{\ast}G$.

\begin{proposition}
The Poisson reduction of $ T^{\ast}T^{\ast}G$ with the action of $%
\mathfrak{g}^{\ast}$ results in $G\circledS\left( \mathfrak{g}_{2}^{\ast
}\times\mathfrak{g}_{3}\right) $ endowed with the bracket
\begin{equation}\label{PoGgg*}
\begin{split}
& \left\{ H,K\right\} _{G\circledS\left( \mathfrak{g}_{2}^{\ast}\times%
\mathfrak{g}_{3}\right) }\left( g,\nu,\xi\right) =\left\langle T_{e}^{\ast}R_{g}\frac{\delta H}{\delta g},\frac{%
\delta K}{\delta\nu}\right\rangle -\left\langle
T_{e}^{\ast}R_{g}\frac{\delta K}{\delta g},\frac{\delta H}{\delta\nu }%
\right\rangle +  \\
& \qquad \qquad \left\langle \xi ,\ad^\ast_{\frac{\delta H}{\delta\nu}} \frac{\delta K}{\delta\xi} - \ad^\ast_{\frac{\delta K}{\delta\nu}}\frac{\delta H}{\delta\xi}\right\rangle +\left\langle \nu,\left[
\frac{\delta H}{\delta\nu},\frac{\delta K}{\delta\nu}\right] \right\rangle .
\end{split}
\end{equation}
The application of Marsden-Weinstein symplectic reduction with the action of
$\mathfrak{g}^{\ast}$ on $ T^{\ast}T^{\ast}G$ having the momentum mapping
\begin{equation*}
\mathbf{J}_{ T^{\ast}T^{\ast}G}^{\mathfrak{g}^{\ast}}: T^{\ast}T^{\ast}G\rightarrow\mathfrak{g}_{2}:\left( g,\mu,\nu,\xi\right)
\rightarrow\xi
\end{equation*}
results in the reduced symplectic space $\left( \mathbf{J}_{ T^{\ast}T^{\ast}G}^{\mathfrak{g}^{\ast}}\right) ^{-1}\left( \xi\right) /%
\mathfrak{g}^{\ast}$ isomorphic to $G\circledS\mathfrak{g}_{3}^{\ast}$ with
the canonical symplectic two-form $\Omega_{G\circledS\mathfrak{g}%
_{2}^{\ast}} $ in Eq.(\ref{Ohm2T*G}).
\end{proposition}

\subsection{Reduction of $T^*T^*G$ by $G\circledS\mathfrak{g}^{\ast}$}

The Lie algebra of the group $G\circledS\mathfrak{g}_{1}^{\ast}$ is the space $%
\mathfrak{g}\circledS\mathfrak{g}^{\ast}$ carrying the bracket%
\begin{equation}
\left[ \left( \xi,\mu\right) ,\left( \eta,\nu\right) \right] _{\mathfrak{g}%
\circledS\mathfrak{g}^{\ast}}=\left( \left[ \xi,\eta\right]
,\ad_{\xi}^{\ast}\nu - \ad_{\eta}^{\ast}\mu\right) .  \label{LBgg*}
\end{equation}
The dual space $\mathfrak{g}_{2}^{\ast}\times\mathfrak{g}_{3}$ carries the
Lie-Poisson bracket
\begin{equation}\label{LPE}
\left\{ F,E\right\} _{\mathfrak{g}_{2}^{\ast}\times\mathfrak{g}_{3}}\left(
\nu,\xi\right)  =\left\langle \nu,\left[ \frac{\delta F}{\delta\nu},\frac{\delta E}{%
\delta\nu}\right] \right\rangle +\left\langle \xi ,\ad^\ast_{\frac{\delta F}{\delta\nu}} \frac{\delta E}{\delta\xi} - \ad^\ast_{\frac{\delta F}{\delta\nu}}\frac{\delta E}{\delta\xi}\right\rangle ,  
\end{equation}
that follows from the Lie algebra bracket in Eq.(\ref{LBgg*}).

\begin{proposition}
The Lie-Poisson bracket, in Eq.(\ref{LPE}), on $\mathfrak{g}_{2}^{\ast}\times%
\mathfrak{g}_{3}$ defines the Hamiltonian vector field $X_{E}^{\mathfrak{g}%
^{\ast}\times\mathfrak{g}}$ by
\begin{equation*}
\left\{ F,E\right\} _{\mathfrak{g}_{2}^{\ast}\times\mathfrak{g}%
_{3}}=-\left\langle dF,X_{E}^{\mathfrak{g}_{2}^{\ast}\times\mathfrak{g}%
_{3}}\right\rangle
\end{equation*}
whose components are the Lie-Poisson equations
\begin{equation}
\frac{d\nu}{dt}=ad_{\frac{\delta H}{\delta\nu}}^{\ast}\nu-ad_{\xi}^{\ast}%
\frac{\delta H}{\delta\xi},\ \ \frac{d\xi}{dt}=[\frac{\delta H}{%
\delta\nu },\xi].  \label{LPg*g}
\end{equation}
\end{proposition}

Although, these equations result from Eq.(\ref{LPE}), it is
possible to obtain them starting from
the Hamilton's equations (\ref{HamT*T*G1})-(\ref{HamT*T*G}) on $ T^{\ast
}T^{\ast}G$ and applying a Poisson reduction with the action
of $G\circledS\mathfrak{g}^{\ast}$ given by
\begin{align}
 T^{\ast}T^{\ast}G \times \left( G\circledS\mathfrak{g}^{\ast}\right)\rightarrow T^{\ast}T^{\ast}G & :((h,\nu,\lambda,\xi );(g,\mu))  \notag \\
& \mapsto(hg,\nu+\Ad_{h}^{\ast}\mu,\lambda,\xi ) .
\label{T*GonT*T*G}
\end{align}
In other words, choosing the Hamiltonian function $H$ in Eqs.(\ref%
{HamT*T*G1})-(\ref{HamT*T*G}) depending on fiber variables only, that is $%
H=H\left( \nu,\xi\right) $, Eq.(\ref{LPg*g}) follows.

To reduce the Hamilton's equations (\ref{HamT*T*G1})-(\ref{HamT*T*G}) on $ T^{\ast}T^{\ast}G$ symplectically, we first compute the momentum mapping
\begin{equation*}
\mathbf{J}_{G\circledS\mathfrak{g}_{3}^{\ast}}^{G_{\xi}}:T^{\ast
}T^{\ast}G\rightarrow\mathfrak{g}^{\ast}\times\mathfrak{g}:\left( g,\mu
,\nu,\xi\right) \rightarrow\left( \nu,\xi\right) ,
\end{equation*}
associated with the action of $G\circledS\mathfrak{g}^{\ast}$ in Eq.(\ref%
{T*GonT*T*G}) and the quotient space%
\begin{equation}
\left. \left( \mathbf{J}_{G\circledS\mathfrak{g}_{3}^{\ast}}^{G_{\xi}}%
\right) ^{-1}\left( \nu,\xi\right) \right/ G_{\left( \nu,\xi\right) }\simeq%
\mathcal{O}_{\left( \nu,\xi\right) }.  \label{Onuxi}
\end{equation}
Here, $G_{\left( \nu,\xi\right) }$ is the isotropy subgroup of $G\circledS%
\mathfrak{g}^{\ast}$ consisting of elements preserved under the coadjoint
action $G\circledS\mathfrak{g}^{\ast}$\ on the dual space $\mathfrak{g}%
^{\ast}\times\mathfrak{g}$ of its Lie algebra
\begin{align}
\Ad^{\ast} & :\left( G\circledS\mathfrak{g}^{\ast}\right) \times\left(
\mathfrak{g}^{\ast}\times\mathfrak{g}\right) \rightarrow\mathfrak{g}^{\ast
}\times\mathfrak{g}  \notag \\
& :\left( \left( g,\mu\right) ,\left( \nu,\xi\right) \right)
\rightarrow\left( \Ad_{g}^{\ast}\nu-\ad_{\Ad_g\xi}^{\ast}\mu
,\Ad_{g}\xi\right)  \label{Coad}
\end{align}
and, the space $\mathcal{O}_{\left( \nu,\xi\right) }$ is the coadjoint orbit
passing through the point $\left( \nu,\xi\right) $ under this coadjoint
action. 

\begin{proposition}
The symplectic reduction of $ T^{\ast}T^{\ast}G$ results in the coadjoint
orbit $\mathcal{O}_{\left( \nu,\xi\right) }$ in $\mathfrak{g}_{2}^{\ast
}\times\mathfrak{g}$ through the point $\left( \nu,\xi\right) $. The
reduced symplectic two-form $\Omega_{ T^{\ast}T^{\ast}G}^{/(G\circledS
\mathfrak{g}^{\ast})}$ (denoted simply by $\Omega_{\mathcal{O}%
_{\left( \nu,\xi\right) }}$) takes the value%
\begin{equation}
\left\langle \Omega_{\mathcal{O}_{\left( \nu,\xi\right) }};\left(
\lambda,\eta\right) ,\left( \bar{\lambda},\bar{\eta}\right) \right\rangle
\left( \nu,\xi\right) =\left\langle \nu,[\bar{\eta},\eta]\right\rangle
+\left\langle \xi,ad_{\eta}^{\ast}\bar{\lambda}-ad_{\bar{\eta}%
}^{\ast}\lambda]\right\rangle  \label{SymOr}
\end{equation}
on two vectors $\left( \lambda,\eta\right) $ and $\left( \bar{\lambda},\bar{%
\eta}\right) $ in $T_{\left( \nu,\xi\right) }\mathcal{O}_{\left(
\nu,\xi\right) }$.
\end{proposition}

Alternatively, this reduction can be performed 
in two steps by applying the Hamiltonian reduction by stages theorem \cite%
{MaMiOrPeRa07}. The first step consists of the symplectic reduction of $%
 T^{\ast}T^{\ast}G$ with the action of $\mathfrak{g}^{\ast}$ which has already been established in previous subsection and
resulted in the reduced symplectic space $\left( \mathbf{J}_{\text{ }%
 T^{\ast}T^{\ast}G}^{\mathfrak{g}^{\ast}}\right) ^{-1}\left( \xi\right) /%
\mathfrak{g}^{\ast}$, isomorphic to $G\circledS\mathfrak{g}_{3}^{\ast}$,
with the canonical symplectic two-from $\Omega_{G\circledS\mathfrak{g}%
_{3}^{\ast}}$ in Eq.(\ref{Ohm2T*G}). For the second step, we recall the
adjoint group action $\Ad_{g^{-1}}$ of $G$ on $\mathfrak{g}$ and define the
isotropy subgroup
\begin{equation}
G_{\xi}=\left\{ g\in G:\Ad_{g^{-1}}\xi=\xi\right\}  \label{Gxi}
\end{equation}
for an element$\ \xi\in\mathfrak{g}$ under the adjoint action. The Lie algebra $%
\mathfrak{g}_{\xi}$ of $G_{\xi}$ consists of vectors $\eta\in\mathfrak{g}$
satisfying $\left[ \eta,\xi\right] =0$. The isotropy subgroup $G_{\xi}$ acts
on $G\circledS\mathfrak{g}_{3}^{\ast}$ by the same way as described in Eq.(%
\ref{GonGxg*}). This action is Hamiltonian and has the momentum mapping
\begin{equation*}
\mathbf{J}_{G\circledS\mathfrak{g}_{3}^{\ast}}^{G_{\xi}}:G\circledS
\mathfrak{g}_{3}^{\ast}\rightarrow\mathfrak{g}_{\xi}^{\ast},
\end{equation*}
where $\mathfrak{g}_{\xi}^{\ast}$ is the dual space of $\mathfrak{g}_{\xi} $%
. The quotient space%
\begin{equation*}
\left. \left( \mathbf{J}_{G\circledS\mathfrak{g}_{3}^{\ast}}^{G_{\xi}}%
\right) ^{-1}\left( \nu\right) \right/ G_{\xi,\nu}\simeq\mathcal{O}_{\left(
\nu,\xi\right) }
\end{equation*}
is diffeomorphic to the coadjoint orbit $\mathcal{O}_{\left( \nu,\xi\right)
} $ in Eq.(\ref{Onuxi}). 
\begin{equation} \label{T*T*G}
\xymatrix{\mathfrak{g}_{1}^{\ast }\circledS(\mathfrak{g}_{2}^{\ast
}\times\mathfrak{g}_{3}) &&&& \mathcal{O}_{\mu }\times
\mathfrak{g}_{1}^{\ast }\times\mathfrak{g}_{3}
\ar@{_{(}->}[llll]_{\txt{\small symplectic leaf}} \\\\ \mathfrak{g}_{2}^{\ast
}\times\mathfrak{g}_{3} \ar[dd]_{\txt{\small Poisson\\ \small embedding}}
\ar[uu]^{\txt{\small Poisson\\ \small embedding}} && (G\circledS\mathfrak{g}_{1}^{\ast
})\circledS(\mathfrak{g}_{2}^{\ast }\times\mathfrak{g}_{3})
\ar[uull]|-{\text{PR by G}} \ar[uurr]|-{\text{SR by G}}
\ar[ddll]|-{\text{PR by } \mathfrak{g}^{\ast }_{1}} \ar[ddrr]|-{\text{SR
by }\mathfrak{g}_{1}^{\ast }} \ar[rr]|-{\text{SR by
}G\circledS\mathfrak{g}^{\ast }} \ar[ll]|-{\text{PR by
}G\circledS\mathfrak{g}_{1}^{\ast }} && \mathcal{O}_{(\mu,\xi)}
\ar[dd]^{\txt{\small symplectic\\\small embedding}} \ar[uu]_{\txt{\small symplectic\\ \small embedding}}
\ar@/_{5pc}/@{.>}[llll]^{\txt{symplectic\\leaf}} \\\\
G\circledS(\mathfrak{g}_{2}^{\ast }\times\mathfrak{g}_{3}) &&&&
G\circledS\mathfrak{g}_{2}^{\ast } \ar@{^{(}->}[llll]^{\txt{\small symplectic
leaf}} 
\\
&&\txt {\small
Reduction of \\ \small $T^*T^*G=(G\circledS\mathfrak{g}_{1}^{\ast
})\circledS(\mathfrak{g}_{2}^{\ast }\times\mathfrak{g}_{3})$}
} 
\end{equation}

\section{Hamiltonian and Lagrangian Dynamics on Tulczyjew Symplectic Space $TT^*G$}

\subsection{Hamiltonian Dynamics on $TT^*G$}

\begin{proposition}
Given a Hamiltonian function $E$ on $ TT^{\ast}G$, the Hamilton's equation
\begin{equation*}
i_{X_{E}^{ TT^{\ast}G}}\Omega_{ TT^{\ast}G}=-dE
\end{equation*}
defines a Hamiltonian right invariant vector field $X_{E}^{ TT^{\ast}G}$ generated by the element
\begin{equation*}
\left( \frac{\delta E}{\delta\nu},-\left( \frac{\delta E}{\delta\xi }+\ad_{%
\frac{\delta E}{\delta\nu}}^{\ast}\mu\right) ,-\frac{\delta E}{\delta \mu}%
+\ad_{\xi}\frac{\delta E}{\delta\nu},-\left( T^{\ast}R_{g}\frac{\delta E}{%
\delta g}+\ad_{\xi}^{\ast}\frac{\delta E}{\delta\xi}+\ad_{\xi}^{\ast }\ad_{%
\frac{\delta E}{\delta\nu}}^{\ast}\mu\right) \right)
\end{equation*}
of the Lie algebra $\left( \mathfrak{g}\circledS\mathfrak{g}^{\ast}\right)
\circledS\left( \mathfrak{g}\circledS\mathfrak{g}^{\ast}\right) .$
Components of $X_{E}^{ TT^{\ast}G}$ define the Hamilton's equations
\begin{equation}
\dot{g}=TR_{g}\left( \frac{\delta E}{\delta\nu}\right) ,\text{ \ \ }\dot {\mu%
}=-\frac{\delta E}{\delta\xi},\text{ \ \ }\dot{\xi}=\frac{\delta E}{\delta\mu%
},\text{ \ \ }\dot{\nu}=\ad_{\frac{\delta E}{\delta\nu}}^{\ast}\nu-T^{%
\ast}R_{g}\left( \frac{\delta E}{\delta g}\right)  \label{HamTT*G}
\end{equation}
\end{proposition}
in the adapted trivialization of $TT^*G$. 

\subsubsection{Reduction of $TT^*G$ by $G$}

\begin{proposition}
The Poisson reduction of $ TT^{\ast}G$ under the action of $G$ results in the total space $\mathfrak{g}_{1}^{\ast}\circledS\left( 
\mathfrak{g}_{2} \circledS\mathfrak{g}_{3}^{\ast}\right)$ endowed with the
Poisson bracket%
\begin{equation}
\left\{ E,F\right\} _{\mathfrak{g}_{1}^{\ast}\circledS\left( 
\mathfrak{g}_{2} \circledS\mathfrak{g}_{3}^{\ast}\right)}\left( \mu,\xi,\nu\right)
=\left\langle \frac{\delta F}{\delta\xi},\frac{\delta E}{\delta\mu }%
\right\rangle -\left\langle \frac{\delta E}{\delta\xi},\frac{\delta F}{%
\delta\mu}\right\rangle +\left\langle \nu,\left[ \frac{\delta E}{\delta\nu },%
\frac{\delta F}{\delta\nu}\right] \right\rangle .  \label{Poig*gg*}
\end{equation}
\end{proposition}

\begin{remark}
Here, the Poisson bracket on $\mathfrak{g}_{1}^{\ast}\circledS\left( 
\mathfrak{g}_{2} \circledS\mathfrak{g}_{3}^{\ast}\right)$ is the direct product of
canonical Poisson bracket on $\mathfrak{g}_{1}^{\ast}\times\mathfrak{g}_{2}$
and Lie-Poisson bracket on $\mathfrak{g}_{3}^{\ast}$ whereas in Eq.(\ref%
{Poigxg*xg*}) we obtained a Poisson bracket, on the isomorphic space $%
\mathfrak{g}\circledS\left( \mathfrak{g}^{\ast}\times\mathfrak{g}%
^{\ast}\right) $, which is not in the form of a direct product.
\end{remark}

The action of $G$ is Hamiltonian with the momentum mapping
\begin{equation}
\mathbf{J}_{TT^{\ast}G}^{G}:TT^{\ast}G\rightarrow
\mathfrak{g}^{\ast}:\left( g,\mu,\xi,\nu\right) \rightarrow\nu+ad_{\xi
}^{\ast}\mu.  \label{MGonTT*G}
\end{equation}
The quotient space of the preimage $\mathbf{J}_{TT^{\ast}G}^{-1}%
\left( \lambda\right) $ of an element $\lambda\in\mathfrak{g}^{\ast}$ under
the action of isotropy subgroup $G_{\lambda}$ is
\begin{equation*}
\left. \mathbf{J}_{TT^{\ast}G}^{-1}\left( \lambda\right) \right/
G_{\lambda}\simeq\mathcal{O}_{\lambda}\times\mathfrak{g}^{\ast}\times%
\mathfrak{g}\text{. }
\end{equation*}
Pushing forward a right invariant vector field $X_{\left( \eta,\upsilon
,\zeta,\tilde{\upsilon}\right) }^{TT^{\ast}G}$ in the form of Eq.(\ref%
{RITT*G}) by the symplectic projection $ TT^{\ast}G\rightarrow \mathcal{O}%
_{\lambda}\times\mathfrak{g}^{\ast}\times\mathfrak{g}$, we obtain the vector
field
\begin{equation}
X_{\left( \eta,\upsilon,\zeta\right) }^{\mathcal{O}_{\lambda}\times
\mathfrak{g}^{\ast}\times\mathfrak{g}}\left( Ad_{g^{-1}}^{\ast}\lambda
,\mu,\xi\right) =\left( ad_{\eta}^{\ast}\circ
Ad_{g^{-1}}^{\ast}\lambda,\upsilon+ad_{\eta}^{\ast}\mu,\zeta+\left[ \xi,\eta%
\right] \right)  \label{VfO}
\end{equation}
on the quotient space $\mathcal{O}_{\lambda}\times\mathfrak{g}^{\ast}\times%
\mathfrak{g}$. We refer to \cite{esen2015tulczyjew} for the proof of the following
proposition.

\begin{proposition}
The reduced Tulczyjew's space $\mathcal{O}_{\lambda}\times\mathfrak{g}^{\ast
}\times\mathfrak{g}$ has an exact symplectic two-form $\Omega_{\mathcal{O}%
_{\lambda}\times\mathfrak{g}^{\ast}\times\mathfrak{g}}$ with two potential
one-forms $\chi_{1}$ and $\chi_{2}$ whose values on vector fields of the
form of Eq.(\ref{RITT*G}) at the point $\left( Ad_{g^{-1}}^{\ast}\lambda,\mu
,\xi\right) $ are
\begin{align}
\left\langle \Omega_{\mathcal{O}_{\lambda}\times\mathfrak{g}^{\ast}\times%
\mathfrak{g}},\left( X_{\left( \eta,\upsilon,\zeta\right) }^{\mathcal{O}%
_{\lambda}\times\mathfrak{g}^{\ast}\times\mathfrak{g}},X_{\left( \bar{\eta},%
\bar{\upsilon},\bar{\zeta}\right) }^{\mathcal{O}_{\lambda}\times\mathfrak{g}%
^{\ast}\times\mathfrak{g}}\right) \right\rangle & =\left\langle \upsilon,%
\bar{\zeta}\right\rangle -\left\langle \bar {\upsilon},\zeta\right\rangle
-\left\langle \lambda,[\eta,\bar{\eta }]\right\rangle , \\
\left\langle \chi_{1},X_{\left( \eta,\upsilon,\zeta\right) }^{\mathcal{O}%
_{\lambda}\times\mathfrak{g}^{\ast}\times\mathfrak{g}}\right\rangle \left(
Ad_{g^{-1}}^{\ast}\lambda,\mu,\xi\right) & =\left\langle \lambda
,\eta\right\rangle -\left\langle \upsilon,\xi\right\rangle , \\
\left\langle \chi_{2},X_{\left( \eta,\upsilon,\zeta\right) }^{\mathcal{O}%
_{\lambda}\times\mathfrak{g}^{\ast}\times\mathfrak{g}}\right\rangle \left(
Ad_{g^{-1}}^{\ast}\lambda,\mu,\xi\right) & =\left\langle \lambda
,\eta\right\rangle +\left\langle \mu,\zeta\right\rangle .
\end{align}
\end{proposition}
The potential one-forms $\theta_{1}$ and $\theta_{2}$ of Eq.(\ref{1}) and Eq.(\ref{2}) for Tulczyjew symplectic structure on $TT^*G$ and the one-forms $\chi_{1}$ and $\chi_{2} $ of reduced Tulczyjew space are related by the equations
\begin{align*}
\left\langle \theta_{1},X_{\left( \eta,\upsilon,\zeta,\tilde{\upsilon }%
\right) }^{ TT^{\ast}G}\right\rangle \left( g,\mu,\xi,\nu\right) &
=\left\langle \chi_{1},X_{\left( \eta,\upsilon,\zeta\right) }^{\mathcal{O}%
_{\lambda}\times\mathfrak{g}^{\ast}\times\mathfrak{g}}\right\rangle \left(
Ad_{g^{-1}}^{\ast}\lambda,\mu,\xi\right) , \\
\left\langle \theta_{2},X_{\left( \eta,\upsilon,\zeta,\tilde{\upsilon }%
\right) }^{ TT^{\ast}G}\right\rangle \left( g,\mu,\xi,\nu\right) &
=\left\langle \chi_{2},X_{\left( \eta,\upsilon,\zeta\right) }^{\mathcal{O}%
_{\lambda}\times\mathfrak{g}^{\ast}\times\mathfrak{g}}\right\rangle \left(
Ad_{g^{-1}}^{\ast}\lambda,\mu,\xi\right).
\end{align*}

\subsubsection{Reduction of $TT^*G$ by $\mathfrak{g}$}

\begin{proposition}
The action of $\mathfrak{g}_{2}$ on $ TT^{\ast}G$ is given, for $\eta\in\mathfrak{g}_{2}$, by
\begin{equation}
\varphi_{\eta}:\  TT^{\ast}G\rightarrow\  TT^{\ast}G:\left(
\left( g,\mu,\xi,\nu\right);\eta \right) \rightarrow\left( g,\mu,\xi
+\eta,\nu\right)   \label{gonTT*G}
\end{equation}
and it is symplectic.
\end{proposition}

\begin{proof}
Push forward of a vector field $X_{\left( \xi_{2},\nu_{2},\xi_{3},\nu
_{3}\right) }^{TT^{\ast}G}$ in the form of Eq.(\ref{RITT*G}) by the
transformation $\varphi_{\eta}$ is also a right invariant vector field
\begin{equation*}
\left( \varphi_{\eta}\right) _{\ast}X_{\left(
\xi_{2},\nu_{2},\xi_{3},\nu_{3}\right) }^{TT^{\ast}G}=X_{\left(
\xi_{2},\nu_{2},\xi_{3}-[\eta,\xi_{2}],\nu_{3}+ad_{\eta}^{\ast}\nu_{2}%
\right) }^{TT^{\ast}G}.
\end{equation*}
By direct calculation, one establishes the identity
\begin{equation}
\varphi_{\eta}^{\ast}\Omega_{ TT^{\ast}G}\left( X,Y\right) \left(
g,\mu,\xi,\nu\right) =\Omega_{ TT^{\ast}G}\left( \left( \varphi_{\eta
}\right) _{\ast}X,\left( \varphi_{\eta}\right) _{\ast}Y\right) \left(
g,\mu,\xi+\eta,\nu\right)  \label{gonTT*Gsym}
\end{equation}
which gives the desired result. In Eq.(\ref{gonTT*Gsym}) $X$ and $Y$ are
right invariant vector fields as in Eq.(\ref{RITT*G}) and $\Omega_{\
 TT^{\ast}G}$ is the symplectic two-form given in Eq.(\ref{SymTT*G}).
\end{proof}

\begin{proposition}
The Poisson reduction of $ TT^{\ast}G$ under the action in Eq.(\ref{gonTT*G}) of $\mathfrak{g}%
_{2}$  results in $\left( G\circledS\mathfrak{g}%
_{1}^{\ast}\right)\circledS\mathfrak{g}_{3}^{\ast} $ endowed with the bracket
\begin{equation*}
\left\{ E,F\right\} _{\left( G\circledS\mathfrak{g}_{1}^{\ast}\right) \circledS%
\mathfrak{g}_{3}^{\ast}}\left( g,\mu,\nu\right) =\left\langle T_{e}^{\ast}R_{g}\frac{\delta E}{\delta g},\frac{%
\delta F}{\delta\nu}\right\rangle -\left\langle
T_{e}^{\ast}R_{g}\frac{\delta F}{\delta g},\frac{\delta E}{\delta\nu }%
\right\rangle +\left\langle \nu,\left[ \frac{\delta E}{%
\delta\nu},\frac{\delta F}{\delta\nu}\right] \right\rangle .
\end{equation*}
\end{proposition}

\begin{remark}
The Poisson bracket $\left\{ E,F\right\} _{\left( G\circledS\mathfrak{g}%
_{1}^{\ast}\right)\circledS\mathfrak{g}_{3}^{\ast} }$ is independent  of functions with respect to $\mu$, that is, it does not involve
$\delta E/\delta\mu$ and $\delta F/\delta\mu$. Its structure resembles the
canonical Poisson bracket in Eq.(\ref{PoissonGg*}) on $G\circledS\mathfrak{g}%
^{\ast}$. We recall that, on $\left( G\circledS\mathfrak{g}^{\ast}\right)\circledS%
\mathfrak{g}^{\ast} $ there is another Poisson bracket given in Eq.(\ref{PoGxg*xg*}) that involves $\delta
E/\delta\mu,$ $\delta F/\delta\mu$, $\delta E/\delta\nu$ and $\delta
F/\delta\nu$. This latter comes from reduction of $T^*TG$ by $\G{g}$.
\end{remark}

The infinitesimal generator $X_{\left( 0,0,\xi_{3},0\right) }^{
 TT^{\ast}G}$ of the action in Eq.(\ref{gonTT*G}) corresponds to the
element $\xi_{3}\in\mathfrak{g}$ and is a right invariant vector field.
Since the action is Hamiltonian, and the symplectic two-form is exact, we
can derive the associated momentum map $\mathbf{J}_{TT^{\ast}G}^{%
\mathfrak{g}_{2}}$ from the equation
\begin{equation*}
\left\langle \mathbf{J}_{TT^{\ast}G}^{\mathfrak{g}_{2}}\left(
g,\mu,\xi,\nu\right) ,\xi_{3}\right\rangle =\left\langle \theta
_{2},X_{\left( 0,0,\xi_{3},0\right) }^{ TT^{\ast}G}\right\rangle
=\left\langle \mu,\xi_{3}\right\rangle ,
\end{equation*}
where $\theta_{2}$ is the potential one-form of Tulczyjew in Eq.(\ref{2}) satisfying $%
d\theta_{2}=\Omega_{ TT^{\ast}G}$. We find that
\begin{equation}
\mathbf{J}_{TT^{\ast}G}^{\mathfrak{g}_{2}}:  TT^{\ast
}G\rightarrow Lie^{\ast}\left( \mathfrak{g}_{2}\right) =\mathfrak{g}^{\ast
}:\left( g,\mu,\xi,\nu\right) \rightarrow\mu  \label{MgonTT*G}
\end{equation}
is the projection to the second entry in $ TT^{\ast}G$. The preimage of
an element $\mu\in\mathfrak{g}^{\ast}$ by $\mathbf{J}_{\text{ }%
 TT^{\ast}G}^{\mathfrak{g}_{2}}$ is the space $G\circledS\left( \mathfrak{%
g}_{2}\circledS\mathfrak{g}_{3}^{\ast}\right) $. Following proposition
describes the symplectic reduction of $ TT^{\ast}G$ with the action of $%
\mathfrak{g}_{2}$.

\begin{proposition}
The symplectic reduction of $ TT^{\ast}G$ under the action of $\mathfrak{g%
}_{2}$ given by Eq.(\ref{gonTT*G}) gives the reduced space
\begin{equation*}
\left. \left( \mathbf{J}_{TT^{\ast}G}^{\mathfrak{g}_{2}}\right)
^{-1}\left( \mu\right) \right/ \mathfrak{g}_{2}\simeq G\circledS \mathfrak{g}%
_{3}^{\ast}
\end{equation*}
with the canonical symplectic two-from $\Omega_{G\circledS\mathfrak{g}%
_{3}^{\ast}}$ as in Eq.(\ref{Ohm2T*G}).
\end{proposition}

\begin{remark}
Existence of the symplectic action of $\mathfrak{g}_{2}$ on $ TT^{\ast}G$
is directly related to the existence of  symplectic diffeomorphism
\begin{equation}
 \bar{\sigma}_{G}:TT^{\ast}G\rightarrow T^{\ast
}TG:\left( g,\mu,\xi,\nu\right) \rightarrow\left( g,\xi,\nu+\ad_{\xi}^{\ast
}\mu,\mu\right)  \label{sig}
\end{equation}
in Tulczyjew triplet described in \cite{esen2014tulczyjew}.
\end{remark}

\subsubsection{Reduction of $TT^*G$ by $\mathfrak{g}^{\ast}$}

Induced from the group operation on $ TT^{\ast}G$, there are two
canonical actions of $\mathfrak{g}^{\ast}$ on $ TT^{\ast}G$
\begin{equation*}
\psi:\mathfrak{g}_{1}^{\ast}\times\  TT^{\ast}G\rightarrow\
 TT^{\ast}G, \qquad \phi:\mathfrak{g}_{3}^{\ast}\times\
 TT^{\ast}G\rightarrow   TT^{\ast}G
\end{equation*}
described by
\begin{align}
\psi_{\lambda}\left( g,\mu,\xi,\nu\right) & = \left( g,\mu+\lambda
,\xi,\nu\right) ,  \label{psi} \\
\phi_{\lambda}\left( g,\mu,\xi,\nu\right) & =\left( g,\mu,\xi,\nu
+\lambda\right) .  \label{phi}
\end{align}

\begin{proposition}
$\psi$ is a symplectic action whereas $\phi$ is not.
\end{proposition}

\begin{proof}
Pushing forward of a vector field $X_{\left( \xi_{2},\nu_{2},\xi_{3},\nu
_{3}\right) }^{ TT^{\ast}G}$ in the form of Eq.(\ref{RITT*G}) by
transformations $\psi_{\lambda}$ and $\phi_{\lambda}$ results in right
invariant vector fields
\begin{align*}
\left( \psi_{\lambda}\right) _{\ast}X_{\left(
\xi_{2},\nu_{2},\xi_{3},\nu_{3}\right) }^{   TT^{\ast}G} & =X_{\left(
\xi_{2},\nu_{2}-ad_{\xi_{2}}^{\ast}\lambda,\xi_{3},\nu_{3}-ad_{\xi}^{%
\ast}ad_{\xi_{2}}^{\ast }\lambda\right) }^{   TT^{\ast}G}, \\
\left( \phi_{\lambda}\right) _{\ast}X_{\left(
\xi_{2},\nu_{2},\xi_{3},\nu_{3}\right) }^{  TT^{\ast}G} & =X_{\left(
\xi_{2},\nu_{2},\xi _{3},\nu_{3}-ad_{\xi_{2}}^{\ast}\lambda\right) }^{ 
 TT^{\ast}G}.
\end{align*}
If $\Omega_{ TT^{\ast}G}$ is the symplectic two-form on $ TT^{\ast}G$
given in Eq.(\ref{SymTT*G}), direct calculations show that the identity
\begin{equation}
\psi_{\lambda}{}^{\ast}\Omega_{ TT^{\ast}G}\left( X,Y\right) \left(
g,\mu,\xi,\nu\right) =\Omega_{ TT^{\ast}G}\left( \left( \psi_{\lambda
}\right) _{\ast}X,\left( \psi_{\lambda}\right) _{\ast}Y\right) \left(
g,\mu+\lambda,\xi,\nu\right)
\end{equation}
holds for all vector fields $X$ and $Y$, and $\lambda\in\mathfrak{g}^{\ast}$
whereas
\begin{equation}
\phi_{\lambda}{}^{\ast}\Omega_{ TT^{\ast}G}\left( X,Y\right) \left(
g,\mu,\xi,\nu\right) =\Omega_{ TT^{\ast}G}\left( \left( \phi_{\lambda
}\right) _{\ast}X,\left( \phi_{\lambda}\right) _{\ast}Y\right) \left(
g,\mu,\xi,\nu+\lambda\right)
\end{equation}
does not necessarily hold. Hence, $\psi_{\lambda}$ is a symplectic action
but not $\phi_{\lambda}$.
\end{proof}

\begin{proposition}
Poisson reduction of $ TT^{\ast}G$ under the action $\psi$ of $\mathfrak{g%
}_{1}^{\ast}$ results in $G\circledS\left( \mathfrak{g}_{2}\circledS\mathfrak{g}%
_{3}^{\ast}\right) $ endowed with the bracket
\begin{equation}
\left\{ E,F\right\} _{G\circledS\left( \mathfrak{g}_{2}\times \mathfrak{g}%
_{3}^{\ast}\right) }\left( g,\xi,\nu\right) =\left\langle T_{e}^{\ast}R_{g}%
\frac{\delta F}{\delta g},\frac{\delta E}{\delta\nu }\right\rangle
-\left\langle T_{e}^{\ast}R_{g}\frac{\delta E}{\delta g},\frac{\delta F}{%
\delta\nu}\right\rangle +\left\langle \nu,\left[ \frac{\delta E}{\delta\nu},%
\frac{\delta F}{\delta\nu}\right] \right\rangle .
\end{equation}
\end{proposition}

\begin{remark}
The Poisson bracket $\left\{ E,F\right\} _{G\circledS\left( \mathfrak{g}%
_{2}\circledS\mathfrak{g}_{3}^{\ast}\right) }$ is independent of derivatives of
functions with respect to $\xi$ and it resembles to the canonical Poisson
bracket in Eq.(\ref{PoissonGg*}) on $G\circledS\mathfrak{g}^{\ast}$. On the other hand, the space $%
G\circledS\left( \mathfrak{g}^{\ast}\times\mathfrak{g}\right) $, which is
isomorphic to $G\circledS\left( \mathfrak{g}_{2}\circledS\mathfrak{g}_{3}^{\ast
}\right) $, has the Poisson bracket in Eq.(\ref{PoGgg*}) involving derivatives with respect to both of $\xi$ and
$\nu$. This latter is obtained from $T^*TG$ via reduction by $\G{g}^*$.
\end{remark}

The infinitesimal generator  $X_{\left( 0,\nu_{2},0,0\right) }^{
 TT^{\ast}G}$ of the action are defined by $\nu_{2}\in Lie\left(
\mathfrak{g}_{1}^{\ast}\right) $. We compute the associated momentum map
from the equation
\begin{equation*}
\left\langle \mathbf{J}_{TT^{\ast}G}^{\mathfrak{g}%
_{1}^{\ast}}\left( g,\mu,\xi,\nu\right) ,\nu_{2}\right\rangle =\left\langle
\theta _{1},X_{\left( 0,\nu_{2},0,0\right) }^{ TT^{\ast}G}\right\rangle
=-\left\langle \nu_{2},\xi\right\rangle ,
\end{equation*}
where $\theta_{1}$ is the Tulczyjew potential one-form in Eq.(\ref{1}). We find that
\begin{equation}
\mathbf{J}_{TT^{\ast}G}^{\mathfrak{g}_{1}^{\ast}}:  TT^{\ast
}G\rightarrow Lie^{\ast}\left( \mathfrak{g}_{1}^{\ast}\right) \simeq
\mathfrak{g}:\left( g,\mu,\xi,\nu\right) \rightarrow-\xi  \label{Mg*onTT*G}
\end{equation}
is minus the projection to third factor in $ TT^{\ast}G$. The preimage of
an element $\xi\in\mathfrak{g}$ is the space $G\circledS\left( \mathfrak{g}%
_{1}^{\ast}\times\mathfrak{g}_{3}^{\ast}\right) $.

\begin{proposition}
The symplectic reduction of $ TT^{\ast}G$ under the action of $\mathfrak{g%
}^{\ast}$ defined in Eq.(\ref{gonTT*G}) results in the reduced space
\begin{equation*}
\left. \left( \mathbf{J}_{TT^{\ast}G}^{\mathfrak{g}%
_{1}^{\ast}}\right) ^{-1}\left( \xi\right) \right/ \mathfrak{g}%
_{1}^{\ast}\simeq\left. G\circledS\left( \mathfrak{g}_{1}^{\ast}\times%
\mathfrak{g}_{3}^{\ast}\right) \right/ \mathfrak{g}_{1}^{\ast}\simeq
G\circledS \mathfrak{g}_{3}^{\ast}
\end{equation*}
with the canonical symplectic two-from $\Omega_{G\circledS\mathfrak{g}%
_{3}^{\ast}}$ as given in Eq.(\ref{Ohm2T*G}).
\end{proposition}

\begin{remark}
The existence of symplectic action of $\mathfrak{g}^{\ast}$ on $ TT^{\ast
}G$ can be traced back to existence of the symplectic diffeomorphism
\begin{equation}
 \Omega_{G\circledS\mathfrak{g}^{\ast}}^{\flat}:TT^{\ast
}G\rightarrow T^{\ast}T^{\ast}G:\left( g,\mu,\xi,\nu\right)
\rightarrow\left( g,\mu,\nu+\ad_{\xi}^{\ast}\mu,-\xi\right)
\end{equation}
described in \cite{esen2014tulczyjew}.
\end{remark}

In the following proposition, we discuss the actions $\psi$ and $\phi$ of $%
\mathfrak{g}^{\ast}$ on $ TT^{\ast}G$ in Eqs(\ref{psi}) and (\ref{phi})
from a different point of view.

\begin{proposition}
The mappings
\begin{align}
Emb_{1} & :G\circledS\mathfrak{g}^{\ast}\hookrightarrow\  TT^{\ast
}G:\left( g,\mu\right) \rightarrow\left( g,\mu,0,0\right)  \notag \\
Emb_{2} & :G\circledS\mathfrak{g}^{\ast}\hookrightarrow\  TT^{\ast
}G:\left( g,\nu\right) \rightarrow\left( g,0,0,\nu\right)  \label{Emb2}
\end{align}
define a Lagrangian and a symplectic, respectively, embeddings of $G\circledS%
\mathfrak{g}^{\ast}$ into $ TT^{\ast}G$.
\end{proposition}

\begin{proof}
The first embedding is Lagrangian because it is the zero section of the
fibration $\  TT^{\ast}G\rightarrow G\circledS\mathfrak{g}_{1}^{\ast}$.
The second one is symplectic because the pull-back of $\Omega_{\
 TT^{\ast}G}$ to $G\circledS\mathfrak{g}^{\ast}$ by $Emb_{2}$ results in
the symplectic two-form $\Omega_{G\circledS\mathfrak{g}^{\ast}}$ in Eq.(\ref%
{OhmT*G}). On the image of $Emb_{2}$, the Hamilton's equations (\ref{HamTT*G}%
) reduce to the trivialized Hamilton's equations (\ref{ULP}) on $G\circledS%
\mathfrak{g}^{\ast }$. Consequently, the embedding $\mathfrak{g}%
_{3}^{\ast}\rightarrow \  TT^{\ast}G$ is a Poisson map. When $E=h\left(
\nu\right) $ the Hamilton's equations (\ref{HamTT*G}) reduce to the
Lie-Poisson equations (\ref{LP}).
\end{proof}

\subsubsection{Reduction of $TT^*G$ by $G\circledS\mathfrak{g}$}

The action
\begin{align}
&\vartheta  : 
 TT^{\ast}G \times \left( G\circledS\mathfrak{g}_{2}\right)\rightarrow TT^{\ast}G  , \\
& \left( 
\left( g,\mu,\xi,\nu\right); \left( h,\eta\right) \right) \mapsto\vartheta_{\left(
h,\eta\right) }\left( g,\mu,\xi,\nu\right)  :=\left(
gh,\mu,\xi+\Ad_{g }\eta,\nu-\ad^\ast_{\Ad_{g}\eta}\mu\right)
\label{varpsi}
\end{align}
of $G\circledS\mathfrak{g}$ on $ TT^{\ast}G$ can be described as a
composition
\begin{equation*}
\vartheta_{\left( h,\eta\right) }=\vartheta_{\left( h,0\right)
}\circ\vartheta_{\left( e,Ad_{g}\eta\right) },
\end{equation*}
where $\vartheta_{\left( h,0\right) }$ and $\vartheta_{\left(
e,Ad_{g}\eta\right) }$ can be identified with the actions of $G$ and $%
\mathfrak{g}$ on $ TT^{\ast}G$, respectively. Since both of these are symplectic, the action $%
\vartheta$ of $G\circledS\mathfrak{g}$ on $ TT^{\ast}G$ is symplectic.

\begin{proposition}
The Poisson reduction of $ TT^{\ast}G$ under the action of $G\circledS
\mathfrak{g}_{2}$ in Eq.(\ref{varpsi})\ results in $\mathfrak{g}_{1}^{\ast
}\times\mathfrak{g}_{3}^{\ast}$ endowed with the bracket
\begin{equation}
\left\{ E,F\right\} _{\mathfrak{g}_{1}^{\ast}\times\mathfrak{g}%
_{3}^{\ast}}\left( \mu,\nu\right) =\left\langle \nu,\left[ \frac{\delta E}{%
\delta\nu },\frac{\delta F}{\delta\nu}\right] \right\rangle .
\label{Poig*g*}
\end{equation}
\end{proposition}

\begin{remark}
Although the Poisson bracket (\ref{Poig*g*}) structurally resembles the
Lie-Poisson bracket on $\mathfrak{g}_{3}^{\ast}$, it is not a Lie-Poisson
bracket on $\mathfrak{g}_{1}^{\ast}\times\mathfrak{g}_{3}^{\ast}$ considered
as dual of Lie algebra $\mathfrak{g}\circledS\mathfrak{g}$ of the group $%
G\circledS\mathfrak{g}$. We refer to the Poisson bracket in Eq.(\ref%
{LPBg*xg*}) for the Lie-Poisson structure on $Lie^{\ast}(G\circledS\mathfrak{%
g})=\mathfrak{g}^{\ast}\times\mathfrak{g}^{\ast}$.
\end{remark}

Right invariant vector field generating the action $\vartheta$ is associted to
two tuples $\left( \xi_{2},\xi_{3}\right) $ in the Lie algebra of $G\circledS%
\mathfrak{g}_{2}$, and is given by%
\begin{equation}
X_{\left( \xi_{2},0,\xi_{3},0\right) }^{ TT^{\ast}G}( g,\mu,\xi,\nu)=\left(
TR_{g}\xi_{2},ad_{\xi_{2}}^{\ast}\mu,\xi_{3}+\left[ \xi,\xi_{2}\right]
,ad_{\xi_{2}}^{\ast}\nu\right) .
\end{equation}
The momentum map  for this Hamiltonian action is defined by the equation
\begin{equation*}
\left\langle \mathbf{J}_{TT^{\ast}G}^{G\circledS\mathfrak{g}%
_{2}}\left( g,\mu,\xi,\nu\right) ,\left( \xi_{2},\xi_{3}\right)
\right\rangle =\left\langle \theta_{2},X_{\left( \xi_{2},0,\xi_{3},0\right)
}^{ TT^{\ast}G}\right\rangle =\left\langle \mu,\xi_{3}\right\rangle
+\left\langle \nu+ad_{\xi}^{\ast}\mu,\xi_{2}\right\rangle ,
\end{equation*}
where $\theta_{2}$, in Eq.(\ref{2}), is the Tulczyjew potential one-form on $ TT^{\ast}G$. We find
\begin{equation*}
\mathbf{J}_{TT^{\ast}G}^{G\circledS\mathfrak{g}_{2}}:\
 TT^{\ast}G\rightarrow Lie^{\ast}\left( G\circledS\mathfrak{g}_{2}\right)
=\mathfrak{g}^{\ast}\times\mathfrak{g}^{\ast}:\left( g,\mu ,\xi,\nu\right)
=\left( \nu+ad_{\xi}^{\ast}\mu,\mu\right) .
\end{equation*}
Note that, we have the following relation
\begin{equation*}
\mathbf{J}_{TT^{\ast}G}^{G\circledS\mathfrak{g}_{2}}\left(
g,\mu,\xi,\nu\right) =\left( \mathbf{J}_{TT^{\ast}G}^{G}\left(
g,\mu,\xi,\nu\right) ,\mathbf{J}_{TT^{\ast}G}^{\mathfrak{g}%
_{2}}\left( g,\mu,\xi,\nu\right) \right)
\end{equation*}
for momentum maps in Eqs.(\ref{MGonTT*G}) and (\ref{MgonTT*G}) for the
actions of $G$ and $\mathfrak{g}_{2}$ on $ TT^{\ast}G$. The preimage of a
fixed element $\left( \lambda,\mu\right) \in\mathfrak{g}^{\ast}\times%
\mathfrak{g}^{\ast}$ is
\begin{equation*}
\left( \mathbf{J}_{TT^{\ast}G}^{G\circledS\mathfrak{g}%
_{2}}\right) ^{-1}\left( \lambda,\mu\right) =\left\{ \left( g,\mu,\xi
,\nu\right) :\nu=\lambda-ad_{\xi}^{\ast}\mu\right\}
\end{equation*}
which we may identify with the semidirect product $G\circledS\mathfrak{g}%
_{2} $. We recall the coadjoint action $Ad_{\left( g,\xi\right) }^{\ast}$,
in Eq.(\ref{coad2}), of the group $G\circledS\mathfrak{g}_{2}$ on the dual $%
\mathfrak{g}^{\ast}\times\mathfrak{g}^{\ast}$ of its Lie algebra. The
isotropy subgroup $\left( G\circledS\mathfrak{g}_{2}\right) _{\left(
\lambda,\mu\right) }$ of this coadjoint action is
\begin{equation*}
\left( G\circledS\mathfrak{g}_{2}\right) _{\left( \lambda,\mu\right)
}=\left\{ \left( g,\xi\right) \in G\circledS\mathfrak{g}_{2}:Ad_{\left(
g,\xi\right) }^{\ast}\left( \lambda,\mu\right) =\left( \lambda,\mu\right)
\right\}
\end{equation*}
and acts on the preimage $\left( \mathbf{J}_{TT^{\ast}G}^{G%
\circledS\mathfrak{g}_{2}}\right) ^{-1}\left( \lambda,\mu\right) $. A
generic quotient space
\begin{equation*}
\left. \left( \mathbf{J}_{TT^{\ast}G}^{G\circledS\mathfrak{g}%
_{2}}\right) ^{-1}\left( \lambda,\mu\right) \right/ \left( G\circledS
\mathfrak{g}_{2}\right) _{\left( \lambda,\mu\right) }\simeq\left. G\circledS%
\mathfrak{g}_{2}\right/ \left( G\circledS\mathfrak{g}_{2}\right) _{\left(
\lambda,\mu\right) }\simeq\mathcal{O}_{\left( \lambda,\mu\right) }
\end{equation*}
is a coadjoint orbit in $\mathfrak{g}^{\ast}\times\mathfrak{g}^{\ast}$
through the point $\left( \lambda,\mu\right) $ under the coadjoint action $%
Ad_{\left( g,\xi\right) }^{\ast}$ in Eq.(\ref{coad2}).

\begin{proposition}
The symplectic reduction of $ TT^{\ast}G$ under the action of $G\circledS%
\mathfrak{g}_{2}$ given in Eq.(\ref{varpsi}) results in the coadjoint orbit $%
\mathcal{O}_{\left( \lambda,\mu\right) }$in $\mathfrak{g}^{\ast}\times%
\mathfrak{g}^{\ast}$ through the point $\left( \lambda ,\mu\right) $ under
the coadjoint action $Ad_{\left( g,\xi\right) }^{\ast}$ in Eq.(\ref{coad2})
as the total space and the symplectic two-from $\Omega_{\mathcal{O}_{\left(
\lambda,\mu\right) }}$ in Eq.(\ref{Symp/Gxg}).
\end{proposition}

It is also possible to obtain the symplectic space $\mathcal{O}_{\left(
\lambda,\mu\right) }$ in two steps. Recall the
symplectic reduction of $ TT^{\ast}G$ under the action of $\mathfrak{g}%
_{2}$ at $\mu\in\mathfrak{g}^{\ast}$ which results in $G\circledS \mathfrak{g%
}_{3}^{\ast}$ with the canonical symplectic two-from $\Omega _{G\circledS%
\mathfrak{g}_{3}^{\ast}}$. Then, consider the action of  isotropy subgroup
$G_{\mu}$ on $G\circledS\mathfrak{g}_{3}^{\ast}$ and apply symplectic
reduction which results in $\left( \mathcal{O}_{\left( \lambda,\mu\right)
},\Omega_{\mathcal{O}_{\left( \lambda,\mu\right) }}\right) $. Following is
the diagram summarizing this two stage reduction of $ TT^{\ast}G$.
\begin{equation}
\xymatrix{&&
(G\circledS\mathfrak{g}_{1}^{\ast})\circledS(\mathfrak{g}_{2}\circledS%
\mathfrak{g}_{3}^{\ast } )\ar[dll]|-{\text{SR by } \mathfrak{g}_{2}\text{
at }\mu \text{ } } \ar[dd]|-{\text{SR by } G\circledS\mathfrak{g}_{2}
\text{ at }(\lambda,\mu)\text{ } } &\qquad \qquad\\ G\circledS\mathfrak{g}_{3}^{\ast}
\ar[drr]|-{\text{SR by } G_{\mu} \text{ at } \lambda\text{ } } \\
&&\mathcal{O}_{(\lambda,\mu)}
\\
&& \text{\small Reductions of $TT^*G$ by $G\circledS \mathfrak{g}$} }
\end{equation}
\subsubsection{Reduction of $TT^*G$ by $G\circledS\mathfrak{g}^{\ast}$}

The action
\begin{equation*}
\alpha: TT^{\ast}G \times \left( G\circledS\mathfrak{g}_{1}^{\ast}\right)\rightarrow TT^{\ast}G :\left( \left( g,\mu,\xi,\nu\right);\left(
h,\lambda\right) \right) \mapsto
\alpha_{\left( h,\lambda\right) }\left( g,\mu,\xi,\nu\right)
\end{equation*}
of $G\circledS\mathfrak{g}_{1}^{\ast}$ on $ TT^{\ast}G$ is given by%
\begin{equation}
\alpha_{(h,\lambda)}(g,\mu,\xi,\nu)
=(gh,\mu+\Ad_{g}^{\ast}\lambda,\xi,\nu).  \label{alpha}
\end{equation}
As in the case of the action of $G\circledS\mathfrak{g}_{2}$, it can also be
described by composition of two actions
\begin{equation*}
\alpha_{\left( h,\lambda\right) }=\alpha_{\left( h,0\right) }\circ
\alpha_{\left( e,Ad_{g}^{\ast}\lambda\right) },
\end{equation*}
where, $\alpha_{\left( h,0\right) }$ and $\alpha_{\left( e,Ad_{g}^{\ast
}\lambda\right) }$ can be identified with the actions of $G$ and $\mathfrak{g%
}_{1}^{\ast}$ on $ TT^{\ast}G$, respectively. Since both of them are symplectic, $\alpha$ is also
symplectic.

\begin{proposition}
Poisson reduction of $ TT^{\ast}G$ under the action (\ref{alpha}) of $%
G\circledS\mathfrak{g}_{1}^{\ast}$ results in $\mathfrak{g}_{2}\circledS%
\mathfrak{g}_{3}^{\ast}$ endowed with the bracket
\begin{equation}
\left\{ F,H\right\} _{\mathfrak{g}_{2}\circledS\mathfrak{g}_{3}^{\ast}}\left(
\xi,\nu\right) =\left\langle \nu,\left[ \frac{\delta E}{\delta\nu},\frac{%
\delta F}{\delta\nu}\right] \right\rangle .  \label{Poigg*}
\end{equation}
\end{proposition}

\begin{remark}
Regarding $\mathfrak{g}^{\ast}\times\mathfrak{g}$ as dual of the Lie algebra
$\mathfrak{g}\circledS\mathfrak{g}^{\ast}$ of $G\circledS\mathfrak{g}^{\ast}$%
, we obtained the Lie-Poisson bracket in Eq.(\ref{LPE}). Although $\mathfrak{g}%
^{\ast }\times\mathfrak{g}$ and $\mathfrak{g}_{2}\times\mathfrak{g}%
_{3}^{\ast}$ are isomorphic as vector spaces, (\ref{LPE}) is different from the Poisson
bracket in Eq.(\ref{Poigg*}) as manifestation of group structure carried by adapted trivialization.
\end{remark}

Infinitesimal generator of $\alpha$ is associated to the two tuple $\left(
\xi_{2},\nu_{2}\right) $ in the Lie algebra $\mathfrak{g}\circledS \mathfrak{%
g}^{\ast}$ of $G\circledS\mathfrak{g}_{1}^{\ast}$ and is in the form
\begin{equation}
X_{\left( \xi_{2},\nu_{2},0,0\right) }^{ TT^{\ast}G}(g,\mu,\xi,\nu)=\left(
TR_{g}\xi_{2},\nu_{2}+\ad_{\xi_{2}}^{\ast}\mu,\left[ \xi,\xi_{2}\right]
,\ad_{\xi_{2}}^{\ast}\nu-\ad_{\xi}^{\ast}\nu_{2}\right) .
\end{equation}
The momentum mapping $\mathbf{J}_{TT^{\ast}G}^{G\circledS
\mathfrak{g}_{1}^{\ast}}$ is defined by the equation
\begin{equation*}
\left\langle \mathbf{J}_{TT^{\ast}G}^{G\circledS\mathfrak{g}%
_{1}^{\ast}}\left( g,\mu,\xi,\nu\right) ,\left( \xi_{2},\nu_{2}\right)
\right\rangle =\left\langle \theta_{1},X_{\left( \xi_{2},\nu_{2},0,0\right)
}^{ TT^{\ast}G}\right\rangle =-\left\langle \nu_{2},\xi\right\rangle
+\left\langle \nu+ad_{\xi}^{\ast}\mu,\xi_{2}\right\rangle ,
\end{equation*}
where $\theta_{1}$ is the potential one-form given by Eq.(\ref{1}). We
obtain
\begin{equation*}
\mathbf{J}_{TT^{\ast}G}^{G\circledS\mathfrak{g}_{1}^{\ast}}:\
 TT^{\ast}G\rightarrow Lie^{\ast}\left( G\circledS\mathfrak{g}%
_{1}^{\ast}\right) =\mathfrak{g}^{\ast}\times\mathfrak{g}:\left( g,\mu
,\xi,\nu\right) \rightarrow\left( \nu+ad_{\xi}^{\ast}\mu,-\xi\right)
\end{equation*}
which can be decomposed as
\begin{equation*}
\mathbf{J}_{TT^{\ast}G}^{G\circledS\mathfrak{g}%
_{1}^{\ast}}\left( g,\mu,\xi,\nu\right) =\left( \mathbf{J}_{\text{ }%
 TT^{\ast}G}^{G}\left( g,\mu,\xi,\nu\right) ,\mathbf{J}_{\text{ }%
 TT^{\ast}G}^{\mathfrak{g}_{1}^{\ast}}\left( g,\mu,\xi,\nu\right) \right)
\end{equation*}
where $\mathbf{J}_{TT^{\ast}G}^{G}$ and $\mathbf{J}_{\text{ }%
 TT^{\ast}G}^{\mathfrak{g}_{1}^{\ast}}$ are momentum mappings in Eqs.(\ref%
{MGonTT*G}) and (\ref{Mg*onTT*G}) for the actions of $G$ and $\mathfrak{g}%
_{1}^{\ast}$ on $ TT^{\ast}G$, respectively. The preimage of an element $%
\left( \lambda,\xi\right) \in\mathfrak{g}^{\ast}\times \mathfrak{g}$ is
\begin{equation*}
\left( \mathbf{J}_{TT^{\ast}G}^{G\circledS\mathfrak{g}_{1}^{\ast
}}\right) ^{-1}\left( \lambda,\xi\right) =\left\{ \left( g,\mu,-\xi
,\nu\right) :\nu=\lambda+ad_{\xi}^{\ast}\mu\right\}
\end{equation*}
which can be identified with the space $G\circledS\mathfrak{g}_{1}^{\ast}$.
The isotropy subgroup of coadjoint action of $G\circledS\mathfrak{g}_{2}$ on
$\mathfrak{g}^{\ast}\times\mathfrak{g}$ is
\begin{equation*}
\left( G\circledS\mathfrak{g}_{1}^{\ast}\right) _{\left( \lambda ,\xi\right)
}=\left\{ \left( g,\mu\right) \in G\circledS\mathfrak{g}_{2}:Ad_{\left(
g,\mu\right) }^{\ast}\left( \lambda,\xi\right) =\left( \lambda,\xi\right)
\right\} ,
\end{equation*}
where the coadjoint action is given by Eq.(\ref{Coad}). Isotropy subgroup
acts on preimage of $\left( \lambda,\xi\right) $ and results in the
coadjoint orbit through the point $\left( \lambda,\xi\right) \in\mathfrak{g}%
^{\ast }\times\mathfrak{g}$
\begin{equation*}
\left. \left( \mathbf{J}_{TT^{\ast}G}^{G\circledS\mathfrak{g}%
_{1}^{\ast}}\right) ^{-1}\left( \lambda,\xi\right) \right/ \left( G\circledS%
\mathfrak{g}_{1}^{\ast}\right) _{\left( \lambda,\xi\right) }\simeq\left.
G\circledS\mathfrak{g}_{1}^{\ast}\right/ \left( G\circledS\mathfrak{g}%
_{2}\right) _{\left( \lambda,\xi\right) }\simeq\mathcal{O}_{\left(
\lambda,\xi\right) }.
\end{equation*}
\begin{proposition}
Symplectic reduction of $ TT^{\ast}G$ under the action of $G\circledS
\mathfrak{g}_{1}^{\ast}$ given by Eq.(\ref{alpha}) results in the coadjoint
orbit $\mathcal{O}_{\left( \lambda,\xi\right) }$ and the symplectic two-from
$\Omega_{\mathcal{O}_{\left( \lambda,\xi\right) }}$ in Eq.(\ref{SymOr}).
\end{proposition}
Similar to the reduction of $ TT^{\ast}G$ by $G\circledS\mathfrak{g}_{2} $,
we may perform symplectic reduction of $ TT^{\ast}G$ with action of $%
G\circledS\mathfrak{g}_{1}^{\ast}$ by two stages. Recall
symplectic reduction of $ TT^{\ast}G$ with action of $\mathfrak{g}%
_{1}^{\ast}$ at $\xi\in\mathfrak{g}$ which results in $G\circledS\mathfrak{g}%
_{3}^{\ast}$ and the canonical symplectic two-from $\Omega_{G\circledS%
\mathfrak{g}_{3}^{\ast}}$. Then, consider the action of  isotropy subgroup
$G_{\xi}$, defined in Eq.(\ref{Gxi}), on $G\circledS\mathfrak{g}_{3}^{\ast}$
and apply symplectic reduction. This gives $\mathcal{O}_{\left(
\lambda,\xi\right) }$ and the symplectic two-form $\Omega_{\mathcal{O}%
_{\left( \lambda,\xi\right) }}$. Following diagram shows this two stage
reduction of $ TT^{\ast}G$ 
\begin{equation}
\xymatrix{&&
(G\circledS\mathfrak{g}_{1}^{\ast})\circledS(\mathfrak{g}_{2}\circledS%
\mathfrak{g}_{3}^{\ast }) \ar[dll]|-{\text{SR by }
\mathfrak{g}_{1}^{\ast}\text{ at }\xi } \ar[dd]|-{\text{SR by }
G\circledS\mathfrak{g}_{1}^{\ast} \text{ at }(\lambda,\xi)} &\qquad \qquad\\
G\circledS\mathfrak{g}_{2}^{\ast} \ar[drr]|-{\text{SR by } G_{\xi} \text{
at } \lambda} \\ &&\mathcal{O}_{(\lambda,\xi)} 
\\
&& \text{\small Reduction of $TT^*G$ by $G\circledS \mathfrak{g}^*_1$}
}
\end{equation}
We summarize diagrammatically all possible reductions of Hamiltonian
dynamics on the Tulczyjew symplectic space $ TT^{\ast}G$.
\begin{equation}
\xymatrix{(G\circledS\mathfrak{g}_{1}^{\ast
})\circledS\mathfrak{g}_{3}^{\ast } && G\circledS\mathfrak{g}_{3}^{\ast }
\ar@{^{(}->}[rr]^{\txt{\small symplectic\\\small leaf}}
\ar@{^{(}->}[ll]_{\txt{\small symplectic\\ \small  leaf}}
&&G\circledS(\mathfrak{g}_{2}\circledS\mathfrak{g}_{3}^{\ast })
\\\\\mathfrak{g}_{1}^{\ast }\circledS\mathfrak{g}_{3}^{\ast }
\ar[uu]^{\txt{\small  Poisson\\ \small embedding}} &&(G\circledS\mathfrak{g}_{1}^{\ast
})\circledS(\mathfrak{g}_{2}\circledS\mathfrak{g}_{3}^{\ast })
\ar[uull]^{\text{PR by }\mathfrak{g}_{2}} \ar[uurr]_{\text{PR by
}\mathfrak{g}_{1}^{\ast }} \ar@/_/[dd]_{\text{SR by }\mathfrak{g}_{2}}
\ar@/^/[dd]^{\text{SR by }\mathfrak{g}_{1}^{\ast }} \ar[ll]|-{\text{PR
by }G\circledS\mathfrak{g}_{2}} \ar[rr]|-{\text{PR by
}G\circledS\mathfrak{g}_{1}^{\ast }} \ar[ddll]_{\text{ SR by
}G\circledS\mathfrak{g}_{2}} \ar[ddrr]^{\text{PR by
}G\circledS\mathfrak{g}_{1}^{\ast }} \ar@/_/[uu]_{\text{SR by
}\mathfrak{g}_{2}} \ar@/^/[uu]^{\text{SR by }\mathfrak{g}_{1}^{\ast }}
&&\mathfrak{g}_{2}\circledS\mathfrak{g}_{3}^{\ast}
\ar[uu]_{\txt{\small  Poisson\\ \small  embedding}} \\\\\mathcal{O}_{(\mu,\nu)}
\ar@{_{(}->}[uu]^{\txt{symplectic\\leaf}} &&G\circledS\mathfrak{g}_{3}^{\ast
} \ar[ll]_{\text{SR by }G_{\mu}} \ar[rr]^{\text{SR by }G_{\xi}}
&&\mathcal{O}_{(\mu,\xi)} \ar@{_{(}->}[uu]_{\txt{\small  symplectic\\ \small  leaf}}
\\
&&\txt{\small Hamiltonian reductions of \\ $TT^*Q=(G\circledS\mathfrak{g}_{1}^{\ast
})\circledS(\mathfrak{g}_{2}\circledS\mathfrak{g}_{3}^{\ast })$}
}
\label{TT*Gd}
\end{equation}

\subsection{Lagrangian Dynamics on $TT^*G$}~

As it is a tangent bundle, we can study Lagrangian dynamics on $%
 TT^{\ast}G\simeq\left( G\circledS\mathfrak{g}_{1}^{\ast}\right)
\circledS\left( \mathfrak{g}_{2}\circledS\mathfrak{g}_{3}^{\ast}\right) $.
We define variation of the base element $\left( g,\mu\right) \in G\circledS%
\mathfrak{g}_{1}^{\ast}$ by tangent lift of right translation of the Lie
algebra element $\left( \eta,\lambda\right) \in\mathfrak{g}\circledS%
\mathfrak{g}^{\ast}$, that is
\begin{equation*}
\delta\left( g,\mu\right) =T_{\left( e,0\right) }R_{\left( g,\mu\right)
}\left( \eta,\lambda\right) =\left( T_{e}R_{g}\eta,\lambda+ad_{\eta}^{\ast
}\mu\right) .
\end{equation*}
To obtain the reduced variational principle $\delta\left( \xi,\nu\right) $
on the Lie algebra $\mathfrak{g}_{2}\circledS\mathfrak{g}_{3}^{\ast}$ we
compute
\begin{equation} \label{var}
\begin{split}
\delta\left( \xi,\nu\right) & =\frac{d}{dt}\left( \eta,\lambda\right)
+[\left( \xi,\nu\right) ,\left( \eta,\lambda\right) ]_{\mathfrak{g}\circledS%
\mathfrak{g}^{\ast}}  \\
& =\frac{d}{dt}\left( \eta,\lambda\right) +\left( [\xi,\eta],ad_{\eta
}^{\ast}\nu-ad_{\xi}^{\ast}\lambda\right)    \\
& =\left( \dot{\eta}+[\xi,\eta],\dot{\lambda}+ad_{\eta}^{\ast}\nu-ad_{\xi
}^{\ast}\lambda\right) ,  
\end{split}
\end{equation}
for any $\left( \eta,\lambda\right) \in\mathfrak{g}\circledS\mathfrak{g}%
^{\ast}$. Assuming $\delta\left( \xi,\nu\right) =\left( \delta\xi,\delta
\nu\right) $ and $\delta\left( g,\mu\right) =\left( \delta g,\delta
\mu\right) $, we have the set of variations
\begin{equation}
\delta g=T_{e}R_{g}\eta,\text{ \ \ }\delta\mu=\lambda+ad_{\eta}^{\ast}\mu,%
\text{ \ \ }\delta\xi=\dot{\eta}+[\xi,\eta]\text{, \ \ }\delta\nu =\dot{%
\lambda}+ad_{\eta}^{\ast}\nu-ad_{\xi}^{\ast}\lambda  \label{VarTT*G}
\end{equation}
for an arbitrary choice of $\left( \eta,\lambda\right) \in\mathfrak{g}%
\circledS\mathfrak{g}^{\ast}$. Note that, these variations are the image of
the right invariant vector field $X_{\left( \eta,\lambda,\dot{\eta},\dot{%
\lambda}\right) }^{ TT^{\ast}G}$ generated by $\left( \eta,\lambda,\dot{%
\eta},\dot{\lambda}\right) .$

\begin{proposition}
For a given Lagrangian $E$ on $ TT^{\ast}G$, extremals of action integral
are defined by the trivialized Euler-Lagrange equations%
\begin{equation}\label{UnEPTT*G}
\begin{split}
\frac{d}{dt}\left( \frac{\delta E}{\delta\xi}\right) &
=T_{e}^{\ast}R_{g}\left( \frac{\delta E}{\delta g}\right) -\ad_{\frac{\delta E%
}{\delta\mu }}^{\ast}\mu+\ad_{\xi}^{\ast}\left( \frac{\delta E}{\delta\xi}%
\right) -\ad_{\frac{\delta E}{\delta\nu}}^{\ast}\nu \\
\frac{d}{dt}\left( \frac{\delta E}{\delta\nu}\right) & =\frac{\delta E}{%
\delta\mu}+\ad_{\xi}\frac{\delta E}{\delta\nu} 
\end{split}
\end{equation}
obtained by the variational principles in Eq.(\ref{VarTT*G}).
\end{proposition}

\begin{proof}
Let us begin with the observation that for any $(g,\xi)\in G\circledS \G{g}$, the variation $\d(g,\xi) = (\d g, \d \xi)$ at $(\eta,\dot{\eta}) \in \G{g}\circledS \G{g}$ may be given by
\[
(\d g, \d \xi) = X^{TG}_{(\eta,\dot{\eta})}(g,\xi) = (TR_g\eta,\dot{\eta}+[\eta,\xi]).
\]
Accordingly, given a Lagrangian $\G{L}:TG \cong G\circledS\G{g}\to\B{R}$, the action integral 
\begin{align*}
& \d\int_a^b\,\G{L}(g,\xi)\,dt = \int_a^b\,\left(\Big\langle\frac{\d\G{L}}{\d g},\d g\Big\rangle + \Big\langle\frac{\d\G{L}}{\d \xi}, \d \xi\Big\rangle\right)\,dt \\
& = \int_a^b\,\left(\Big\langle\frac{\d\G{L}}{\d g},TR_g\eta\Big\rangle + \Big\langle\frac{\d\G{L}}{\d \xi}, \dot{\eta}\Big\rangle + \Big\langle\frac{\d\G{L}}{\d \xi}, [\eta,\xi]\Big\rangle\right)\,dt  \\
& =\left.\Big\langle\frac{\d\G{L}}{\d \xi}, \eta\Big\rangle\right|_a^b + \int_a^b\,\left(\Big\langle T^\ast R_g\frac{\d\G{L}}{\d g},\eta\Big\rangle + \Big\langle -\frac{d}{dt}\frac{\d\G{L}}{\d \xi}, \eta\Big\rangle + \Big\langle \ad^\ast_{\xi}\frac{\d\G{L}}{\d \xi}, \eta\Big\rangle\right)\,dt
\end{align*}
leads to the (trivialized) Euler-Lagrange equations
\[
\frac{d}{dt}\frac{\d\G{L}}{\d \xi} = T^\ast R_g\frac{\d\G{L}}{\d g} + \ad^\ast_{\xi}\frac{\d\G{L}}{\d \xi}.
\]
Accordingly, the (trivialized) Euler-Lagrange equations on $TT^\ast G$ are given by
\[
\frac{d}{dt}\frac{\d E}{\d (\xi,\nu)} = T^\ast R_{(g,\mu)}\frac{\d E}{\d (g,\mu)} + \ad^\ast_{(\xi,\nu)}\frac{\d E}{\d (\xi,\nu)},
\]
where the first summand on the right hand side is 
\[
T^\ast R_{(g,\mu)}\frac{\d E}{\d (g,\mu)}  = \Big(T^\ast R_g\frac{\d E}{\d g} - \ad^\ast_{\frac{\d E}{\d \mu}}\mu,\frac{\d E}{\d \mu}\Big),
\]
while the second summand being
\[
\ad^\ast_{(\xi,\nu)}\frac{\d E}{\d (\xi,\nu)} = \Big(\ad^\ast_{\xi}\frac{\d E}{\d \xi} - \ad^\ast_{\frac{\d E}{\d \nu}}\nu,\frac{\d E}{\d \nu}+\ad_\xi\frac{\d E}{\d \nu}\Big).
\]
\end{proof}

\begin{proposition}
Given a Lagrangian $E=E(g,\mu,\xi,\nu)$ on $TT^\ast G$, the quantity 
\begin{equation}
\Big\langle \frac{\delta E}{\delta\xi},\xi \Big\rangle
+\Big\langle \frac{\delta E}{\delta\nu},\nu \Big\rangle
-E
\end{equation}
is constant.
\end{proposition}

\begin{proof}
Let us begin with
\begin{align*}
 \frac{d E}{dt} &= \Big\langle \frac{\p E}{\p g},\dot{g} \Big\rangle + \Big\langle \frac{\p E}{\p \mu},\dot{\mu} \Big\rangle + \Big\langle \frac{\p E}{\p \xi},\dot{\xi} \Big\rangle + \Big\langle \frac{\p E}{\p \nu},\dot{\nu} \Big\rangle  \\
 & = \Big\langle \frac{\p E}{\p g}, TR_g\xi \Big\rangle + \Big\langle \frac{\p E}{\p \mu},\nu+\ad^\ast_{\xi}\mu \Big\rangle + \Big\langle \frac{\p E}{\p \xi},\dot{\xi} \Big\rangle + \Big\langle \frac{\p E}{\p \nu},\dot{\nu} \Big\rangle \\
 & = \Big\langle T^\ast R_g\left(\frac{\p E}{\p g}\right) - \ad^\ast_{\frac{\p E}{\p \mu}}\mu, \xi \Big\rangle + \Big\langle \frac{\p E}{\p \mu},\nu \Big\rangle + \Big\langle \frac{\p E}{\p \xi},\dot{\xi} \Big\rangle + \Big\langle \frac{\p E}{\p \nu},\dot{\nu} \Big\rangle .
\end{align*}
Next, substituting the (trivialized) Euler-Lagrange equations \eqref{UnEPTT*G}, we obtain
\begin{align*}
 \frac{d E}{dt} & = \Big\langle \frac{d}{dt}\left(\frac{\p E}{\p \xi} \right) - \ad^\ast_{\xi}\frac{\p E}{\p \xi} +  \ad^\ast_{\frac{\p E}{\p \nu}}\nu, \xi \Big\rangle + \Big\langle \nu, \frac{d}{dt}\left(\frac{\p E}{\p \nu}\right) - \ad_{\xi}\frac{\p E}{\p \nu} \Big\rangle + \Big\langle \frac{\p E}{\p \xi},\dot{\xi} \Big\rangle + \Big\langle \frac{\p E}{\p \nu} ,\dot{\nu}\Big\rangle \\
 & = \Big\langle \frac{d}{dt}\left(\frac{\p E}{\p \xi} \right), \xi \Big\rangle + \Big\langle \frac{\p E}{\p \xi},\dot{\xi} \Big\rangle  + \Big\langle \dot{\nu},\frac{\p E}{\p \nu} \Big\rangle + \Big\langle \nu, \frac{d}{dt}\left(\frac{\p E}{\p \nu}\right) \Big\rangle + \\
 & \Big\langle   \ad^\ast_{\frac{\p E}{\p \nu}}\nu, \xi \Big\rangle - \Big\langle  \ad^\ast_{\xi}\frac{\p E}{\p \xi} , \xi \Big\rangle  -\Big\langle \nu, \ad_{\xi}\frac{\p E}{\p \nu} \Big\rangle  \\
 & = \frac{d}{dt}\left(\Big\langle \frac{\p E}{\p \xi} , \xi \Big\rangle +\Big\langle \nu,\frac{\p E}{\p \nu} \Big\rangle\right),
\end{align*}
from which we result follows.
\end{proof}

\subsubsection{Reductions on $TT^*G$} When the Lagrangian density $E$ in the trivialized Euler-Lagrange equations (%
\ref{UnEPTT*G}) is independent of the group variable $g\in G$, we arrive at
Euler-Lagrange equations (\ref{g*gg*-}) on $\mathfrak{g}_{1}^{\ast}\circledS\left(
\mathfrak{g}_{2}\circledS\mathfrak{g}_{3}^{\ast}\right) $. In addition, if the
Lagrangian $E$ depends only on fiber coordinates $E=E\left( \xi,\nu\right) ,$
we have the Euler-Poincaré equations (\ref{EPgg*}).

\begin{proposition}
The Euler-Poincaré equations on the Lie algebra $\mathfrak{g}_{2}\circledS%
\mathfrak{g}_{3}^{\ast}$ are
\begin{equation}
\frac{d}{dt}\left( \frac{\delta E}{\delta\xi}\right) =\ad_{\xi}^{\ast}\left(
\frac{\delta E}{\delta\xi}\right) -\ad_{\frac{\delta E}{\delta\nu}}^{\ast}\nu,%
\text{ \ \ }\frac{d}{dt}\left( \frac{\delta E}{\delta\nu}\right) =-\ad_{\xi}%
\frac{\delta E}{\delta\nu}.  \label{EPgg*}
\end{equation}
\end{proposition}

 If, moreover, $%
E=E\left( \xi\right) ,$ the Euler-Poincaré equations (\ref{EPEq}) on $%
\mathfrak{g}_{2}$ arise. This procedure is called reduction by stages \cite%
{n2001lagrangian, holm1998euler}.

Alternatively, the Lagrangian density $E$ in trivialized Euler-Lagrange
equations (\ref{UnEPTT*G}) can be independent of $\mu\in\mathfrak{g}%
_{1}^{\ast}$, that is, $E$ can be invariant under the action of $\mathfrak{g}%
_{1}^{\ast}$ on $ TT^{\ast}G$. In this case, we have Euler-Lagrange
equations \eqref{Ggg*-} on $G\circledS\left( \mathfrak{g}_{2}\circledS\mathfrak{g}_{3}^{\ast
}\right) $. When $E=E\left( g,\xi\right) $, we have trivialized
Euler-Lagrange equations \eqref{Gg-} on $G\circledS\mathfrak{g}_{2}$. The following
diagram summarizes this discussion.

\begin{align}
& \xymatrix{ & {\begin{array}{c}(G\circledS \mathfrak{g}_{1}^{\ast
})\circledS (\mathfrak{g}_{2}\circledS \mathfrak{g}_{3}^{\ast })\\\text{EL
in (\ref{UnEPTT*G})}\end{array}} \ar[ddl]|-{\text{L.R. by }G}
\ar[ddr]|-{\text{L.R. by }\mathfrak{g}_{1}^{\ast }} \ar[dd]|-{\text{EPR by
}G\circledS \mathfrak{g}_{1}^{\ast }} \\\\
{\begin{array}{c} \mathfrak{g}_{1}^{\ast }\circledS(\mathfrak{g}_{2}\circledS
\mathfrak{g}_{3}^{\ast })\\\text{EL in
(\ref{g*gg*-})}\end{array}} &
{\begin{array}{c}(\mathfrak{g}_{2}\circledS\mathfrak{g}_{3}^{\ast
})\\\text{EP in (\ref{EPgg*})}\end{array}} & {\begin{array}{c} G\circledS (
\mathfrak{g}_{2}\circledS\mathfrak{g}_{3}^{\ast}) \\ \text{EL in
(\ref{Ggg*-})}\end{array}} \\\\ & {\begin{array}{c}\mathfrak{g}_{2}
\\\text{EP in (\ref{EPEq})}\end{array}} \ar@{^{(}->}[uu]|-{\txt{canonical \\
immersion}} \ar@{^{(}->}[uul]|-{\txt{canonical \\ immersion}}
\ar@{^{(}->}[uur]|-{\txt{canonical \\ immersion}} & {\begin{array}{c}
G\circledS\mathfrak{g}_{2}\\\text{EL in (\ref{Gg-})} \end{array}}
\ar@{^{(}->}[uu]|-{\txt{canonical \\ immersion} }}  \label{TT*Gl} 
\end{align}
\begin{small}
\begin{center}
Lagrangian reductions on $TT^*G$
\end{center}
\end{small}
\begin{align}
& \boxed{ \begin{aligned} &\frac{d}{dt}\left( \frac{\delta E}{\delta \xi
}\right) = T_{e}^{\ast }R_{g}\left( \frac{\delta E}{\delta g}\right)
+\ad_{\xi }^{\ast }\left( \frac{\delta E}{\delta \xi }\right)
-\ad_{\frac{\delta E}{\delta \nu }}^{\ast }\nu \\& \frac{d}{dt}\left(
\frac{\delta E}{\delta \nu }\right) =-\ad_{\xi }\frac{\delta E}{\delta \nu }
\end{aligned}}  \label{Ggg*-} \\
& \boxed{ \begin{aligned} &\frac{d}{dt}\left( \frac{\delta E}{\delta \xi
}\right) =T_{e}^{\ast }R_{g}\left( \frac{\delta E}{\delta g}\right) +\ad_{\xi
}^{\ast }\left( \frac{\delta E}{\delta \xi }\right) \end{aligned}}
\label{Gg-} \\
& \boxed{ \begin{aligned} &\frac{d}{dt}\left( \frac{\delta E}{\delta \xi
}\right) =-\ad_{\frac{\delta E}{\delta \mu }}^{\ast }\mu +\ad_{\xi }^{\ast
}\left( \frac{\delta E}{\delta \xi }\right) -\ad_{\frac{\delta E}{\delta \nu
}}^{\ast }\nu \\& \frac{d}{dt}\left( \frac{\delta E}{\delta \nu }\right)
=\frac{\delta E}{\delta \mu }-\ad_{\xi }\frac{\delta E}{\delta \nu }
\end{aligned}}  \label{g*gg*-}
\end{align}

\section{Summary, Discussions and Prospectives}

We write Hamilton's equations on the cotangent bundles $ T^{\ast}TG$ and $%
 T^{\ast}T^{\ast}G$. Symplectic and Poisson reductions of $ T^{\ast}TG$
are performed under actions of $G,$ $\mathfrak{g}$ and $G\circledS \mathfrak{%
g}$ as shown in diagram (\ref{T*TG}). $ T^{\ast}T^{\ast}G$ is also
reduced by actions of $G,$ $\mathfrak{g}^{\ast}$ and $G\circledS \mathfrak{g}%
^{\ast}$ c.f. diagram (\ref{T*T*G}).

On the Tulczyjew's symplectic space $ TT^{\ast}G=\left( G\circledS
\mathfrak{g}_{1}^{\ast}\right) \circledS\left( \mathfrak{g}_{2}\circledS%
\mathfrak{g}_{3}^{\ast}\right) $, we obtain both Hamilton's and
Euler-Lagrange equations. Hamilton's equations are reduced by symplectic and
Poisson actions of $G,$ $\mathfrak{g}_{1}^{\ast}$, $\mathfrak{g}_{2}$, $%
G\circledS\mathfrak{g}_{2}$ and $G\circledS\mathfrak{g}_{1}^{\ast}$. These
reductions are summarized in diagram (\ref{TT*Gd}). As it is a tangent
bundle, Lagrangian reductions are performed with actions of $G$, $\mathfrak{g%
}_{1}^{\ast}$, $G\circledS\mathfrak{g}_{1}^{\ast}$ and $G\circledS\left(
\mathfrak{g}_{1}^{\ast}\times\mathfrak{g}_{2}\right) $ and, are shown in
diagram (\ref{TT*Gl}).

Hamiltonian reductions of the Tulczyjew's symplectic space $ TT^{\ast}G$
can be generalized to symplectic reduction of tangent bundle of a symplectic
manifold with lifted symplectic structure. This may be a first step towards
reduction of special symplectic structures and reduction of Tulczyjew's
triplet for arbitrary configuration manifold $\mathcal{Q}$.
In order to obtain this more general picture for  trivialization and reduction of Tulczyjew triplet, we plan to pursue a new project where the reduction is applied to Lagrangian dynamics on $TQ$ and Hamiltonian dynamics on $T^*Q$ for an arbitrary manifold $Q$. In this case, the reduced Lagrangian dynamics on the orbit space $TQ/G$ is called Lagrange-Poincar\'{e} equations. If, particularly, $Q=G$ then the Lagrange-Poincar\'{e} equations turns out to be Euler-Poincar\'{e} equations on $\mathfrak{g}$. Similarly, the Hamiltonian dynamics on $T^*Q/G$ is called Hamilton-Poincar\'{e} equations and, reduce to Lie-Poisson equations on $\mathfrak{g}^*$ for the case of $Q=G$. In the first paper \cite{EsKuSu20} of that series, we have already presented the trivialization and reduction of Tulczyjew triplet for an arbitrary manifold under the presence of an Ehresmann connection.

\bibliographystyle{plain}
\bibliography{references}{}

\def\polhk#1{\setbox0=\hbox{#1}{\ooalign{\hidewidth
  \lower1.5ex\hbox{`}\hidewidth\crcr\unhbox0}}} \def\cprime{$'$}
  \def\cprime{$'$} \def\cprime{$'$} \def\cprime{$'$} \def\cprime{$'$}
  \def\cprime{$'$} \def\cprime{$'$} \def\cprime{$'$} \def\cprime{$'$}
  \def\cprime{$'$} \def\cprime{$'$} \def\Dbar{\leavevmode\lower.6ex\hbox to
  0pt{\hskip-.23ex \accent"16\hss}D}
  \def\cfac#1{\ifmmode\setbox7\hbox{$\accent"5E#1$}\else
  \setbox7\hbox{\accent"5E#1}\penalty 10000\relax\fi\raise 1\ht7
  \hbox{\lower1.15ex\hbox to 1\wd7{\hss\accent"13\hss}}\penalty 10000
  \hskip-1\wd7\penalty 10000\box7}
  \def\cftil#1{\ifmmode\setbox7\hbox{$\accent"5E#1$}\else
  \setbox7\hbox{\accent"5E#1}\penalty 10000\relax\fi\raise 1\ht7
  \hbox{\lower1.15ex\hbox to 1\wd7{\hss\accent"7E\hss}}\penalty 10000
  \hskip-1\wd7\penalty 10000\box7} \def\cprime{$'$}
\begin{thebibliography}{10}

\bibitem{abraham1978foundations}
R.~Abraham and J.~E. Marsden.
\newblock {\em Foundations of mechanics}.
\newblock Benjamin/Cummings Publishing Co., Inc., Advanced Book Program,
  Reading, Mass., 1978.

\bibitem{abrunheiro2011cubic}
L~Abrunheiro, M~Camarinha, and J~Clemente-Gallardo.
\newblock Cubic polynomials on lie groups: reduction of the hamiltonian system.
\newblock {\em Journal of Physics A: Mathematical and Theoretical},
  44(35):355203, 2011.

\bibitem{AlekGrabMarmMich94}
D.~Alekseevsky, J.~Grabowski, G.~Marmo, and P.~W. Michor.
\newblock Poisson structures on the cotangent bundle of a {L}ie group or a
  principal bundle and their reductions.
\newblock {\em J. Math. Phys.}, 35(9):4909--4927, 1994.

\bibitem{BobeSuri99-II}
A.~I. Bobenko and Y.~B. Suris.
\newblock Discrete time {L}agrangian mechanics on {L}ie groups, with an
  application to the {L}agrange top.
\newblock {\em Comm. Math. Phys.}, 204(1):147--188, 1999.

\bibitem{BouMars09}
N.~Bou-Rabee and J.~E. Marsden.
\newblock {H}amilton-{P}ontryagin integrators on {L}ie groups. {I}.
  {I}ntroduction and structure-preserving properties.
\newblock {\em Found. Comput. Math.}, 9(2):197--219, 2009.

\bibitem{burnett2011geometric}
C.~L. Burnett, D.~D. Holm, and D.~M. Meier.
\newblock Geometric integrators for higher-order mechanics on lie groups.
\newblock {\em arXiv preprint arXiv:1112.6037}, 2011.

\bibitem{cendra1998maxwell}
H.~Cendra, D.~D. Holm, M.J.W. Hoyle, and J.E. Marsden.
\newblock The maxwell--vlasov equations in euler--poincar{\'e} form.
\newblock {\em Journal of Mathematical Physics}, 39(6):3138--3157, 1998.

\bibitem{cendra1998lagrangian}
H.~Cendra, D.~D. Holm, J.~E. Marsden, and T.~S. Ratiu.
\newblock Lagrangian reduction, the {E}uler-{P}oincar\'{e} equations, and
  semidirect products.
\newblock {\em Amer. Math. Soc. Transl. Ser. 2}, 186:1--25, 1998.

\bibitem{CendMarsPekaRati03}
H.~Cendra, J.~E. Marsden, S.~Pekarsky, and T.~S. Ratiu.
\newblock Variational principles for {L}ie-{P}oisson and
  {H}amilton-{P}oincar\'e equations.
\newblock {\em Mosc. Math. J.}, 3(3):833--867, 1197--1198, 2003.
\newblock \{Dedicated to Vladimir Igorevich Arnold on the occasion of his 65th
  birthday\}.

\bibitem{n2001lagrangian}
H.~Cendra, J.~E. Marsden, and T.~S. Ratiu.
\newblock Lagrangian reduction by stages.
\newblock {\em Mem. Amer. Math. Soc.}, 152(722):x+108, 2001.

\bibitem{colombo2014higher}
L.~Colombo and D.~de~Diego.
\newblock Higher-order variational problems on {L}ie groups and optimal control
  applications.
\newblock {\em Journal of Geometric Mechanics}, 6(4), 2014.

\bibitem{colombo2013optimal}
L.~Colombo and D.~M. de~Diego.
\newblock Optimal control of underactuated mechanical systems with symmetries.
\newblock {\em Discrete Contin. Dyn. Syst.}, (Dynamical systems, differential
  equations and applications. 9th AIMS Conference. Suppl.):149--158, 2013.

\bibitem{colombo2013higher}
L.~Colombo, D.~M. de~Diego, and M.~Zuccalli.
\newblock Higher-order discrete variational problems with constraints.
\newblock {\em J. Math. Phys.}, 54(9):093507, 17, 2013.

\bibitem{engo2003partitioned}
K~Eng\o.
\newblock Partitioned {R}unge-{K}utta methods in {L}ie-group setting.
\newblock {\em BIT}, 43(1):21--39, 2003.

\bibitem{esen2014tulczyjew}
O.~Esen and H.~G\"{u}mral.
\newblock Tulczyjew's triplet for {L}ie groups {I}: {T}rivializations and
  reductions.
\newblock {\em J. Lie Theory}, 24(4):1115--1160, 2014.

\bibitem{esen2015tulczyjew}
O.~Esen and H.~G\"{u}mral.
\newblock Tulczyjew's triplet for {L}ie groups {II}: {D}ynamics.
\newblock {\em J. Lie Theory}, 27(2):329--356, 2017.

\bibitem{EsKuSu20}
O.~Esen, M.~Kudeyt, and S.~S{\"u}tl{\"u}.
\newblock Tulczyjew's triplet with an ehresmann connection i: Trivialization
  and reduction.
\newblock {\em arXiv preprint arXiv:2007.11662}, 2020.

\bibitem{EsSu16}
O.~Esen and S.~S\"utl\"u.
\newblock Hamiltonian dynamics on matched pairs.
\newblock {\em Int. J. Geom. Methods Mod. Phys.}, 13(10):1650128, 24, 2016.

\bibitem{gay2012invariant}
F.~Gay-Balmaz, D.~D. Holm, D.~M. Meier, T.~S. Ratiu, and F.-X. Vialard.
\newblock Invariant higher-order variational problems.
\newblock {\em Communications in Mathematical Physics}, 309(2):413--458, 2012.

\bibitem{grabowska2016tulczyjew}
K.~Grabowska and M.~Zaj\c{a}c.
\newblock The {T}ulczyjew triple in mechanics on a {L}ie group.
\newblock {\em J. Geom. Mech.}, 8(4):413--435, 2016.

\bibitem{hindeleh2006tangent}
F.~Y. Hindeleh.
\newblock {\em Tangent and cotangent bundles, automorphism groups and
  representations of {L}ie groups}.
\newblock ProQuest LLC, Ann Arbor, MI, 2006.

\bibitem{holm2008geometric}
D.~D. Holm.
\newblock {\em Geometric mechanics. {P}art {I}}.
\newblock Imperial College Press, London, second edition, 2011.
\newblock Dynamics and symmetry.

\bibitem{holm1998euler}
D.~D. Holm, J.~E. Marsden, and T.~S. Ratiu.
\newblock The {E}uler-{P}oincar\'{e} equations and semidirect products with
  applications to continuum theories.
\newblock {\em Adv. Math.}, 137(1):1--81, 1998.

\bibitem{holm1986hamiltonian}
D.D. Holm, J.~E. Marsden, and T.~S. Ratiu.
\newblock The hamiltonian structure of continuum mechanics in material, inverse
  material, spatial and convective representations.
\newblock {\em Hamiltonian structure and Lyapunov stability for ideal continuum
  dynamics}, 100(BOOK\_CHAP):11--124, 1986.

\bibitem{Ho09}
D.D. Holm, T.~Schmah, and C.~Stoica.
\newblock {\em Geometric mechanics and symmetry}, volume~12 of {\em Oxford
  Texts in Applied and Engineering Mathematics}.
\newblock Oxford University Press, Oxford, 2009.
\newblock From finite to infinite dimensions, With solutions to selected
  exercises by David C. P. Ellis.

\bibitem{KolaMichSlov-book}
I.~Kol{\'a}{\v{r}}, P.~W. Michor, and J.~Slov{\'a}k.
\newblock {\em Natural operations in differential geometry}.
\newblock Springer-Verlag, Berlin, 1993.

\bibitem{kupershmidt1983canonical}
B.~A. Kupershmidt and T.~S. Ratiu.
\newblock Canonical maps between semidirect products with applications to
  elasticity and superfluids.
\newblock {\em Communications in Mathematical Physics}, 90(2):235--250, 1983.

\bibitem{lichnerowicz1988lie}
Andr{\'e} Lichnerowicz and Alberto Medina.
\newblock On lie groups with left-invariant symplectic or k{\"a}hlerian
  structures.
\newblock {\em Letters in Mathematical Physics}, 16(3):225--235, 1988.

\bibitem{manga2015geometry}
Bakary Manga.
\newblock On the geometry of cotangent bundles of lie groups.
\newblock {\em arXiv preprint arXiv:1505.00375}, 2015.

\bibitem{MaMiOrPeRa07}
J.~E. Marsden, G.~Misio{\l}ek, J.-P. Ortega, M.~Perlmutter, and T.~S. Ratiu.
\newblock {\em {H}amiltonian reduction by stages}, volume 1913 of {\em Lecture
  Notes in Mathematics}.
\newblock Springer, Berlin, 2007.

\bibitem{marsden1998symplectic}
J.~E. Marsden, G.~Misio{\l}ek, M.~Perlmutter, and T.~S. Ratiu.
\newblock Symplectic reduction for semidirect products and central extensions.
\newblock {\em Differential Geometry and its Applications}, 9(1-2):173--212,
  1998.

\bibitem{marsden1984semidirect}
J.~E. Marsden, T.~S. Ra\c{t}iu, and A.~Weinstein.
\newblock Semidirect products and reduction in mechanics.
\newblock {\em Trans. Amer. Math. Soc.}, 281(1):147--177, 1984.

\bibitem{marsden1991symplectic}
J.~E. Marsden, T.~Ratiu, and G.~Raugel.
\newblock Symplectic connections and the linearisation of {H}amiltonian
  systems.
\newblock {\em Proc. Roy. Soc. Edinburgh Sect. A}, 117(3-4):329--380, 1991.

\bibitem{MarsdenRatiu-book}
J.~E. Marsden and T.~S. Ratiu.
\newblock {\em Introduction to mechanics and symmetry}, volume~17 of {\em Texts
  in Applied Mathematics}.
\newblock Springer-Verlag, New York, second edition, 1999.

\bibitem{marsden2000reduction}
J.~E. Marsden, T.~S. Ratiu, and J.~Scheurle.
\newblock Reduction theory and the {L}agrange--{R}outh equations.
\newblock {\em Journal of mathematical physics}, 41(6):3379--3429, 2000.

\bibitem{marsden1983hamiltonian2}
J.~E. Marsden, T.~S. Ratiu, R.~Schmid, R.~G. Spencer, and A.~Weinstein.
\newblock Hamiltonian systems with symmetry, coadjoint orbits and plasma
  physics.
\newblock {\em Atti della Accademia delle scienze di Torino}, 117(1):289--340,
  1983.

\bibitem{marsden1984reduction--}
J.~E. Marsden, T.~S. Ratiu, and A.~Weinstein.
\newblock Reduction and hamiltonian structures on duals of semidirect product
  lie algebras.
\newblock {\em Cont. Math. AMS}, 28:55--100, 1984.

\bibitem{marsden1974reduction}
J.~E. Marsden and A.~Weinstein.
\newblock Reduction of symplectic manifolds with symmetry.
\newblock {\em Reports on mathematical physics}, 5(1):121--130, 1974.

\bibitem{marsden1983coadjoint}
J.~E. Marsden and A.~Weinstein.
\newblock Coadjoint orbits, vortices, and clebsch variables for incompressible
  fluids.
\newblock {\em Physica D: Nonlinear Phenomena}, 7(1-3):305--323, 1983.

\bibitem{michor2008topics}
P.~W. Michor.
\newblock {\em Topics in differential geometry}, volume~93.
\newblock American Mathematical Soc., 2008.

\bibitem{Rati80}
T.~S. Ratiu.
\newblock The motion of the free {$n$}-dimensional rigid body.
\newblock {\em Indiana Univ. Math. J.}, 29(4):609--629, 1980.

\bibitem{ratiu1982euler}
T.~S. Ratiu.
\newblock Euler-poisson equations on lie algebras and the n-dimensional heavy
  rigid body.
\newblock {\em American journal of mathematics}, pages 409--448, 1982.

\bibitem{ratiu1982lagrange}
T.~S. Ratiu and P.~van Moerbeke.
\newblock The {L}agrange rigid body motion.
\newblock volume~32, pages viii, 211--234, 1982.

\bibitem{Tu77}
W.~M. Tulczyjew.
\newblock The {L}egendre transformation.
\newblock In {\em Annales de l'IHP Physique th{\'e}orique}, volume~27, pages
  101--114, 1977.

\bibitem{We83}
A.~Weinstein.
\newblock The local structure of {P}oisson manifolds.
\newblock {\em J. Differential Geom.}, 18(3):523--557, 1983.

\end{thebibliography}

\end{document}